\documentclass[cjm]{ipart-ppn}
\RequirePackage{hyperref}

\usepackage{amsfonts}
\usepackage{amsmath}
\usepackage{amssymb}
\usepackage{amsthm}
\usepackage{mathrsfs}
\usepackage{bbm}

\RequirePackage{silence}
\newtheorem{theorem}{Theorem}[section]
\newtheorem{lemma}[theorem]{Lemma}
\newtheorem{corollary}[theorem]{Corollary}
\newtheorem{proposition}[theorem]{Proposition}
\newtheorem{remark}[theorem]{Remark}

\startlocaldefs
\newcommand{\1}{\mathbbm1}
\newcommand{\R}{{\mathord{\mathbb R}}}

\newcommand{\N}{{\mathord{\mathbb N}}}

\newcommand{\Sp}{{\mathord{\mathbb S}}}
\newcommand{\Sph}{\mathbb{S}}

\newcommand{\bb}{g}
\renewcommand{\(}{\left(}
\renewcommand{\)}{\right)}
\newcommand{\ird}[1]{\int_{\R^d}{#1}\,dx}
\newcommand{\isd}[1]{\int_{\Sph^d}{#1}\,d\mu}
\newcommand{\nrm}[2]{\left\|{#1}\right\|_{\mathrm L^{#2}(\R^d)}}
\newcommand{\nrmS}[2]{\left\|{#1}\right\|_{\mathrm L^{#2}(\Sp^d)}}
\newcommand{\Deficit}{\boldsymbol{\delta}_{\rm Sob}}
\newcommand{\deficit}{\mathop{\mbox{\sc d}}}
\newcommand{\Imu}{\mathscr I}
\newcommand{\be}[1]{\begin{equation}\label{#1}}
\newcommand{\ee}{\end{equation}}
\newcommand{\irdmu}[1]{\int_{\R^d}{#1}\,d\mu_d}
\newcommand{\irNg}[1]{\int_{\R^N}{#1}\,d\gamma}
\endlocaldefs

\pubyear{2025}
\volume{0}
\issue{0}
\firstpage{1}
\lastpage{1}
\begin{document}

%%%%%%%%%%%%%%%%%%%%%%%%%%%%%%%%%%%%%%%%%%%%%%%%%%
%%%%%%%%%%%%%%%%%%%%%%%%%%%%%%%%%%%%%%%%%%%%%%%%%%
\begin{frontmatter}

\title[Sharp stability for Sobolev and log-Sobolev inequalities]{Sharp stability for Sobolev and log-Sobolev inequalities, with optimal dimensional dependence}

\begin{aug}
\author{\fnms{Jean} \snm{Dolbeault\thanks{Partial support through the French National Research Agency (ANR) project EFI (ANR-17-CE40-0030, J.D.).}}\ead[label=eDo]{dolbeaul@ceremade.dauphine.fr}},
\address{CEREMADE (CNRS UMR No.~7534), PSL University, Universit\'e Paris-Dauphine, Place de Lattre de Tassigny, 75775 Paris 16\\
France\\
\printead{eDo}}
\author{\fnms{Maria J.} \snm{Esteban\thanks{Partial support through the French National Research Agency (ANR) project molQED (ANR-17-CE29-0004, M.J.E.).}}\ead[label=eEs]{esteban@ceremade.dauphine.fr}},
\address{CEREMADE (CNRS UMR No.~7534), PSL University, Universit\'e Paris-Dauphine, Place de Lattre de Tassigny, 75775 Paris 16\\
France\\
\printead{eEs}}
\author{\fnms{Alessio} \snm{Figalli\thanks{Partial support through the European Research Council under the Grant Agreement No. 721675 (RSPDE) {\it Regularity and Stability in Partial Differential Equations.}}}\ead[label=eFi]{alessio.figalli@math.ethz.ch}},
\address{Mathematics Department, ETH Z\"urich, Ramistrasse 101, 8092 Z\"urich\\
Switzerland\\
\printead{eFi}}
\author{\fnms{Rupert L.} \snm{Frank\thanks{Partial support through the US National Science Foundation grant DMS-1954995 and the Deutsche Forschungsgemeinschaft (DFG, German Research Foundation) Germany's Excellence Strategy EXC-2111-390814868.}}\ead[label=eFr]{r.frank@lmu.de}},
\address{Department of Mathematics, LMU Munich, Theresienstr. 39, 80333 M\"unchen, Germany, and Munich Center for Quantum Science and Technology, Schellingstr.~4, 80799 M\"unchen\\
Germany\\
and\\
Mathematics 253-37, Caltech, Pasadena, CA 91125\\
United States of America\\
\printead{eFr}}
\and
\author{\fnms{Michael} \snm{Loss\thanks{Partial support through US National Science Foundation grant DMS-2154340.}}\ead[label=eLo]{loss@math.gatech.edu}},
\address{School of Mathematics, Georgia Institute of Technology Atlanta, GA 30332\\
United States of America\\
\printead{eLo}}
\end{aug}

%% History:
%\received{\sday{3} \smonth{1} \syear{2022}}

\begin{abstract}
We prove a sharp quantitative version for the stability of the Sobolev inequality with explicit constants. Moreover, the constants have the correct behavior in the limit of large dimensions, which allows us to deduce an optimal quantitative stability estimate for the Gaussian log-Sobolev inequality with an explicit dimension-free constant. Our proofs rely on several ingredients such as competing symmetries, a flow based on continuous Steiner symmetrization that interpolates continuously between a function and its symmetric decreasing rearrangement, and refined estimates on the Sobolev functional in the neighborhood of the optimal Aubin--Talenti functions.
\end{abstract}



\vspace*{2cm}
\tableofcontents
\end{frontmatter}
\newpage

%%%%%%%%%%%%%%%%%%%%%%%%%%%%%%%%%%%%%%%%%%%%%%%%%%
%%%%%%%%%%%%%%%%%%%%%%%%%%%%%%%%%%%%%%%%%%%%%%%%%%
\section{Introduction and main results}\label{sec:intro}

The classical Sobolev inequality on $\R^d$, $d\ge3$, states that
$$
\Vert \nabla f \Vert^2_{\mathrm L^2(\R^d)} \geq S_d\,\Vert f \Vert^2_{\mathrm L^{2^*}(\R^d)} \quad\forall\,f\in\dot{\mathrm H}^1(\R^d)\,,
$$
where $2^*=\frac{2\,d}{d-2}$ is the Sobolev exponent, $S_d = \tfrac14\,d\,(d-2)\,|\Sp^d|^{2/d}$ is the sharp Sobolev constant, and $|\Sp^d|$ denotes the $d$-dimensional volume of the unit sphere in $\mathbb S^d\subset\R^{d+1}$.
Here $\dot{\mathrm H}^1(\R^d)$ is the closure of $C^\infty_c(\R^d)$ with respect to the seminorm $\Vert f\Vert_{\dot{\mathrm H}^1(\R^d)}:=\Vert \nabla f \Vert^2_{\mathrm L^2(\R^d)}$.
In addition, equality holds if and only if $f$ belongs to the $(d+2)$-dimensional manifold
\begin{align}
\mathcal M:=&\left\{g_{a,b,c}\,:\,(a,b,c)\in(0,+\infty)\times\R^d\times\R\right\}\nonumber\\
&\quad\mbox{where}\quad g_{a,b,c}(x) =c\,\bar g\Big(\frac{x-b}a\Big)\quad\mbox{and}\quad \bar g(x)=\Big(\frac2{1 + |x|^2}\Big)^\frac{d-2}2\,.\label{eq:optimizers}
\end{align}
In~\cite{BrezisLieb} Brezis and Lieb asked the following question:
\begin{center}
\parbox{12.6cm}{\it Do there exist constants $\kappa$, $\alpha>0$ such that\\[4pt]
\centerline{$\displaystyle
\Deficit(f):=\frac{\Vert \nabla f \Vert^2_{\mathrm L^2(\R^d)}}{\Vert f \Vert^2_{\mathrm L^{2^*}(\R^d)}} - S_d
\geq \kappa\,{\rm dist}(f,\mathcal M)^\alpha
$}\\[4pt]
where ${\rm dist}(\cdot ,\mathcal M)$ denotes some `natural distance' from the set of optimizers?}
\end{center}
In the modern terminology, $\Deficit(f)$ is usually called the {\it Sobolev deficit}.
In this kind of stability questions, one can try to obtain `the best possible result' by finding the strongest possible topology to define the distance and the best possible constant $\kappa$ and exponent $\alpha$.
A beautiful answer to Brezis and Lieb's question has been given
by Bianchi and Egnell in~\cite{BianchiEgnell}: for any $d\geq 3$ there is a dimensional constant $\mathbf{\mathscr C}_{d,\rm BE}>0$ such that
\begin{equation}\label{eq:bianchi-egnell0}
\Deficit(f)\ge\mathbf{\mathscr C}_{d,\rm BE}\,\inf_{\bb\in\mathcal M}\Vert \nabla f - \nabla\bb \Vert^2_{\mathrm L^2(\R^d)}
\end{equation}
for any $f \in \dot{\mathrm H}^1(\R^d)$ such that $\Vert f \Vert_{\mathrm L^{2^*}(\R^d)}=1$. It is worth observing that this result is optimal both in terms of the distance used (the $\dot{\mathrm H}^1$ norm) and in terms of the exponent $2$.
Its proof is based on two principles:
\begin{itemize}
\item[(i)] {\it Local-to-global:\/} it suffices to prove the inequality in a neighborhood of $\mathcal M$;
\item[(ii)] {\it Local analysis:\/}~\eqref{eq:bianchi-egnell0} holds near $\mathcal M$.
\end{itemize}
As shown in~\cite{BianchiEgnell}, these two steps are achieved as follows:
\begin{itemize}
\item[(i)] By Lions's concentration-compactness theorem, if $\Deficit(f)$ is small, then $f$ is close in $\dot{\mathrm H}^1$ to~$\mathcal M$.
\item[(ii)] Given $f$ close to $\mathcal M$, one can assume that $\bar g\in \mathcal M$ is the closest point to $f$. Then, if one writes $f=\bar g+\epsilon\, \varphi$
with $\epsilon:=\Vert \nabla f -\nabla \bar g\Vert_{\mathrm L^2(\R^d)}$ and $\Vert \nabla\varphi \Vert_{\mathrm L^2(\R^d)}=1$, a Taylor expansion gives
$$
\Deficit(\bar g+\epsilon\, \varphi)\ge\epsilon^2\,Q_{\bar g}[\varphi]-\frac2{2^*}\,\epsilon^{2^*} \,,
$$
where $Q_{\bar g}[\,\cdot\,]$ is a quadratic form depending on $\bar g$ (see Section~\ref{sec:firstbound} below for more details). In addition, spectral analysis shows that $Q_{\bar g}[\varphi]\geq \frac{4}{d+4}$ and this inequality is sharp, proving that
\begin{equation}
\label{eq:Taylor}
\Deficit(\bar g+\epsilon\, \varphi)\geq \frac{4}{d+4}\,\epsilon^2-\frac2{2^*}\,\epsilon^{2^*} \,.
\end{equation}
In particular, if $\epsilon$ is sufficiently small then~\eqref{eq:bianchi-egnell0} follows.
\end{itemize}

Although Bianchi and Egnell's result gives a very satisfactory answer to the question raised by Brezis and Lieb, their method gives no information about the constant $\mathbf{\mathscr C}_{d,\rm BE}$. More precisely:
\begin{itemize}
\item[(i)] Since the local-to-global argument is based on compactness, there is no information about the size of $\mathbf{\mathscr C}_{d,\rm BE}$ outside a small $\dot{\mathrm H}^1$-neighborhood of $\mathcal M$.
\item[(ii)] Even if we restrict to functions close to $\mathcal M$, the bound provided by Bianchi and Egnell is very unsatisfactory for large dimensions. Indeed,~\eqref{eq:Taylor} implies that
$\Deficit(g+\epsilon\, \varphi)\gtrsim \frac1{d}\,\epsilon^2$ provided $\epsilon^{2^*-2}\lesssim \frac1{d}$, or equivalently $\epsilon \lesssim d^{-d/4}$.
In other words, for large dimensions, the neighborhood of $\mathcal M$ where the Taylor expansion of Bianchi and Egnell provides a lower bound is super-exponentially small with respect to $d$.\end{itemize}

\medskip The goal of this paper is to provide a new proof of the Bianchi-Egnell estimate that leads to a completely sharp result.
More precisely, by a series of new ideas and techniques, we shall provide:
\begin{itemize}
\item[(i)] a quantitative local-to-global principle, based on competing symmetries and continuous Stei\-ner symmetrization, that allows us to reduce the global estimate to a local estimate;
\item[(ii)] a refined local analysis that provides a bound on the form $\Deficit(g+\epsilon\, \varphi)\geq \frac{c_0}{d}\,\epsilon^2$ for $\epsilon \leq \epsilon_0$, where $c_0$ and $\epsilon_0$ are {\it independent\/} of the dimension.
\end{itemize}
These techniques allow us to prove the following explicit stability constant estimate.
%-------------------------------------------------
\begin{theorem}\label{main} There is an explicit constant $\beta>0$ such that, for all $d\ge3$ and all $f\in\dot{\mathrm H}^1(\R^d)$,
\[
\Vert \nabla f \Vert^2_{\mathrm L^2(\R^d)} - S_d\,\Vert f \Vert^2_{\mathrm L^{2^*}(\R^d)}\ge\frac{\beta}{d}\,{\inf_{\bb\in\mathcal M} \Vert \nabla f - \nabla\bb \Vert^2_{\mathrm L^2(\R^d)}}\,.
\]
\end{theorem}
%-------------------------------------------------
To our knowledge, this is the first estimate where one obtains a complete dimensionally sharp result for the deficit of a Sobolev inequality. If $\mathbf{\mathscr C}_{d,\rm BE}$ denotes the sharp constant in~\eqref{eq:bianchi-egnell0}, which we shall assume from now on, then Theorem~\ref{main} can be succinctly written
\[
\mathbf{\mathscr C}_{d,\rm BE} \geq \frac\beta d\,,
\]
where $\beta$ is independent of $d$. To emphasize the robustness of our result we can prove, as a direct consequence of Theorem~\ref{main} when $d\to\infty$, a new stability result for the Gaussian log-Sobolev inequality. More precisely, on $\R^N$ with $N\geq 1$, we consider the Gaussian~measure
$$
d\gamma(x) = e^{-\,\pi\,|x|^2}\,dx\,.
$$
We abbreviate $\mathrm L^2(\gamma)=\mathrm L^2(\R^N,d\gamma)$ and denote by $\mathrm H^1(\gamma)$ the space of all $u\in \mathrm L^2(\gamma)$ with distributional gradient in $\mathrm L^2(\gamma)$.
%-------------------------------------------------
\begin{corollary}\label{logsob} With $\beta>0$ as in Theorem~\ref{main}, we have that, for all $N\in\N$ and all $u\in\mathrm H^1(\gamma)$,
\[
\int_{\R^N} |\nabla u|^2\,d\gamma - \pi \int_{\R^N} u^2 \ln\( \frac{u^2}{\| u \|_{\mathrm L^2(\gamma)}^2}\)\,d\gamma \geq \frac{\beta\,\pi}2 \inf_{b\in\R^N\!,\,c\in\R} \int_{\R^N} \big(u - c\,e^{b\cdot x}\big)^2\,d\gamma\,.
\]
\end{corollary}
%-------------------------------------------------
As we shall discuss later, also this corollary is optimal, in terms of the power that we control.
\medskip

%%%%%%%%%%%%%%%%%%%%%%%%%%%%%%%%%%%%%%%%%%%%%%%%%%
\subsection*{Some references}\label{sec:history}

The question of optimality in the Sobolev inequality has a long history. Rodemich~\cite{Rodemich}, Aubin~\cite{Aubin} and Talenti~\cite{Talenti} (see also~\cite{Rosen}) proved that the Sobolev deficit is nonnegative.
Moreover, it was shown by Lieb~\cite{Lieb}, Gidas, Ni and Nirenberg~\cite{GidasNiNirenberg} and Caffarelli, Gidas and Spruck~\cite{CaffarelliGidasSpruck} that the deficit vanishes if and only if the function $f$ is in the $(d+2)$-dimensional manifold~$\mathcal M$ of the `Aubin--Talenti functions' of the form~\eqref{eq:optimizers}. Lions~\cite{MR834360} has shown that if the Sobolev deficit is small for some function $f$, then $f$ has to be close to the set $\mathcal M$ of Sobolev optimizers, as a consequence of the concentration-compactness method (see also~\cite{MR2431434} for a textbook presentation). In that case, the closeness is measured in the strongest possible sense, namely with respect to the norm in $\dot{\mathrm H}^1(\R^d)$. The Bianchi--Egnell inequality~\eqref{eq:bianchi-egnell0} makes the qualitative result of Lions quantitative. In particular, it shows that the distance to the manifold vanishes at least like the square root of the Sobolev deficit. Such `stability' estimates have been established in other contexts as well, {\it e.g.}, for the isoperimetric inequality or for classical inequalities in real and harmonic analysis. In fact, stability has attracted a lot of attention in recent years. We refer to~\cite{FigalliMaggiPratelli,Figallietal} and references therein for quantitative stability results for isoperimetric inequalities based on a generalization of the Fraenkel asymmetry and to~\cite{FuscoMaggiPratelli,CianchiFuscoMaggiPratelli} for related results for functional inequalities of Sobolev type; to~\cite{CicaleseLeonardi} for a selection principle which provides an alternative proof of these results; to~\cite{ChenFrankWeth} for stability in the fractional Sobolev inequality; to~\cite{FigalliZhang} for a sharp quantitative stability result corresponding to the embedding of $\dot{\mathrm W}^{1,p}$ into the critical $\mathrm L^q(\R^d)$ space (and references therein for earlier papers); to~\cite{MR3227280} for stability results based on the duality between Sobolev and Hardy–Littlewood–Sobolev inequalities and to~\cite{CarlenFrankLieb} for the question of the stability of the lowest eigenvalue of the Schr\"odinger operator under the constraint of a given $\mathrm L^p$ norm of the potential based on Gagliardo--Nirenberg inequalities, by a Keller duality; to~\cite{Christ_HY,Christ} for respectively sharpened Hausdorff--Young and Riesz rearrangement inequalities, with applications in~\cite{FrankLieb2,FrankLieb}; to~\cite{MR3024094} and~\cite{MR3695890,MR3989143} for an extension of the Bianchi-Egnell method to a Sobolev inequality for continuous dimensions (using weights) motivated by Gagliardo--Nirenberg inequalities, whose stability (under tail restriction) is also obtained by entropy estimates and regularizing properties of fast diffusion flows in~\cite{BonforteDolbeaultNazaretSimonov}, with constructive estimates; to~\cite{Frank_2022,arxiv.2211.13180} for related results on the unit sphere for subcritical interpolation inequalities. In several of these papers the strategy of Bianchi and Egnell or its generalizations play an important role.

An interesting point about~\eqref{eq:bianchi-egnell0} and other inequalities obtained by this method is that nothing seems to be known about the value of the constant $\mathbf{\mathscr C}_{d,\rm BE}$ except for the fact that it is strictly positive and bounded from above by
\begin{equation}
\label{eq:upperbound}
\mathbf{\mathscr C}_{d,\rm BE} \leq \frac{4}{d+4}\,,
\end{equation}
as a consequence of the sharpness of the leading order term in~\eqref{eq:Taylor} (see also the proof of~\cite[Lemma 1]{BianchiEgnell} or~\cite[Introduction]{ChenFrankWeth}). As mentioned before, the proof of~\eqref{eq:bianchi-egnell0} in~\cite{BianchiEgnell} proceeds by a spectral estimate combined with a compactness argument and hence cannot give any information about $\mathbf{\mathscr C}_{d,\rm BE}$. In~\cite{arXiv:2210.08482} K\"onig shows that the upper bound in~\eqref{eq:upperbound} is strict and in~\cite{arXiv:2211.14185} that the infimum defining $\mathbf{\mathscr C}_{d,\rm BE}$ is attained\footnote{In fact, the results of K\"onig in~\cite{arXiv:2210.08482,arXiv:2211.14185} provide affirmative answers to questions that we had asked in a first version of this paper.}. This is reminiscent of the planar isoperimetric inequality, where the constant in the quantitative isoperimetric inequality with Frankel asymmetry is strictly smaller than the constant in the corresponding spectral gap inequality and where one can prove the existence of an optimizing domain; see~\cite{BianchiniCroceHenrot}. For further studies under an additional convexity assumption, see~\cite{Campi,AlvinoFeroneNitsch,CicaleseLeonardi2}. Explicit lower estimates are known only for distances to $\mathcal M$ measured by weaker norms than in~\eqref{eq:bianchi-egnell0} and for functions satisfying additional constraints, while much more is known for subcritical interpolation inequalities than for Sobolev-type inequalities. We refer to~\cite{BDGV,MR3103175,BonforteDolbeaultNazaretSimonov} for Euclidean Gagliardo--Nirenberg--Sobolev inequalities, to~\cite{MR2375056} for improvements for Gaussian weights, and to~\cite{MR3640894,Frank_2022,chen2022sharp,arxiv.2211.13180} for interpolation inequalities on the sphere. After the completion of this paper, an extension of our method to fractional Sobolev inequalities has been obtained in~\cite{MR4758206} with interesting consequences for Hardy--Littlewood--Sobolev inequalities.

The {\it logarithmic Sobolev inequality\/} on a finite dimensional Euclidean space (with either Gaussian or Lebesgue measures) can be seen as a {\it large dimensional limit\/} of the Sobolev inequality, for instance by considering Sobolev's inequality on a sphere of radius $\sqrt d$ applied to a function depending only on $N$ real variables as in~\cite[p.~4818]{MR1164616} and~\cite{MR353471}. Also see~\cite[Remark~4, p.~254]{Vershik:2007nr} for some historical comments. The classical versions of the logarithmic Sobolev inequality are usually attributed to Stam~\cite{MR0109101}, Federbush~\cite{Federbush}, Gross~\cite{Gross75}, and also Weissler~\cite{MR479373} for a scale-invariant form. There is a huge literature on logarithmic Sobolev inequalities and we refer to~\cite{MR2325763} for a survey on many early results. Equality cases in the logarithmic Sobolev inequality have been characterized by Carlen in~\cite[Theorem~5]{MR1132315}, even with a remainder term, see~\cite[Theorem~6]{MR1132315}. Other remainder terms are given in~\cite{Bobkov_2014,MR3567822,MR3493423,brigati2023stability,indrei2023sharp} and, using weaker notions of distances, in~\cite{Bobkov_2014,MR3271181,MR3567822,MR3666798,MR4305006}, while some obstructions to stability results involving strong notions of distance are given in~\cite{MR4455233,MR4116725}. Under some restrictions, improved forms of the inequality are known for instance from~\cite{MR3567822,MR3493423,MR4116725,arxiv.2211.13180}. Also see~\cite{MR4475270} for a connection between deficit estimates for the logarithmic Sobolev inequality and the Mahler conjecture in convex geometry, and~\cite{MR4305006} for a detailed list of earlier related results. However, as far as we know, the Bianchi--Egnell strategy has so far not been applied to the logarithmic Sobolev inequality, probably because $u\mapsto u^2\, \ln(u^2)$ is not twice differentiable at the origin. Here we overcome this issue as a consequence of the optimal $d^{-1}$ decay of $\mathbf{\mathscr C}_{d,\rm BE}$.\medskip

%%%%%%%%%%%%%%%%%%%%%%%%%%%%%%%%%%%%%%%%%%%%%%%%%%
\subsection*{Comments}\label{sec:remarks}

Stability of functional inequalities is a natural question in the Calculus of Variations. As soon as the set~$\mathcal M$ of all minimizers of a given functional is known, the next question is: if for some function $f$ the functional takes a value above the minimum, can we control a distance ${\rm dist}(f,\mathcal M)$ of $f$ to~$\mathcal M$ in terms of the value of the functional? This is precisely what we do with the deficit of the Sobolev inequality $\Deficit$. Applications range from the justification of the use of Taylor expansions and spectral estimates, which is essential in many areas of physics, to the computation of \emph{a posteriori} errors in numerical analysis. Stability in the Sobolev inequality is of particular interest because a whole range of stability estimates in subcritical inequalities can be deduced by interpolation. This stability also applies to inequalities based on duality, like the Hardy--Littlewood--Sobolev inequalities, with applications to mean-field models and nonlinear equations involving fractional operators. The knowledge of an explicit stability constant is also an invitation to revisit various problems of analysis, like blow-up phenomena in which Aubin--Talenti functions play a key role, or rates of convergence in nonlinear parabolic equations, for instance fast diffusion equations, in which Barenblatt profiles are nothing more than Aubin--Talenti functions in a different setting.

An important point when discussing stability of functional inequalities is the notion of distance that is employed. In the setting of the Sobolev inequality, the distance ${\rm dist}(f,g)=\|\nabla f-\nabla g\|_2$ used by Bianchi and Egnell is the strongest possible notion of distance. The situation is less clear in the setting of the log-Sobolev inequality. It is well known that the entropy, that is, $\int_{\R^N}f\ln f\,d\gamma$, controls $\|f-1\|_{\mathrm L^1(\gamma)}$, if $\int_{\R^N}f\,d\gamma=1$, by the standard Pinsker--Csisz\'ar--Kullback inequality, and the Wasserstein distance by the Talagrand inequality (see for instance~\cite[Notes and references of Chapters 5 and 9]{MR3155209}). Here $f$ plays the role of $u^2$ in our Corollary~\ref{logsob}. Stability of the log-Sobolev inequality in Wasserstein distance is by now classical, with results that go back to~\cite{MR3271181,Bobkov_2014}, but is weaker than in $\mathrm L^2(\R^d,d\gamma)$ distance as in Corollary~\ref{logsob}. One may wonder whether one can prove stability with respect to the $\mathrm H^1(\R^N,d\gamma)$-distance. However, the corresponding bound does not hold according to~\cite{indrei2023sharp}; see~\cite{brigati2023stability,Indrei2024} for recent advances on this issue. In addition, within the $\mathrm L^p(\R^N,d\gamma)$ spaces, $p=2$ is the largest natural exponent for which such a stability estimate can hold; see~\cite{MR4455233} for a result in this direction.

%%%%%%%%%%%%%%%%%%%%%%%%%%%%%%%%%%%%%%%%%%%%%%%%%%
\subsection*{Strategy of the proofs and outline}\label{sec:strategy}

Let us start with the proof of Theorem~\ref{main}, concerning the stability of the Sobolev inequality. It consists of three main parts. The first and second parts deal with nonnegative functions, while in the third part we deduce the inequality for arbitrary functions from that for nonnegative functions. The latter argument uses a concavity property of the problem. Potentially this argument comes with a loss in the constant, but we show that it does not destroy the $d^{-1}$ behavior that we need to prove Corollary~\ref{logsob}.

We now discuss the first and the second parts in more detail. These two parts correspond to the two ingredients mentioned at the beginning of the introduction, namely to the {\it local analysis\/} (ii) and the {\it local-to-global\/} principle (i), respectively. The region where the local analysis applies is where the quantity $\inf_{\bb\in\mathcal M} \Vert\nabla f-\nabla\bb\Vert^2_{\mathrm L^2(\R^d)}/\Vert\nabla f\Vert^2_{\mathrm L^2(\R^d)}\leq\delta$, while the remaining region will be treated using the local-to-global principle. Here $\delta\in(0,1/2)$ is a free parameter that will be chosen appropriately at the end. The crucial point is that $\delta$ can be chosen independently of the dimension $d$.

The first part of the proof (see Theorem~\ref{unifboundclose} in Section~\ref{sec:close}) is concerned with a nonnegative function~$f$ that is close to the set of optimizers. The basic strategy is to expand the quantity $\|f\|_{\mathrm L^q(\R^d)}^2$, with the main term given by this quantity when $f$ is replaced by the closest optimizer~$g$. By this choice there will be no linear term in the expansion, and for the quadratic term one uses a spectral gap inequality (Section~\ref{sec:spectralGap}). A first version of this argument appears in the proof of Proposition~\ref{firstbound} in Section~\ref{sec:firstbound}. Such a naive expansion, however, is not good enough to reproduce the correct $d^{-1}$ behavior of the constant $\mathbf{\mathscr C}_{d,\rm BE}$. Instead, a refined argument (Sections~\ref{Sec:Pieces} and~\ref{Sec:Detailed}) is needed where we cut the function $f/g$ in various parts of its range and treat the different parts by \it ad hoc \rm arguments. Three different ranges of the function are treated and, while each of these arguments individually is not sufficient, by carefully combining them we obtain the final result. We mention that the spectral gap inequality is only used for an $\mathrm L^\infty$-bounded part of the perturbation.

Parenthetically we point out that we actually prove something stronger. Namely, we assume a decomposition $f=\bb+r$ with $\bb\in\mathcal M$ and a perturbation $r$ satisfying certain orthogonality conditions. These orthogonality conditions for $r$ are guaranteed when $g$ realizes the infimum $\inf_{g'\in\mathcal M} \|\nabla f - \nabla g'\|_{\mathrm L^2(\R^d)}^2$, but our argument does not make use of this minimality of $g$.

In the second part of the proof of Theorem~\ref{main}, described in Section~\ref{sec:flow}, we obtain a lower bound on
\be{Def:E}
\mathcal E(f) :=\frac{ \Vert \nabla f \Vert^2_{\mathrm L^2(\R^d)} - S_d\,\Vert f \Vert^2_{\mathrm L^{2^*}(\R^d)}} {\inf_{\bb\in\mathcal M} \Vert \nabla f - \nabla\bb \Vert^2_{\mathrm L^2(\R^d)}}\quad\forall\,f\in\dot{\mathrm H}^1(\R^d)\setminus\mathcal M
\ee
for nonnegative functions $f$ satisfying $\inf_{\bb\in\mathcal M}\! \Vert \nabla f - \nabla\bb \Vert^2_{\mathrm L^2(\R^d)}>\delta\,\Vert \nabla f \Vert_{\mathrm L^2(\R^d)}^2$; see Theorem~\ref{Cor:Summary} for a detailed statement. Bianchi and Egnell~\cite{BianchiEgnell} handle this part by a compactness argument and this is the reason why up to now there did not exist an explicit lower bound on $\mathbf{\mathscr C}_{d,\rm BE}$. Our solution is to replace this argument by a constructive procedure using an idea taken from a paper by Christ~\cite{Christ}, in which he establishes a quantitative error term for the Riesz rearrangement inequality. To implement
this strategy in our context we construct, using competing symmetries~\cite{CarlenLoss} and continuous rearrangement~\cite{Brock}, a family of functions $f_\tau, 0\le \tau < \infty$, such that $f_0=f$, $\Vert f_\tau \Vert_{2^*} = \Vert f \Vert_{2^*}$, $\tau \mapsto \Vert \nabla f_\tau \Vert_2$ is nonincreasing and $\inf_{g\in\mathcal M} \Vert \nabla(f_\tau-g) \Vert_2\to 0$ as $\tau \to \infty$. Clearly,
\begin{multline*}
\mathcal E(f) \ge \frac{ \Vert \nabla f \Vert^2_{\mathrm L^2(\R^d)} - S_d\,\Vert f \Vert^2_{\mathrm L^{2^*}(\R^d)}} {\Vert \nabla f \Vert^2_{\mathrm L^2(\R^d)}} = 1-S_d\,\frac{ \Vert f \Vert^2_{\mathrm L^{2^*}(\R^d)}} {\Vert \nabla f \Vert^2_{\mathrm L^2(\R^d)}} \\
\ge \frac{ \Vert \nabla f_\tau \Vert^2_{\mathrm L^2(\R^d)} - S_d\,\Vert f_\tau \Vert^2_{\mathrm L^{2^*}(\R^d)}} {\Vert \nabla f_\tau \Vert^2_{\mathrm L^2(\R^d)}}\,.
\end{multline*}
Starting with $\inf_{\bb\in\mathcal M} \Vert \nabla f-\nabla\bb \Vert_{\mathrm L^2(\R^d)}^2 > \delta\,\Vert \nabla f \Vert_{\mathrm L^2(\R^d)}^2$, one would like to run the flow until at a certain point $\tau_0$ one has
\begin{equation}\label{eq:equality}
\inf_{\bb\in\mathcal M} \Vert \nabla(f_{\tau_0}-\bb) \Vert_{\mathrm L^2(\R^d)}^2 = \delta\,\Vert \nabla f_{\tau_0} \Vert_{\mathrm L^2(\R^d)}^2 \,,
\end{equation}
so that
$$
\mathcal E(f) \ge \frac{ \Vert \nabla f_{\tau_0} \Vert^2_{\mathrm L^2(\R^d)} - S_d\,\Vert f_{\tau_0} \Vert^2_{\mathrm L^{2^*}(\R^d)}} {\Vert \nabla f_{\tau_0} \Vert^2_{\mathrm L^2(\R^d)}}
=\delta\,\frac{ \Vert \nabla f_{\tau_0} \Vert^2_{\mathrm L^2(\R^d)} - S_d\,\Vert f_{\tau_0} \Vert^2_{\mathrm L^{2^*}(\R^d)}} {\inf_{\bb\in\mathcal M} \Vert \nabla(f_{\tau_0}-\bb) \Vert_{\mathrm L^2(\R^d)}^2}\,.
$$
This would allow us to apply the first part of the proof to the function $f_{\tau_0}$ and obtain the desired bound. The details of this argument are more involved than presented here, mostly because the function $\tau\mapsto\|\nabla f_\tau\|_{\mathrm L^2(\R^d)}$ need not be continuous, so the existence of a $\tau_0$ as in~\eqref{eq:equality} is not guaranteed.

Continuous rearrangement flows in the setting of Steiner symmetrizations have been used by P\'olya--Szeg\H o~\cite[Note~B]{MR0043486}, Brock~\cite{Brock,Brock2} and others. For a recent application see~\cite{MR4022083}. In the setting of symmetric decreasing rearrangements of sets they were used by Bucur--Henrot~\cite{BucurHenrot} and we will generalize them to functions. Additional results on this flow, which might be useful in other contexts as well, are given in Appendix~\ref{contrearr}.

\medskip The proof of Corollary~\ref{logsob}, concerning stability for the log-Sobolev inequality, is given in Section~\ref{sec:logsob}. The underlying idea is that this inequality on $\R^N$ can be obtained by taking an appropriate limit in the Sobolev inequalities in dimension $d$, in the limiting regime as $d\to+\infty$, and that the same property should also be true for the corresponding stability inequalities. This large dimensional limit is accompanied by a rescaling argument and it is for scaling reasons that the $\dot{\mathrm H}^1(\R^d)$ distance in the stability term for the Sobolev inequality gives rise only to a stability estimate in $\mathrm L^2(\R^N,d\gamma)$ for the logarithmic Sobolev inequality. This $\mathrm L^2(\R^N,d\gamma)$-stability is not an artifact of our method of proof since, as we have already mentioned, the stability for the log-Sobolev inequality does not hold in $\mathrm H^1(\R^N,d\gamma)$ according to~\cite{indrei2023sharp}.

We also note that in~\cite{DEFFL2024} we give an alternative, direct proof of the stability for the log-Sobolev inequality, which runs through rearrangements on Gauss space, but otherwise the strategy is essentially the same as in the proof of our Theorem~\ref{main}.

\medskip Throughout this paper we deal with real-valued functions. With minor additional effort our arguments can be extended to the case of complex-valued functions. In order to make notations lighter, we will write $\|\cdot\|_q=\|\cdot\|_{\mathrm L^q(\R^d)}$ whenever the space is~$\R^d$ with Lebesgue measure.

%%%%%%%%%%%%%%%%%%%%%%%%%%%%%%%%%%%%%%%%%%%%%%%%%%
%%%%%%%%%%%%%%%%%%%%%%%%%%%%%%%%%%%%%%%%%%%%%%%%%%
\section{Local stability for nonnegative functions}\label{sec:2}

Our goal in this section is to prove a quantitative stability inequality for nonnegative functions close to the manifold of optimizers. In order to simplify the notation, we write in this section
$$
q = 2^* = 2\,d/(d-2) \,,
\quad
\theta=q-2=4/(d-2)
$$
and
\be{A}
\mathsf A= \tfrac14\,d\,(d-2)\,.
\ee

%%%%%%%%%%%%%%%%%%%%%%%%%%%%%%%%%%%%%%%%%%%%%%%%%%
\subsection{The Sobolev inequality on the sphere}\label{sec:preliminary}

It is well known that the Sobolev inequality on $\R^d$ has an equivalent formulation on $\Sph^d$, the unit sphere in $\R^{d+1}$. It will be convenient for us at several steps of our proof to carry out the arguments in the setting of $\Sph^d$. Let us give some details.

We denote by $\omega = (\omega_1, \omega_2, \dots, \omega_{d+1})$ the coordinates in $\R^{d+1}$. Then the unit sphere $\Sp^d \subset \R^{d+1}$ can be parametrized in terms of stereographic coordinates by
$$
\omega_j = \frac{2\,x_j}{1+|x|^2}\,,\quad j=1, \dots, d\,,\quad \omega_{d+1} = \frac{1-|x|^2}{1+|x|^2}\,.
$$
To a function $f$ on $\R^d$ we associate a function $F$ on $\Sph^d$ via
\begin{equation} \label{eq:euclidtosphere}
F(\omega) = \left( \frac{1+|x|^2}2\right)^{\frac{d-2}2} f(x)\quad\forall\,x\in\R^d\,.
\end{equation}
Then, since $\big(2/(1+|x|^2)\big)^d$ is the Jacobian of the inverse stereographic projection $x\mapsto \omega$,
$$
|\Sph^d| \int_{\Sph^d} |F(\omega)|^{2^*}\,d\mu(\omega) = \int_{\R^d} |f(x)|^{2^*}\,dx\,,
$$
where $\mu$ denotes the {\it uniform probability measure\/} on $\Sph^d$. Moreover, $F\in \mathrm H^1(\Sph^d)$ if and only if $f\in\dot{\mathrm H}^1(\R^d)$, and in this case
$$
|\Sph^d| \int_{\Sph^d} \left( |\nabla F|^2 + \mathsf A\,|F|^2\right)d\mu(\omega) = \int_{\R^d} |\nabla f|^2\,dx\,.
$$
Therefore, with $\mathsf A$ given by~\eqref{A}, the sharp Sobolev inequality on $\R^d$ is equivalent to the following sharp Sobolev inequality on $\Sph^d$,
$$
\int_{\Sph^d} \left( |\nabla F|^2 + \mathsf A\,|F|^2\right)d\mu \geq \mathsf A \left( \int_{\Sph^d} |F|^{2^*}\,d\mu \right)^{2/2^*}\quad\forall\,F\in\mathrm H^1(\Sp^d,d\mu)\,,
$$
with equality exactly for the functions
$$
G(\omega) = c\,\big(a+b\cdot \omega\big)^{-\frac{d-2}2}\,,
$$
where $a>0$, $b\in \R^d$ and $c \in \R$ are constants with $|b|<a$. We denote the corresponding set of functions by~$\mathscr M$. Then the above equivalence shows that
$$
\mathcal E(f) = \frac{ \Vert \nabla f \Vert^2_2 - S_d\,\Vert f \Vert^2_{2^*}} {\inf_{\bb\in\mathcal M}\kern-2pt \Vert \nabla f - \nabla\bb \Vert^2_2} \!=\!
\frac{ \Vert \nabla F \Vert^2_{\mathrm L^2(\Sp^d)} + \mathsf A\,\Vert F\Vert_{\mathrm L^2(\Sp^d)}^2 - S_d\,\Vert F \Vert^2_{\mathrm L^{2^*}(\Sp^d)}} {\inf_{G \in \mathscr M}\kern-2pt\left\{\Vert \nabla F- \nabla G \Vert^2_{\mathrm L^2(\Sp^d)} \!+\! \mathsf A\,\Vert F\!\!-\!\!G\Vert_{\mathrm L^2(\Sp^d)}^2\right\} }\,.
$$

%%%%%%%%%%%%%%%%%%%%%%%%%%%%%%%%%%%%%%%%%%%%%%%%%%
\subsection{A stability result for functions close to the manifold of optimizers}\label{sec:close}

%-------------------------------------------------
\begin{theorem}\label{unifboundclose}
Let $q = 2^* = 2\,d/(d-2)$ and $\theta=q-2=4/(d-2)$. There are explicit constants $\epsilon_0>0$ and $\tilde\delta\in(0,1/2)$ such that for all $d\geq 3$ and for all $-\,1 \leq r\in \mathrm H^1(\Sph^d)$ satisfying
\begin{equation}
\label{eq:smallsphere}
\left( \isd{|r|^q} \right)^{2/q} \leq\tilde\delta
\end{equation}
and
\begin{equation}
\label{eq:orthosphere}
\isd r = 0 = \isd{\omega_j\,r}\,,
\quad j=1,\ldots,d+1\,,
\end{equation}
one has
\begin{multline*}
\isd{\left( |\nabla r|^2 +\mathsf A\,(1+r)^2 \right)} -\mathsf A\left( \isd{(1+r)^q}\right)^{2/q} \\ \geq \theta\,\epsilon_0 \isd{\left( |\nabla r|^2 +\mathsf A\,r^2\right)}\,.
\end{multline*}
\end{theorem}
%-------------------------------------------------
The key feature of this theorem is that the constant $\theta\,\epsilon_0$ behaves like $4\,\epsilon_0\,d^{-1}$ for large $d$. This $d^{-1}$ behavior leads to a corresponding lower bound on the behavior of $\mathbf{\mathscr C}_{d,\rm BE}$, which in view of~\eqref{eq:upperbound} is optimal.
%-------------------------------------------------
\begin{remark}\label{unifboundcloserem}
In fact, we show that for every $0<\epsilon_0<\tfrac13$ there is a $\tilde\delta>0$ such that the assertion in the theorem holds for all $d\geq 6$. The same argument also gives that for every $0<\epsilon_0<\tfrac12$ there is a~$D$ and a $\tilde\delta>0$ such that the assertion of the theorem holds for all $d\geq D$. The explicit expression for $\tilde\delta>0$ can be found in the proofs of Theorem~\ref{unifboundclose}, Proposition~\ref{Prop:estI3} and in~\eqref{delta1} below.
\end{remark}
%-------------------------------------------------
The proof of Theorem~\ref{unifboundclose} will take up the rest of this section.

%%%%%%%%%%%%%%%%%%%%%%%%%%%%%%%%%%%%%%%%%%%%%%%%%%
\subsection{The spectral gap inequality}\label{sec:spectralGap}

Of crucial importance in our analysis, just like in that of Bianchi and Egnell~\cite{BianchiEgnell}, is the following spectral bound. It appears, for instance, in Rey's paper~\cite[Appendix~D]{Rey} slightly before the work of Bianchi and Egnell.
%-------------------------------------------------
\begin{lemma}\label{gap}
Let $d\ge3$ and assume that $r\in \mathrm H^1(\Sph^d)$ satisfies~\eqref{eq:orthosphere}. Then
$$
\isd{\left( |\nabla r|^2 - d\,r^2 \right)} \geq \frac4{d+4} \isd{\left( |\nabla r|^2+\mathsf A\,r^2 \right)}\,.
$$
\end{lemma}
%-------------------------------------------------
\begin{proof}
We recall that the Laplace--Beltrami operator on $\Sph^d$ is diagonal in the basis of spherical harmonics and that its eigenvalue on spherical harmonics of degree $\ell$ is $\ell\,(\ell+d-1)$.

Conditions~\eqref{eq:orthosphere} mean that $r$ is orthogonal to spherical harmonics of degrees $\ell\leq 1$. Diagonalizing the Laplace--Beltrami operator, the claimed inequality becomes
$$
\ell\,(\ell+d-1) - d \geq \frac4{d+4} \big( \ell\,(\ell+d-1) + \mathsf A\big)
\quad\text{for all}\quad\ell\geq 2\,.
$$
This is elementary to check.
\end{proof}

%%%%%%%%%%%%%%%%%%%%%%%%%%%%%%%%%%%%%%%%%%%%%%%%%%
\subsection{Warm-up: A bound with suboptimal dimension dependence}\label{sec:firstbound}

In this subsection we prove a preliminary version of Theorem~\ref{unifboundclose} where the constant $\theta\,\epsilon_0$ on the right side is replaced by some $d$-dependent constant, which decreases much faster than $d^{-1}$ as $d$ increases.

The motivation for proving this preliminary version is threefold. First, it explains the basic strategy of the proof without the additional difficulty of tracking the dependence on $d$. The latter will require some rather elaborate additional arguments. Second, this more involved proof works nicely when the exponent $q=2^*$ is $\leq 3$, which means $d\geq 6$. (It is, however, not difficult to adjust it to arbitrary $d$.) Therefore our chosen proof of Theorem~\ref{unifboundclose} will combine the inequality proved in this subsection for $d=3$, $4$, $5$ with the inequality proved in the next subsection for $d\geq 6$. Third, the simpler argument in this subsection gives simpler expressions for the relevant constants, which might be preferable in certain applications in low dimensions where the values of these constants play a role.
%-------------------------------------------------
\begin{proposition}\label{firstbound}
For all $\tilde\delta>0$ and for all $-\,1 \leq r\in \mathrm H^1(\Sph^d)$ satisfying~\eqref{eq:smallsphere} and~\eqref{eq:orthosphere} one has
\begin{multline*}
\isd{\left( |\nabla r|^2 +\mathsf A\,(1+r)^2 \right)} -\mathsf A\,\left( \isd{(1+r)^q} \right)^{2/q} \\ \geq \mathsf m(\tilde\delta^{1/2}) \isd{\left( |\nabla r|^2 +\mathsf A\,r^2 \right)}\,,
\end{multline*}
where $d\mu$ is the uniform probability measure, with
\begin{equation} \label{eq:mudelta}
\begin{aligned}
\mathsf m(\nu):=\,&\tfrac4{d+4} - \tfrac2q\,\nu^{q-2} && \text{if}\quad d\geq 6\,,\\
\mathsf m(\nu):=\,&\tfrac4{d+4} -\tfrac13\,(q-1)\,(q-2)\,\nu - \tfrac2q\,\nu^{q-2} && \text{if}\quad d=4\,,\,5\,,\\
\mathsf m(\nu):=\,&\tfrac47 - \tfrac{20}3\,\nu - 5\,\nu^2- 2\,\nu^3 - \tfrac13\,\nu^4 && \text{if}\quad d=3\,.
\end{aligned}
\end{equation}
\end{proposition}
%-------------------------------------------------
We note that for any $d\geq 3$ there is a $\nu_d$ such that $\mathsf m(\nu)>0$ for $\nu<\nu_d$. Thus, for $\tilde\delta<\nu_d^2$ we obtain a stability inequality.

\medskip We begin the proof of Proposition~\ref{firstbound} with some elementary inequalities.
%-------------------------------------------------
\begin{lemma}\label{ineq2}
If $q\geq 2$, then, for all $t\geq 0$,
$$
(1+t)^{\frac 2q} \leq 1 + \tfrac2q\,t\,.
$$
\end{lemma}
%-------------------------------------------------
This is well known and we omit its simple proof.
%-------------------------------------------------
\begin{lemma}\label{ineq1}
We have the following bounds.
\begin{itemize}
\item If $2\leq q\leq 3$, then, for all $t\ge -\,1$,
$$
(1+t)^q \leq 1+ q\,t + \tfrac12\,q\,(q-1)\,t^2 + t_+^q\,.
$$
\item If $3\leq q\leq 4$, then, for all $t\ge -\,1$,
$$
(1+t)^q \leq 1+ q\,t + \tfrac12\,q\,(q-1)\,t^2 + \tfrac16\,q\,(q-1)\,(q-2)\,t^3 + |t|^q\,.
$$
\end{itemize}
\end{lemma}
%-------------------------------------------------

Similar bounds can also be derived for real $q\in(4,\infty)$. They become increasingly more complicated each time $q$ passes an integer. The only bound for $q>4$ that we shall need corresponds to the critical exponent $q=6$ when $d=3$. In that case, we rely on the binomial expansion $(1+t)^6=1+6\,t+15\,t^2+20\,t^3+15\,t^4+6\,t^5+t^6$.

\begin{proof}
The case $q=2$ is trivial. We begin with the case $2< q\leq 3$ and set
$$
\phi(t) := (1+t)^q - 1 - q\,t - \tfrac12\,q\,(q-1)\,t^2 - t_+^q\,.
$$
For any $t\ge -\,1$, we compute
\begin{align*}
\phi'(t) & = q \left( (1+t)^{q-1} - 1 - (q-1)\,t - t_+^{q-1} \right),\\
\phi''(t) & = q\,(q-1) \left( (1+t)^{q-2} -1 -t_+^{q-2} \right).
\end{align*}
For $-\,1\leq t\leq 0$ we clearly have $(1+t)^{q-2} -1 -t_+^{q-2} = (1-|t|)^{q-2} - 1 \leq 0$. For $t\geq 0$ we have, by a well-known elementary inequality, $(1+t)^{q-2} -1 -t_+^{q-2} = (1+t)^{q-2} - 1 -t^{q-2} \leq 0$. To summarize, $\phi$ is concave on $[-1,\infty)$. We conclude that, for all $t\ge -\,1$,
$$
\phi(t) \leq \phi(0) - \phi'(0)\,t\,.
$$
Since $\phi(0)=\phi'(0)=0$, this is the claimed inequality.

We now turn to the case $3\leq q\leq 4$ and set this time
$$
\phi(t) := (1+t)^q - 1 - q\,t - \tfrac12\,q\,(q-1)\,t^2 - \tfrac16\,q\,(q-1)\,(q-2)\,t^3 - |t|^q\,.
$$
Again, we compute
\begin{align*}
\phi'(t) & = q \left( (1+t)^{q-1} - 1 - (q-1)\,t - \tfrac12\,(q-1)\,(q-2)\,t^2 - |t|^{q-2}\,t \right),\\
\phi''(t) & = q\,(q-1)\,\Big( (1+t)^{q-2} - 1 - (q-2)\,t - |t|^{q-2} \Big)\,.
\end{align*}
Since again $\phi(0)=\phi'(0)=0$, the claimed inequality will follow if we can show concavity of $\phi$ on $[-1,\infty)$, that is, $\psi\leq 0$ on $[-1,\infty)$ where
$$
\psi(t):= (1+t)^{q-2} - 1 - (q-2)\,t - |t|^{q-2}\,.
$$
We compute
\begin{align*}
\psi'(t) & = (q-2) \left( (1+t)^{q-3} - 1 - |t|^{q-4}\,t \right),\\
\psi''(t) & = (q-2)\,(q-3) \left( (1+t)^{q-4} - |t|^{q-4} \right).
\end{align*}
We discuss $\psi$ separately on $[-1,0]$ and on $(0,\infty)$.
\begin{itemize}
\item[$\circ$] We begin with the second case. For $t>0$ we have, by the same elementary inequality as before, $(1+t)^{q-3}-1-t^{q-3}< 0$. Thus, $\psi'<0$ on $(0,\infty)$. Since $\psi(0)=0$, we deduce $\psi<0$ on $(0,\infty)$.
\item[$\circ$] Now let us consider the interval $[-1,0]$. We see that $\psi''>0$ on $(-1,-1/2)$ and $\psi''<0$ on $(-1/2,0)$. Therefore $\psi'$ is increasing on $(-1,-1/2)$ and decreasing on $(-1/2,0)$. Since $\psi'(-1)=\psi'(0)=0$, we conclude that $\psi'>0$ on $(-1,0)$ and therefore $\psi$ is increasing on $(-1,0)$. Since $\psi(0)=0$ we conclude that $\psi<0$ on $[-1,0)$, as claimed.
\end{itemize}
This completes the proof of the lemma.
\end{proof}

{}From Lemmas~\ref{ineq2} and~\ref{ineq1} we easily obtain the following inequalities.
%-------------------------------------------------
\begin{proposition}\label{expand}
Let $(X,d\mu)$ be a measure space and $u$, $r\in \mathrm L^q(X,d\mu)$ for some $q\geq 2$ with $u\geq 0$ and $u+r\geq 0$. Assume also that $\int_X u^{q-1}\,r\,d\mu =0$.
\begin{itemize}
\item If $2\leq q\leq 3$, then
$$
\|u+r\|_q^2 \leq \|u\|_q^2 + \|u\|_q^{2-q} \left( (q-1) \!\int_X u^{q-2}\,r^2\,d\mu + \frac2q \!\int_X r_+^q\,d\mu \right).
$$
\item If $3\leq q\leq 4$, then
\begin{multline*}
\|u+r\|_q^2 \leq \,\|u\|_q^2+(q-1)\,\|u\|_q^{2-q}  \!\int_X u^{q-2}\,r^2\,d\mu\\
+ \|u\|_q^{2-q} \left(  \tfrac13\,(q-1)\,(q-2) \!\int_X u^{q-3}\,r^3\,d\mu + \tfrac2q \!\int_X |r|^q\,d\mu \right).
\end{multline*}
\item If $q=6$, then
\begin{multline*}
\|u+r\|_q^2 \leq\,\|u\|_q^2 + \|u\|_q^{2-q}\,\left( 5\!\int_X u^{q-2}\,r^2\,d\mu + \tfrac{20}3\!\int_X u^{q-3}\,r^3\,d\mu\right.\\
\left. +\,5\!\int_X u^{q-4}\,r^4\,d\mu + 2\!\int_X u^{q-5}\,r^5\,d\mu + \tfrac13\!\int_X r^6\,d\mu \right)\,.
\end{multline*}
\end{itemize}
\end{proposition}
%-------------------------------------------------
\begin{proof}
For $2\leq q\leq 3$ we have, by Lemma~\ref{ineq1}, almost everywhere on $X$,
$$
(u+r)^q \leq u^q + q\,u^{q-1}\,r + \tfrac12\,q\,(q-1)\,u^{q-2}\,r^2 + r_+^q\,.
$$
Integrating this and using the assumed orthogonality condition, we obtain
$$
\int_X (u+r)^q\,d\mu \leq \int_X u^q\,d\mu + \tfrac12\,q\,(q-1) \!\int_X u^{q-2}\,r^2\,d\mu + \int_X r_+^q\,d\mu\,.
$$
Applying Lemma~\ref{ineq2}, we obtain
\begin{multline*}
\left( \int_X (u+r)^q\,d\mu \right)^\frac2q \leq \left( \int_X u^q\,d\mu \right)^\frac2q \\ + \left( \int_X u^q\,d\mu \right)^\frac{2-q}q \left( (q-1)\! \int_X u^{q-2}\,r^2\,d\mu + \tfrac2q\! \int_X r_+^q\,d\mu \right).
\end{multline*}
This is the claimed inequality for $2\leq q\leq 3$. The proof for $3<q\leq 4$ is similar and the inequality for $q=6$ follows from expanding the polynomial.
\end{proof}

\begin{proof}[Proof of Proposition~\ref{firstbound}]
Let $r$ be as in Theorem~\ref{unifboundclose}. Because of the mean-zero condition we can apply Proposition~\ref{expand} with $u=1$ on $X=\Sph^d$ and $d\mu$ the uniform probability measure. We simplify the resulting term using H\"older and Sobolev, which imply for $2<t\leq q$,
$$
\isd{|r|^t} \leq \left( \isd{|r|^q} \right)^{t/q} \leq \tilde\delta^\frac{t-2}2\,\mathsf A^{-1} \isd{\left( |\nabla r|^2 +\mathsf A\,r^2 \right)}\,.
$$
In this way, we obtain
\begin{multline*}
\left( \isd{(1+r)^q} \right)^{2/q} \leq 1 + (q-1)\isd{r^2} \\ + \mathsf n(\tilde\delta^{1/2})\,\mathsf A^{-1} \isd{\left( |\nabla r|^2 +\mathsf A\,r^2 \right)}\,,
\end{multline*}
where
\[
\begin{aligned}
\mathsf n(\nu):=\,& \tfrac2q\,\nu^{q-2} && \text{if}\quad d\geq 6\,,\\
\mathsf n(\nu):=\,& \tfrac13\,(q-1)\,(q-2)\,\nu + \tfrac2q\,\nu^{q-2} && \text{if}\quad d=4\,,\,5\,,\\
\mathsf n(\nu):=\,& \tfrac{20}3\,\nu + 5\,\nu^2 + 2\,\nu^3 + \tfrac13\,\nu^4 && \text{if}\quad d=3\,.
\end{aligned}
\]
Using $\mathsf A\,(q-2)=d$, we deduce that
\begin{multline*}
\isd{\left( |\nabla r|^2 +\mathsf A\,(1+r)^2 \right)} -\mathsf A\left( \isd{(1+r)^q }\right)^{2/q}\\
\geq \isd{\left( |\nabla r|^2 - d r^2 \right)} - \mathsf n(\tilde\delta^{1/2}) \isd{\left( |\nabla r|^2 +\mathsf A\,r^2 \right)}\,.
\end{multline*}
Using the spectral gap inequality in Lemma~\ref{gap} and noting that $\mathsf m(\nu) = \frac{4}{d+4}- \mathsf n(\nu)$, we obtain the claimed inequality.
\end{proof}
%-------------------------------------------------
\begin{remark}\label{firstboundrem}

The estimates of Proposition~\ref{firstbound} are good enough for proving Theorem~\ref{unifboundclose} for~$d$ finite, but fail for proving that the stability constant is of the order of $\theta\,\epsilon_0$ in the large $d$ limit, for some positive~$\epsilon_0$ independent of $d$ and $\theta=q-2=4/(d-2)$. Indeed, if we write that $\mathsf m(\nu)\ge\theta\,\epsilon_0$, we obtain
\[
\nu^{q-2}\le\frac q2\(\frac4{d+4}-(q-2)\,\epsilon_0\)\le\frac q2\,\frac4{d+4}=\frac{4\,d}{(d-2)\,(d+4)}\le\frac4{d-2}\,,
\]
which means $\nu\le\big(\tfrac{d-2}4\big)^{-\tfrac{d-2}4}<\sqrt{\tilde\delta}$ for $d$ large enough, for any given $\tilde\delta>0$. Theorem~\ref{unifboundclose} cannot be deduced from Proposition~\ref{firstbound} as $d\to+\infty$ and this is why we need better estimates.
\end{remark}
%-------------------------------------------------

%%%%%%%%%%%%%%%%%%%%%%%%%%%%%%%%%%%%%%%%%%%%%%%%%%
\subsection{Cutting \texorpdfstring{$r$}{r} into pieces}\label{Sec:Pieces}

We turn now to the proof of Theorem~\ref{unifboundclose} with the optimal dependence of the constant on the dimension. Thus, until the end of Section~\ref{sec:close} we will assume that $r$ satisfies the assumptions of Theorem~\ref{unifboundclose}. The following proposition gives an upper bound on
$$
(1+r)^q-1-q\,r
$$
for real numbers $r$ in terms of three numbers
\be{r1r2r3}
r_1:=\min\{r,\gamma\}\,,\quad r_2:=\min\{(r-\gamma)_+,M-\gamma\}\quad\mbox{and}\quad r_3:=(r-M)_+
\ee
where $\gamma$ and $M$ are parameters such that $0<\gamma<M$. Notice that
\[
r=r_1+r_2+r_3\,.
\]
We will later apply this when $r$ is a function. Our goal is to obtain a bound in terms of
\be{theta}
\theta:=q-2\quad\mbox{where}\quad q=2^*=\frac{2\,d}{d-2}
\,.
\ee
We have in mind to let $d\to+\infty$ so that $\theta\to0_+$.
%-------------------------------------------------
\begin{proposition}\label{Prop:ptw}
Given $M\in(0,+\infty)$ and $\overline M\in[\sqrt e,+\infty)$, there are two positive constants $C_M\,$ and $C_{M,\overline M}\,$ depending respectively only on $M$ and $\{M,\overline M\}$ such that, for any $\gamma\in(0,M]$, $q\in[2,3]$ and $r\in [-1,\infty)$, we have
\begin{multline}\label{eq:propinitial}
(1+r)^q - 1 - q\,r \leq \tfrac12\,q\,(q-1)\,(r_1+r_2)^2
\\ +2\,(r_1+r_2)\,r_3 + \left(1+C_M\,\theta\,\overline M^{-1}\ln\overline M\right) r_3^q\\
+ \left( \tfrac32\,\gamma\,\theta\,r_1^2 + C_{M,\overline M}\,\theta\,r_2^2 \right) \1_{\{r\leq M\}} + C_{M,\overline M}\,\theta\,M^2\,\1_{\{r>M\}}
\end{multline}
with $r_1$, $r_2$, $r_3$ and $\theta$ given by~\eqref{r1r2r3} and~\eqref{theta}.
\end{proposition}
%-------------------------------------------------

\medskip For the proof of Proposition~\ref{Prop:ptw}, we need two elementary lemmas.
%-------------------------------------------------
\begin{lemma}\label{firstlemma}
If $2\leq q\leq 3$, then for all $r\in[-1,\infty)$,
$$
(1+r)^q \leq 1+ q\,r + \tfrac12\,q\,(q-1)\,r^2 + (q-2)\,r_+^3\,.
$$
\end{lemma}
%-------------------------------------------------
\begin{proof}
The inequality for $-\,1\leq r\leq 0$ follows from Lemma~\ref{ineq1}. Let now $r\geq 0$. Then
$$
(1+r)^q - 1- q\,r - \tfrac12\,q\,(q-1)\,r^2 \!=\! q\,(q-1)\,(q-2) \!\int_0^r \!\int_0^s \!\int_0^t (1+u)^{q-3}\,du\,dt\,ds\,.
$$
Since $q\leq 3$ we have $(1+u)^{q-3}\leq 1$ and therefore
\begin{multline*}
q\,(q-1)\,(q-2) \int_0^r \int_0^s \int_0^t (1+u)^{q-3}\,du\,dt\,ds\\ \qquad
\leq  q\,(q-1)\,(q-2) \int_0^r \int_0^s \int_0^t\,du\,dt\,ds
= \tfrac q3\,\tfrac{q-1}2\,(q-2)\,r^3 \leq (q-2)\,r^3\,,
\end{multline*}
as claimed.
\end{proof}
%-------------------------------------------------
\begin{lemma}\label{secondlemma}
For all $q\geq 2$ and all $v\geq \overline M\geq \sqrt e$ we have
\begin{multline*}
q\,v^{q-1} - 2\,v \leq \frac{1+2\ln\overline M}{\overline M}\,(q-2)\,v^q\quad\text{and}\\
\tfrac12\,q\,(q-1)\,v^{q-2} - 1 \leq \frac{\tfrac{1+q}2 +\ln\overline M}{\overline M^2}\,(q-2)\,v^q\,.
\end{multline*}
\end{lemma}
%-------------------------------------------------
\begin{proof}
Let
$$
v_*^{(1)}:= \big(2\,\tfrac{q-1}q\big)^\frac1{q-2}
\quad\text{and}\quad
v_*^{(2)} := \big(\tfrac1{q-1}\big)^\frac1{q-2}\,.
$$
Then an elementary computation shows that $v\mapsto q\,v^{-1} - 2\,v^{1-q}$ is increasing on $\big(0,v_*^{(1)}\big]$ and decreasing on $\big[v_*^{(1)},\infty\big)$. Similarly $v\mapsto \tfrac12\,q\,(q-1)\,v^{-2} - v^{-q}$ is increasing on $\big(0,v_*^{(2)}\big]$ and decreasing on $\big[v_*^{(2)},\infty\big)$. Thus,
$$
q\,v^{q-1} - 2\,v \leq \left( q\,\overline M^{-1} - 2\,\overline M^{1-q} \right) v^q
\quad\text{for all}\quad v\geq \overline M \geq v_*^{(1)}
$$
and
$$
\tfrac12\,q\,(q-1)\,v^{q-2} - 1 \leq \left( \tfrac12\,q\,(q-1)\,\overline M^{-2} - \overline M^{-q} \right)_+\,v^q\;,
\quad\forall\, v\geq \overline M \geq v_*^{(2)}\,.
$$
One has $v_*^{(1)}\geq 1 \geq v_*^{(2)}$ and, using $\ln t\leq t-1$ for all $t>0$, we find
$$
\ln v_*^{(1)} \leq \tfrac1q \leq\tfrac12\,,
\quad\text{that is,}\quad\,v_*^{(1)} \leq \sqrt e\,.
$$
Thus, the above inequality hold, in particular, for $v\geq \overline M\geq\sqrt e$.

Moreover, using $1-t^{-1}\leq\ln t$ for $t>1$ we can bound
\begin{multline*}
q\,\overline M^{-1} - 2\,\overline M^{1-q} = (q-2)\,\overline M^{-1} + 2\(\overline M^{-1} - \overline M^{1-q}\) \\\leq (q-2)\,\overline M^{-1}\(1+ 2\,\ln\overline M\)
\end{multline*}
and
\begin{multline*}
\tfrac12\,q\,(q-1)\,\overline M^{-2} - \overline M^{-q} = \( \tfrac12\,q\,(q-1) - 1\)\overline M^{-2} + \(\overline M^{-2}-\overline M^{-q}\) \\ \leq (q-2)\,\overline M^{-2} \( \tfrac{1+q}2 +\ln\overline M\)\,.
\end{multline*}
This proves the assertion.
\end{proof}

\begin{proof}[Proof of Proposition~\ref{Prop:ptw}]
We now turn to the proof of~\eqref{eq:propinitial}. Assume first that $r\leq M$. We apply Lemma~\ref{firstlemma} and obtain
$$
(1+r)^q - 1 - q\,r \leq \tfrac12\,q\,(q-1)\,(r_1+r_2)^2 + \theta\,(r_1+r_2)_+^3\,.
$$
If $r\leq\gamma$, then $r_2=0$ and~\eqref{eq:propinitial} follows from $(r_1)_+^3\leq \gamma\,r_1^2 \leq \tfrac32\,\gamma\,r_1^2$. If $\gamma<r\leq M$, we have, since $r_1=\gamma$ and $3\,r_1\,r_2\le\tfrac12\,r_1^2+\tfrac92\,r_2^2$, we have
$$
(r_1+r_2)_+^3 = \gamma\,r_1^2 + 3\,\gamma\,r_1\,r_2 + 3\,\gamma\,r_2^2 + r_2^3 \leq \tfrac32\,\gamma\,r_1^2 + \(\tfrac{15}2\,\gamma+M\)r_2^2\,.
$$
Since $\gamma\leq M$ this proves~\eqref{eq:propinitial} with $C_{M,\overline M}\geq\tfrac{17}2\,M$.

\medskip From here on, let us consider the case $r>M$. Using $r=M+r_3$ we can write
\begin{multline*}
(1+r)^q -1-q\,r = (1+r)^q - (1+r)^2 +(1+M)^2 \\ -1-q\,M - (q-2)\,r_3 + r_3^2 + 2\,M\,r_3\,.
\end{multline*}
We use
\begin{multline*}
(1+M)^2 - 1 -q\,M - \tfrac12\,q\,(q-1)\,M^2 \\ = -\,\tfrac12\,(q-2)\,M\,\big(2+(q+1)\,M\big)\leq 0
\end{multline*}
as well as $-\,(q-2)\,r_3\leq 0$, to get
\begin{align}\label{eq:ptwlarge}
(1+r)^q -1-q\,r \leq \tfrac12\,q\,(q-1)\,M^2 + 2\,M\,r_3 + r_3^2 + (1+r)^q - (1+r)^2\,.
\end{align}
Note that the terms $\,2Mr_3=2\,(r_1+r_2)\,r_3\,$ and $\,\tfrac12\,q\,(q-1)\,M^2 = \tfrac12\,q\,(q-1)\,(r_1+r_2)^2$ are already of the form required in~\eqref{eq:propinitial}. In the following we bound the remaining terms $r_3^2 + (1+r)^q - (1+r)^2$. We do this separately in the cases $M<r\leq M+\overline M$ and $r>M+\overline M$, where $\overline M\geq 0$ is an additional parameter.

If $M<r\leq M+\overline M$, we have
$$
(1+r)^q - (1+r)^2 \leq C_{M,\overline M}^{(1)}\,\theta\quad\mbox{and}\quad r_3^2 - r_3^q \leq C_{\overline M}^{(1)}\,\theta\,.
$$
Inserting this into~\eqref{eq:ptwlarge}, we have for $M<r\leq M+ \overline M$
\begin{align*}
(1+r)^q -1-q\,r \leq 2\,M\,r_3 + r_3^q + \left( \tfrac12\,q\,(q-1) + C_{M,\overline M}\,\theta \right) M^2\,,
\end{align*}
provided
$$
C_{M,\overline M}\geq M^{-2}\(C_{M,\overline M}^{(1)} + C_{\overline M}^{(1)}\)\,.
$$
This is a bound of the form~\eqref{eq:propinitial}, since $r_1+r_2=M$ for $r>M$.

Next, we consider the case $r>M+\overline M$, that is $r_3=r-M>\overline M$. By Lemma~\ref{firstlemma} we have
\begin{multline*}
(1+r)^q  = (1+M + r_3)^q = r_3^q\,\big( 1 + \tfrac{1+M}{r_3} \big)^q\\
 \leq r_3^q + q\,r_3^{q-1}\,(1+M) + \tfrac12\,q\,(q-1)\,r_3^{q-2}\,(1+M)^2 + \theta\,r_3^{q-3}\,(1+M)^3\\
 \leq r_3^q + q\,r_3^{q-1}\,(1+M) + \tfrac12\,q\,(q-1)\,r_3^{q-2}\,(1+M)^2 + \theta\,\overline M^{q-3}\,(1+M)^3\\
\quad = r_3^q + q\,r_3^{q-1}\,(1+M) + \tfrac12\,q\,(q-1)\,r_3^{q-2}\,(1+M)^2 + C_{M,\overline M}^{(2)}\,\theta\,.
\end{multline*}
In the last inequality, we used $q\leq 3$ and $r_3>\overline M$. This, together with
$$
(1+r)^2 = (1+M+r_3)^2 = r_3^2 + 2\,r_3\,(1+M) + (1 + M)^2\,,
$$
gives
\begin{multline*}
\tfrac12\,q\,(q-1)\,M^2 + 2\,M\,r_3 + r_3^2 + (1+r)^q - (1+r)^2\\
\leq 2\,M\,r_3 + r_3^q + \left( q\,r_3^{q-1} - 2\,r_3 \right)(1+M)\\
+ \left( \tfrac12\,q\,(q-1)\,r_3^{q-2} - 1 \right)(1+M)^2
+ C_{M,\overline M}^{(2)}\,\theta + \tfrac12\,q\,(q-1)\,M^2\,.
\end{multline*}
We now assume that $\overline M\geq \sqrt e$. Then, by Lemma~\ref{secondlemma},
\begin{multline*}
q\,r_3^{q-1} - 2\,r_3 \leq \frac{1+2\,\ln\overline M}{\overline M}\,\theta\,r_3^q\quad\mbox{and}\\ \tfrac12\,q\,(q-1)\,r_3^{q-2} - 1 \leq \frac{2+\ln\overline M}{\overline M^2}\,\theta\,r_3^q\,.
\end{multline*}
Thus,
\begin{multline*}
\tfrac12\,q\,(q-1)\,M^2+2\,M\,r_3+r_3^2+(1+r)^q-(1+r)^2\\
\leq2\,M\,r_3+\left(1+\frac{C_M\,\ln\overline M}{\overline M}\,\theta\right)r_3^q+C_{M,\overline M}^{(2)}\,\theta+\tfrac12\,q\,(q-1)\,M^2\,,
\end{multline*}
where $C_M\,$ is a constant satisfying
$$
\frac{1+2\,\ln\overline M}{\overline M}\,(1+M) + \frac{2+\ln\overline M}{\overline M^2}\,(1+M)^2 \leq \frac{C_M\,\ln\overline M}{\overline M}
\quad\text{for all}\quad\overline M\geq \sqrt e\,.
$$
Combining this with~\eqref{eq:ptwlarge} we obtain a bound of the form~\eqref{eq:propinitial}, provided the constant $C_{M,\overline M}\,$ satisfies
$$
C_{M,\overline M}\geq M^{-2}\,C_{M,\overline M}^{(2)}\,.
$$
This concludes the proof with $$C_{M,\overline M}=M^{-2}\,\max\left\{C_{M,\overline M}^{(1)} + C_{\overline M}^{(1)},\,C_{M,\overline M}^{(2)}\right\}\,.$$
\end{proof}

%-------------------------------------------------
\begin{corollary}\label{Cor:ptw}
Given $\epsilon>0$, $M>0$, and $\gamma\in(0,M/2)$, there is a constant $C_{\gamma,\epsilon,M}>0$ with the following property: if $2\leq q\leq 3$, $r\in [-1,\infty)$, then
\begin{multline}\label{Claim:CuttingEstim}
(1+r)^q - 1 - q\,r \leq \(\tfrac12\,q\,(q-1) + 2\,\gamma\,\theta\) r_1^2 +\(\tfrac12\,q\,(q-1) + C_{\gamma,\epsilon,M}\,\theta\)r_2^2\\
+ 2\,r_1\,r_2 + 2\,(r_1+r_2)\,r_3 + (1+\epsilon\,\theta)\,r_3^q
\end{multline}
with $r_1$, $r_2$, $r_3$ and $\theta$ given by~\eqref{r1r2r3} and~\eqref{theta}.
\end{corollary}
%-------------------------------------------------
\begin{proof}
Since
\begin{multline*}
q\,(q-1)\,r_1\,r_2 = 2\,r_1\,r_2 + ( 3 + \theta )\,\theta\,r_1\,r_2 \leq 2\,r_1\,r_2 + 4\,\theta\,r_1\,r_2 \\ \leq 2\,r_1\,r_2 + \tfrac\gamma2\,\theta\,r_1^2 + \tfrac8\gamma\,\theta\,r_2^2
\end{multline*}
and
$$
C_{M,\overline M}\,M^2\,\1_{\{r>M\}} \leq 4\,C_{M,\overline M}\,(M-\gamma)^2\,\1_{\{r>M\}} \leq 4\,C_{M,\overline M}\,r_2^2\,,
$$
we deduce from~\eqref{eq:propinitial} that
\begin{multline*}
(1+r)^q - 1 - q\,r \leq\(\tfrac12\,q\,(q-1) + 2\,\gamma\,\theta\) r_1^2 + 2\,r_1\,r_2 + 2\,(r_1+r_2)\,r_3 \\ +\(\tfrac12\,q\,(q-1) + \tfrac8\gamma\,\theta+5\,C_{M,\overline M}\,\theta\)r_2^2
 + \left(1+C_M\,\theta\,\overline M^{-1}\ln\overline M\right) r_3^q\,.
\end{multline*}
Given any $M\geq2\,\gamma$, we choose $\overline M$ such that $\overline M\geq \sqrt e$ and $C_M\,\overline M^{-1}\ln\overline M\leq \epsilon$. Then~\eqref{Claim:CuttingEstim} follows with $C_{\gamma,\epsilon,M}=\tfrac8\gamma+5\,C_{M,\overline M}$.
\end{proof}

We will apply Corollary~\ref{Cor:ptw} for $q$ close to $2$ and the main point is how the constants depend on $q$. Apart from the `natural' terms $\tfrac12\,q\,(q-1)\,r_1^2$, $\tfrac12\,q\,(q-1)\,r_2^2$, $2\,r_1\,r_2$ and $2\,(r_1+r_2)\,r_3$, all other terms are multiplied by $\theta$, which is small in our application. Moreover, we have the freedom to choose~$\gamma$ and $\epsilon$ as small as we please (independent of $q$) and so the prefactors of the terms $r_1^2$ and $r_3^q$ are almost the natural ones. The price to be paid is a rather large constant in front of the error term involving~$r_2^2$. In order to have better estimates as $d\to+\infty$, more work is needed.

%%%%%%%%%%%%%%%%%%%%%%%%%%%%%%%%%%%%%%%%%%%%%%%%%%
\subsection{A detailed estimate of the deficit}\label{Sec:Detailed}

We assume that $-\,1\leq r\in \mathrm H^1(\Sph^d)$ satisfies the orthogonality conditions~\eqref{eq:orthosphere} as well as the smallness condition~\eqref{eq:smallsphere} with some $\tilde\delta$, and we show that, if this $\tilde\delta$ is small enough, given $\epsilon_0\in(0,\tfrac13)$, we obtain the claimed inequality.

Given two parameters $\epsilon_1$, $\epsilon_2>0$ we apply Corollary~\ref{Cor:ptw} with
\be{epsilon1-epsilon2-C}
\gamma=\frac{\epsilon_1}2\,,\quad\epsilon=\epsilon_2\quad\mbox{and}\quad C_{\gamma,\epsilon,M}=C_{\epsilon_1,\epsilon_2}\,.
\ee
In terms of these parameters, we decompose $r=r_1+r_2+r_3$. We obtain
$$
\isd{|\nabla r|^2} = \isd{|\nabla r_1|^2} + \isd{|\nabla r_2|^2} + \isd{|\nabla r_3|^2}
$$
and, since $r$ has mean zero,
$$
\isd{(1+r)^2} = 1 + \isd{r^2}\,.
$$
Moreover,
\begin{multline*}
\isd{r^2} = \isd{r_1^2} + \isd{r_2^2} + \isd{r_3^2} \\ + 2 \isd{r_1\,r_2} + 2 \isd{(r_1+r_2)\,r_3}\,.
\end{multline*}
According to Corollary~\ref{Cor:ptw} and using again the fact that $r$ has mean zero, we have
\begin{multline*}
\isd{(1+r)^q}  \leq 1 +\(\tfrac12\,q\,(q-1) + \epsilon_1\,\theta\) \isd{r_1^2} \\ +\(\tfrac12\,q\,(q-1) + C_{\epsilon_1,\epsilon_2}\,\theta\) \isd{r_2^2}
 + 2 \isd{r_1\,r_2} \\ + 2 \isd{(r_1+r_2)\,r_3} + (1+\epsilon_2\,\theta) \isd{r_3^q}\,.
\end{multline*}
Using $(1+x)^{2/q}\leq 1 + \tfrac2q\,x$, we obtain
\begin{multline*}
\left( \isd{(1+r)^q} \right)^{2/q}  \leq 1 + (q-1 + \tfrac2q\,\epsilon_1\,\theta) \isd{r_1^2} + \\ (q-1 + \tfrac 2q\, C_{\epsilon_1,\epsilon_2}\,\theta) \isd{r_2^2}
+ \tfrac4q \isd{r_1\,r_2} \\+ \tfrac4q \isd{(r_1+r_2)\,r_3} + \tfrac 2q\,(1+\epsilon_2\,\theta) \isd{r_3^q}\\
 \quad\leq 1 + (q-1 + \epsilon_1\,\theta) \isd{r_1^2} + (q-1 + C_{\epsilon_1,\epsilon_2}\,\theta) \isd{r_2^2}\\
 \qquad + 2 \isd{r_1\,r_2} + 2 \isd{(r_1+r_2)\,r_3} + \tfrac 2q\,(1+\epsilon_2\,\theta) \isd{r_3^q}\,.
\end{multline*}
In the last inequality we used $\frac 2q\leq 1$. For the final term, however, it is vital that we keep $\tfrac2q$. We thus have, for any $0<\epsilon_0\leq\theta^{-1}$,
\begin{multline*}
\isd{\left( |\nabla r|^2 +\mathsf A\,(1+r)^2 \right)} -\mathsf A\,\left( \isd{(1+r)^q} \right)^{2/q}\\
\geq \theta\,\epsilon_0 \isd{\left( |\nabla r|^2 +\mathsf A\,r^2 \right)}\\
\quad + (1-\theta\,\epsilon_0) \isd{\left( |\nabla r_1|^2 +\mathsf A\,r_1^2 \right)} -\mathsf A\,(q-1 + \epsilon_1\,\theta) \isd{r_1^2}\\
\quad + (1-\theta\,\epsilon_0) \isd{\left( |\nabla r_2|^2 +\mathsf A\,r_2^2 \right)} -\mathsf A\,(q-1 + C_{\epsilon_1,\epsilon_2}\,\theta) \isd{r_2^2}\\
\quad + (1-\theta\,\epsilon_0) \isd{\left( |\nabla r_3|^2 +\mathsf A\,r_3^2 \right)} - \tfrac 2q\,\mathsf A\,(1+\epsilon_2\,\theta) \isd{r_3^q}\,.
\end{multline*}
With another parameter $\sigma_0>0$ we define
\begin{multline*}
I_1  := (1-\theta\,\epsilon_0) \isd{\left( |\nabla r_1|^2 +\mathsf A\,r_1^2 \right)} \\ -\mathsf A\,(q-1 + \epsilon_1\,\theta) \isd{r_1^2} +\mathsf A\,\sigma_0\,\theta \isd{(r_2^2+r_3^2)}\,,\\
I_2  := (1-\theta\,\epsilon_0) \isd{\left( |\nabla r_2|^2 +\mathsf A\,r_2^2 \right)} -\mathsf A\,\big(q-1 + (\sigma_0 + C_{\epsilon_1,\epsilon_2})\,\theta\big) \isd{r_2^2}\,,\\
I_3  := (1-\theta\,\epsilon_0) \isd{\left( |\nabla r_3|^2 +\mathsf A\,r_3^2 \right)} \\ - \tfrac 2q\,\mathsf A\,(1+\epsilon_2\,\theta) \isd{r_3^q} -\mathsf A\,\sigma_0\,\theta \isd{r_3^2}\,.
\end{multline*}
We recall that $\mathsf A=\frac14\,d\,(d-2)$. For later purposes, we note that $\mathsf A\,\theta =\mathsf A\,(q-2)=d$ and
\begin{multline*}
I_1  = (1-\theta\,\epsilon_0) \isd{|\nabla r_1|^2} - d\,(1 + \epsilon_0 + \epsilon_1 ) \isd{r_1^2} \\ + d\,\sigma_0 \isd{(r_2^2+r_3^2)}\,,\\
I_2  = (1-\theta\,\epsilon_0) \isd{|\nabla r_2|^2} - d\,(1 + \epsilon_0 + \sigma_0 + C_{\epsilon_1,\epsilon_2}) \isd{r_2^2}\,.
\end{multline*}
To summarize, we have
\begin{multline*}
\isd{\left( |\nabla r|^2 +\mathsf A\,(1+r)^2 \right)} -\mathsf A\left( \isd{(1+r)^q} \right)^{2/q} \\\geq \theta\,\epsilon_0 \isd{\left( |\nabla r|^2 +\mathsf A\,r^2 \right)} + \sum_{k=1}^3 I_k\,.
\end{multline*}
In the following we will show that $I_1$, $I_3$ and $I_2$ are nonnegative, in this order.

%%%%%%%%%%%%%%%%%%%%%%%%%%%%%%%%%%%%%%%%%%%%%%%%%%
\subsubsection{Bound on \texorpdfstring{$I_1$}{I1}}\label{sec:i1}

The intuition here is the same as in the proof of the spectral gap inequality in Lemma~\ref{gap}. Namely, the lowest $\mathrm L^2$-eigenvalue of $\isd{|\nabla u|^2}$ on functions orthogonal to spherical harmonics of degree less or equal than $1$ is $2\,(d+1)$, while the term that we are subtracting corresponds to a component that is multiplied by a number only slightly larger than $d$. Therefore, there is space to accommodate the errors coming from $\epsilon_0$ and $\epsilon_1$. Another source of an error comes from the fact that, while $r$ is orthogonal to spherical harmonics of degree less or equal than $1$, $r_1$ need not be. However, as we will see, it nearly is. To control the corresponding error from orthogonality we need the positive terms involving $\sigma_0$.
%-------------------------------------------------
\begin{proposition}\label{Prop:estI1}
For any $0<\epsilon_0<\tfrac13$, there is a constant $\overline\sigma_0(\gamma,\epsilon_0,\tilde\delta)>0$ depending explicitly on~$\gamma$, $\epsilon_0$ and $\tilde\delta$ such that for all $d\geq 6$ and all $r\in \mathrm H^1(\Sph^d)$ such that $r\ge-1$ and satisfying~\eqref{eq:smallsphere} and~\eqref{eq:orthosphere} as in Theorem~\ref{unifboundclose}, with $\theta$ given by~\eqref{theta},
\be{epsilon1}
\epsilon_1=\tfrac12\,(1-3\,\epsilon_0)
\ee
and $\sigma_0\ge\overline\sigma_0(\gamma,\epsilon_0,\tilde\delta)$, one has
$$
I_1\geq 0\,.
$$
\end{proposition}
%-------------------------------------------------
Notice that $\theta=q-2\le1$ with $q=2\,d/(d-2)$ means $d\ge6$. An expression of $\overline\sigma_0$ is given below in~\eqref{sigma0}.
\begin{proof} We split the proof in three simple steps.

\medskip\noindent{\it Step 1.\/} Let $\tilde r_1$ be the orthogonal projection of $r_1$ onto the space of spherical harmonics of degree $\geq 2$, that is,
$$
\tilde r_1 = r_1 - \isd{r_1} - (d+1)\,\omega \cdot \isd{\omega'\,r_1(\omega')}(\omega')
$$
as $\sqrt{d+1}\,\omega_j$ is $\mathrm L^2$-normalized with respect to the uniform probability measure on the sphere for any $j=1,\,2,\,\ldots,d+1$. Then
\begin{multline*}
I_1  = (1-\theta\,\epsilon_0) \isd{|\nabla \tilde r_1|^2} - d\,(1 + \epsilon_0 + \epsilon_1 ) \isd{\tilde r_1^2} \\ + d\,\sigma_0 \isd{(r_2^2+r_3^2)}
- d\,(1 + \epsilon_0 + \epsilon_1 ) \left( \isd{r_1} \right)^2
\\ - d\,(d+1)\,\big( (1+\theta)\,\epsilon_0 + \epsilon_1 \big) \left| \isd{\omega\,r_1}\right|^2\\
 \geq \big( 2\,(d+1)\,(1-\theta\,\epsilon_0) - d\,(1 + \epsilon_0 + \epsilon_1 ) \big) \isd{\tilde r_1^2} \\ + d\,\sigma_0 \isd{(r_2^2+r_3^2)}
 - d\,(1 + \epsilon_0 + \epsilon_1 ) \left( \isd{r_1} \right)^2
\\ - d\,(d+1)\,\big( (1+\theta)\,\epsilon_0 + \epsilon_1 \big) \left| \isd{\omega\,r_1} \right|^2\,.
\end{multline*}
In the equality, we used the fact that the $\omega_j$'s are eigenfunctions of the Laplace--Beltrami operator with eigenvalue $d$. In the inequality, we used the fact that the operator is bounded from below by $2\,(d+1)$ on the orthogonal complement of spherical harmonics of degree less or equal than $1$.

\medskip\noindent{\it Step 2.\/} With $\epsilon_1$ given by~\eqref{epsilon1}, it is easy to see that for any $\epsilon_0<\tfrac13$, using $\theta\le1$, we have
\begin{equation}
\label{eq:i1boundmain}
2\,(d+1)\,(1-\theta\,\epsilon_0) - d\,(1 + \epsilon_0 + \epsilon_1 ) \ge \tfrac d2\,(1 - 3\,\epsilon_0)+2\,(1-\epsilon_0) > d\,\epsilon_1>0\,.
\end{equation}
Using
$$
\isd{\tilde r_1^2}=\isd{r_1^2}-\left( \isd{r_1} \right)^2 -(d+1)\, \left| \isd{\omega\,r_1} \right|^2
$$
and $\theta\le1$, we obtain
\begin{multline*}
\frac1d\,I_1\ge \;\epsilon_1\isd{\tilde r_1^2}+\sigma_0 \isd{(r_2^2+r_3^2)}\\
 -\,(1 + \epsilon_0 + \epsilon_1 ) \left( \isd{r_1} \right)^2
-(d+1)\,\big( (1+\theta)\,\epsilon_0 + \epsilon_1 \big) \left| \isd{\omega\,r_1} \right|^2\\
\ge \;\epsilon_1\isd{r_1^2}+\sigma_0 \isd{(r_2^2+r_3^2)}\\
-\,(1 + \epsilon_0) \left( \isd{r_1} \right)^2
-2\,(d+1)\,\epsilon_0\,\left| \isd{\omega\,r_1} \right|^2\,.
\end{multline*}

\medskip\noindent{\it Step 3.\/} Let us take care of the rank one terms coming from the orthogonality conditions. We will show that $I_1\ge0$ for an appropriately chosen $\sigma_0$ as a consequence of
\begin{multline}\label{eq:i1boundortho}
(1 + \epsilon_0) \left( \isd{r_1} \right)^2
+2\,(d+1)\,\epsilon_0\,\left| \isd{\omega\,r_1} \right|^2 \\ \leq\epsilon_1 \isd{r_1^2} + \sigma_0 \isd{(r_2^2+r_3^2)}\,.
\end{multline}
Let~$Y$ be one of the functions $1$ and $a\cdot\omega$, $a\in\R^{d+1}$. Then, since $\isd{Y\,r} =0$ by~\eqref{eq:orthosphere},
\begin{multline*}
\left( \isd{Y\,r_1} \right)^2 = \left( \isd{Y (r_2+r_3)} \right)^2 \\ \leq \nrmS Y4^2 \,\mu\big(\{r_2+r_3>0\}\big)^{1/2}\,\nrmS{r_2+r_3}2^2\,.
\end{multline*}
Since $\{r_2+r_3>0\}\subset\{r_1\geq\gamma\}$, we have
$$
\mu(\{r_2+r_3>0\}) \leq \mu(\{r_1\geq \gamma\}) \leq \frac1{\gamma^2} \isd{r_1^2}=\frac1{\gamma^2}\,\nrmS{r_1}2^2\,.
$$
Thus we have
\be{Y}
\left( \isd{Y\,r_1} \right)^2 \leq \nrmS Y4^2 \,\frac{\sqrt{2\,\tilde\delta}}\gamma\,\nrmS{r_1}2\(\isd{\(r_2^2+r_3^2\)}\)^{1/2}
\ee
using $\nrmS{r_2+r_3}2^2\le\sqrt{2\,\tilde\delta}\(\isd{\(r_2^2+r_3^2\)}\)^{1/2}$ because $\nrmS{r_2+r_3}2^2\le2\isd{\(r_2^2+r_3^2\)}$ and
$$
\nrmS{r_2+r_3}2\le\nrmS r2\le\nrmS rq\le\sqrt{\tilde\delta}\,.
$$
If $Y=1$, then clearly $\nrmS Y4=1$ and~\eqref{Y} gives
$$
\left( \isd{r_1} \right)^2 \leq \frac{\sqrt{2\,\tilde\delta}}\gamma\,\nrmS{r_1}2\(\isd{\(r_2^2+r_3^2\)}\)^{1/2}\,.
$$
If $Y=a\cdot\omega$, then a quick computation gives
$$
\nrmS Y4^4 = \frac{\int_0^\pi \cos^4\theta\,\sin^{d-1}\theta\,d\theta}{\int_0^\pi \sin^{d-1}\theta\,d\theta}\,|a|^4 = \frac{3\,|a|^4}{(d+3)\,(d+1)} \leq \frac{3\,|a|^4}{(d+1)^2}\,.
$$
{}From~\eqref{Y} applied with $a=\isd{\omega\,r_1}$, we obtain
\begin{multline*}
(d+1)\,\left| \isd{\omega\,r_1} \right|^2=\frac{d+1}{|a|^2}\left(\isd{Y\,r_1}\right)^2 \\ \le\sqrt3\,\frac{\sqrt{2\,\tilde\delta}}\gamma\,\nrmS{r_1}2\(\isd{\(r_2^2+r_3^2\)}\)^{1/2}\,.
\end{multline*}
Summing up, we have
\begin{multline*}
\epsilon_1\,\nrmS{r_1}2^2 + \sigma_0\,\isd{\(r_2^2+r_3^2\)}-(1+\epsilon_0)\left( \isd{r_1} \right)^2 \\ -2\,(d+1)\,\epsilon_0\,\left| \isd{\omega\,r_1} \right|^2\\
\ge\epsilon_1\,\nrmS{r_1}2^2 + \sigma_0\isd{\(r_2^2+r_3^2\)} \\-\big(1+(2\,\sqrt3+1)\,\epsilon_0\big)\,\frac{\sqrt{2\,\tilde\delta}}\gamma\,\nrmS{r_1}2\(\isd{\(r_2^2+r_3^2\)}\)^{1/2}
\end{multline*}
and the right-hand side~is nonnegative under a nonpositive discriminant condition which is satisfied by $\sigma_0\ge\overline\sigma_0(\gamma,\epsilon_0,\tilde\delta)$ with
\be{sigma0}
\overline\sigma_0(\gamma,\epsilon_0,\delta):=\frac1{2\,\epsilon_1}\,\big(1+(2\,\sqrt3+1)\,\epsilon_0\big)^2\,\frac{\delta}{\gamma^2}\,.
\ee
This choice establishes~\eqref{eq:i1boundortho} and allows us to conclude that $I_1\ge0$.
\end{proof}

Let us define
\be{delta1}
\delta_1:=\frac{4\,\epsilon_1\,\epsilon_2\,\gamma^2}{q\,\big(1+(2\,\sqrt3+1)\,\epsilon_0\big)^2}\,.
\ee
The condition $\sigma_0\ge\overline\sigma_0(\gamma,\epsilon_0,\tilde\delta)$ of Proposition~\ref{Prop:estI1} can be inverted as follows.
%-------------------------------------------------
\begin{corollary}\label{Cor:estI1}
For any $0<\epsilon_0<\tfrac13$ and $\sigma_0>0$, for all $d\geq 6$ and all $r\in \mathrm H^1(\Sph^d)$ such that $r\ge-1$ and satisfying~\eqref{eq:smallsphere} and~\eqref{eq:orthosphere} as in Theorem~\ref{unifboundclose}, with $\theta$, $\epsilon_1$, $\epsilon_2$ and $\delta_1$ respectively given by~\eqref{theta},~\eqref{epsilon1},~\eqref{epsilon1-epsilon2-C} and~\eqref{delta1}, if
$$
0<\tilde\delta\le\delta_1\,\frac{q\,\sigma_0}{2\,\epsilon_2}\,,
$$
then one has $I_1\geq 0$.
\end{corollary}
%-------------------------------------------------

%-------------------------------------------------
\begin{remark}
The assumption $\epsilon_0<\tfrac13$ is used in~\eqref{epsilon1} to guarantee that $\epsilon_1$ takes positive values. A less restrictive condition can be obtained by requesting that the left-hand side~in~\eqref{eq:i1boundmain} is actually $0$. We see that if $\epsilon_0<1$, then a similar bound as in~\eqref{eq:i1boundmain}, namely with $\frac12\,(1-\epsilon_0)$ on the right-hand side, holds for all sufficiently large $d$, depending on $\epsilon_0$.
\end{remark}
%-------------------------------------------------

%%%%%%%%%%%%%%%%%%%%%%%%%%%%%%%%%%%%%%%%%%%%%%%%%%
\subsubsection{Bound on \texorpdfstring{$I_3$}{I3}}

The idea for bounding this term is to use the Sobolev inequality. The extra coefficient $\frac2q<1$ gives us enough room to accommodate all error terms.
%-------------------------------------------------
\begin{proposition}\label{Prop:estI2} Assume that $\tilde\delta\in(0,1)$ and $0<\epsilon_0<\frac13$. With
\be{epsilon2}
\epsilon_2 := \frac14\,(1-3\,\epsilon_0)
\ee
and $\sigma_0=\frac2q\,\epsilon_2$, for all $d\geq 6$, all $\tilde\delta\leq1$ and all $r$ as in Theorem~\ref{unifboundclose}, one has
$$
I_3\geq 0\,.
$$
\end{proposition}
%-------------------------------------------------
\begin{proof} Taking into account the choice for $\sigma_0$, we have
\begin{multline*}
I_3=(1-\theta\,\epsilon_0) \isd{\left( |\nabla r_3|^2 +\mathsf A\,r_3^2 \right)} \\ - \tfrac 2q\,\mathsf A\((1+\epsilon_2\,\theta) \isd{r_3^q} + \epsilon_2\,\theta \isd{r_3^2}\).
\end{multline*}
We have $\nrmS{r_3}q^q\le\nrmS{r_3}q^2$ because $\nrmS{r_3}q\le \nrmS rq\le1$ and $\nrmS{r_3}2\le\nrmS{r_3}q$ by H\"older's inequality. Thus, we obtain
\begin{align*}
I_3 & \geq (1-\theta\,\epsilon_0) \isd{ \(|\nabla r_3|^2 +\mathsf A\,r_3^2\)} -\mathsf A\,\tfrac2q (1+2\,\epsilon_2\,\theta) \left( \isd{r_3^q} \right)^{2/q}\\
& \geq \frac\theta q\,\(1-q\,\epsilon_0-4\,\epsilon_2\) \isd{ \(|\nabla r_3|^2 +\mathsf A\,r_3^2\)}\ge0\,,
\end{align*}
using $\theta=q-2\le1$ and Sobolev's inequality: $\nrmS{\nabla r_3}2^2+\mathsf A\,\nrmS{r_3}2^2\ge\mathsf A\,\nrmS{r_3}q^2$.
\end{proof}

%-------------------------------------------------
\begin{remark}
The restriction $\epsilon_0<\tfrac13$ can be relaxed to $\epsilon_0<\tfrac12$ at the expense of having the inequality valid only in sufficiently high dimensions $d$, depending on $\epsilon_0$. Indeed, ignoring the influence of $\epsilon_2$ and $\sigma_0$ for the moment, the inequality at the end of the previous proof requires $1-\frac q2\,\epsilon_0>0$ and this is possible in all sufficiently high dimensions if and only if $\epsilon_0<\tfrac12$. Since this inequality is strict, the errors from $\epsilon_2$ and $\sigma_0$ can then be accommodated as well.
\end{remark}
%-------------------------------------------------

%%%%%%%%%%%%%%%%%%%%%%%%%%%%%%%%%%%%%%%%%%%%%%%%%%
\subsubsection{Bound on \texorpdfstring{$I_2$}{I2}}

At this point in the proof, for given $0<\epsilon_0<\tfrac13$, we have fixed the parameters $\epsilon_1$ and $\epsilon_2$ and we have found a $\delta_3$ such that $I_1$, $I_3\geq 0$ under the assumption $\tilde\delta\leq\delta_3$. Here we show that, by further decreasing $\tilde\delta$ if necessary, we can ensure that $I_3\geq 0$. The idea to achieve this is to use that $r_2$ satisfies an improved spectral gap inequality.
%-------------------------------------------------
\begin{proposition}\label{Prop:estI3}
For any $0<\epsilon_0<\frac13$, let $\sigma_0=\frac2q\,\epsilon_2$. Then there is a $\delta_2\in(0,1)$ such that, for all $d\geq 6$, all $\tilde\delta\leq\delta_2$ and all~$r$ as in Theorem~\ref{unifboundclose}, one has
$$
I_2\geq 0\,.
$$
\end{proposition}
%-------------------------------------------------
\begin{proof}
We first claim that for any $\mathrm L^2$-normalized spherical harmonic $Y$ of degree $k\in\N$, we have
\begin{equation}
\label{eq:projectionr2}
\left| \isd{ Y\,r_2} \right| \leq 3^\frac k2\,\gamma^{-\frac q4}\,\tilde\delta^{\frac q8}\,\nrmS{r_2}2\,.
\end{equation}
Indeed, according to~\cite[Theorem~1]{MR908654}, for any such spherical harmonic and any $p\in[2,\infty)$ we have
$$
\nrmS Yp \leq (p-1)^\frac k2\,.
$$
Thus, we can bound
\begin{multline*}
\left| \isd{ Y\,r_2} \right| \leq \nrmS Y4\,\mu\big(\{ r_2 >0\}\big)^\frac14\,\nrmS{r_2}2 \\ \leq 3^\frac k2\,\mu\big(\{ r_2 >0\}\big)^\frac14\,\nrmS{r_2}2\,.
\end{multline*}
Meanwhile,
$$
\mu\big(\{ r_2 >0\}\big) = \mu\big(\{r>\gamma\}\big) \leq \frac1{\gamma^q}\,\nrmS rq^q \leq \frac{\tilde\delta^{q/2}}{\gamma^q}\,.
$$
This leads to the claimed bound~\eqref{eq:projectionr2}.

If $\pi_k\,r_2$ denotes the projection of $r_2$ onto spherical harmonics of degree~$k$, from~\eqref{eq:projectionr2} to $Y=\pi_k\,r_2/\nrmS{\pi_k\,r_2}2$, it follows that
$$
\nrmS{ \Pi_k\,r_2}2 \leq 3^\frac k2\,\gamma^{-\frac q4}\,\tilde\delta^\frac q8\,\nrmS{r_2}2\,.
$$
Next, for any $K\in\N$, if $\Pi_K\,r_2 := \sum_{k<K} \pi_k\,r_2$ denotes the projection of $r_2$ onto spherical harmonics of degree less than $K$, then
\begin{multline*}
\nrmS{ \Pi_K\,r_2}2 = \Big( {\scriptstyle\sum_{k<K}}\nrmS{\pi_k\,r_2}2^2 \Big)^{\!1/2} \\ \leq \gamma^{-\frac q4}\,\tilde\delta^\frac q8\nrmS{r_2}2\,\sqrt{{\scriptstyle\sum_{k<K}} 3^k} \leq 3^\frac K2\,\gamma^{-\frac q4}\,\tilde\delta^\frac q8\,\nrmS{r_2}2\,.
\end{multline*}
{}From this we conclude that
\begin{align*}
\isd{ |\nabla r_2|^2} \geq& \isd{ |\nabla (1-\Pi_K)\,r_2|^2}\\
\geq&\; K\,(K+d-1) \isd{ |(1-\Pi_K)\,r_2|^2}\\
&= K\,(K+d-1) \left( \nrmS{r_2}2^2 - \nrmS{ \Pi_K\,r_2}2^2 \right)\\
&\quad\geq K\,(K+d-1) \left( 1- 3^K\,\gamma^{-\frac q2}\,\tilde\delta^\frac q4 \right) \nrmS{r_2}2^2\,.
\end{align*}
Consequently,
\begin{multline*}
I_2 \geq \left( (1-\theta\,\epsilon_0)\,K\,(K+d-1) \left( 1- 3^K\,\gamma^{-\frac q2}\,\tilde\delta^\frac q4 \right)\right)\nrmS{r_2}2^2\\  - d \left( 1+ \epsilon_0 + \sigma_0 + C_{\epsilon_1,\epsilon_2} \right)  \nrmS{r_2}2^2\,.
\end{multline*}
We choose $K\in\N$ and $\delta_2>0$ such that
\be{K-delta2}
K:=1+\left[2\,\frac{1+ \epsilon_0 + \sigma_0 + C_{\epsilon_1,\epsilon_2}}{1-\epsilon_0}\right]\quad\mbox{and}\quad\delta_2:=\frac14\,\frac{\gamma^2}{3^{2K}}
\ee
where $[x]$ denotes the integer part of $x\in\R$ and $\delta_3$ is given by~\eqref{epsilon2}. From the definition of $\delta_2$, if $\tilde\delta\le\delta_2$, we have $1- 3^K\,\gamma^{-\frac q2}\,\tilde\delta^\frac q4\ge\frac12$ and conclude that $I_2\ge0$ because $K+d-1\ge d$.
\end{proof}

%%%%%%%%%%%%%%%%%%%%%%%%%%%%%%%%%%%%%%%%%%%%%%%%%%
\subsection{Proof of Theorem~\ref{unifboundclose}}

We assume that $d\ge6$ and fix some $\epsilon_0\in(0,1/3)$. With the choice
$$
\gamma=\epsilon_2=2\,\epsilon_1=\tfrac14\,(1-3\,\epsilon_0)\quad\mbox{and}\quad\sigma_0=\frac2q\,\epsilon_2
$$
according to~\eqref{epsilon1-epsilon2-C},~\eqref{epsilon1}, and~\eqref{epsilon2} on the one hand so that the assumptions of Corollary~\ref{Cor:estI1}, Proposition~\ref{Prop:estI2} and Proposition~\ref{Prop:estI3} are fulfilled, and an arbitrary choice of
$$
M\ge2\,\gamma\,,\quad\overline M\ge\sqrt e\quad\mbox{and}\quad\epsilon=\gamma
$$
which determines $C_{\epsilon_1,\epsilon_2}=C_{\gamma,\epsilon,M}$ according to~\eqref{epsilon1-epsilon2-C} on the other hand, and with the condition
$$
\tilde\delta=\min\big\{\delta_1,\delta_2\big\}
$$
with $\delta_1$ and $\delta_2$ given by~\eqref{delta1} and~\eqref{K-delta2}, we claim that $I_1$, $I_2$ and $I_3$ are nonnegative, which completes the proof of Theorem~\ref{unifboundclose} for $q\leq 3$, that is $d\geq 6$. The assertion for $d=3$, $4$, $5$ follows from the result proved in Subsection~\ref{sec:firstbound}.\qed

%%%%%%%%%%%%%%%%%%%%%%%%%%%%%%%%%%%%%%%%%%%%%%%%%%
%%%%%%%%%%%%%%%%%%%%%%%%%%%%%%%%%%%%%%%%%%%%%%%%%%
\section{From a local to a global stability result}\label{sec:3}

We work with nonnegative functions in Section~\ref{sec:flow} and extend the method to sign-changing functions in Section~\ref{sec:posneg}. Our goal is to prove Theorem~\ref{main}: see Section~\ref{sec:proof}.

%%%%%%%%%%%%%%%%%%%%%%%%%%%%%%%%%%%%%%%%%%%%%%%%%%
\subsection{Nonnegative functions away from the manifold of optimizers}\label{sec:flow}

Here we prove a stability inequality for nonnegative functions that are `far' away from the manifold of optimizers. With $\mathcal E$ defined by~\eqref{Def:E}, let us introduce
\begin{equation}
\label{eq:mudelta2}
\Imu(\delta)\! :=\! \inf\!\left\{ \mathcal E(f) : 0\leq f\in\dot{\mathrm H}^1(\R^d)\!\setminus\!\mathcal M\,,\inf_{\bb\in\mathcal M} \Vert \nabla f - \nabla\bb\Vert^2_2 \leq \delta\,\Vert \nabla f \Vert^2_2 \right\}.
\end{equation}
%-------------------------------------------------
\begin{theorem}\label{Cor:Summary}
Let $\delta\in(0,1/2)$ and assume that $0\leq f\in\dot{\mathrm H}^1(\R^d)\setminus\mathcal M$ satisfies
$$
\inf_{\bb\in\mathcal M} \Vert \nabla f - \nabla\bb\Vert_2^2 \ge \delta\,\Vert \nabla f \Vert_2^2\,.
$$
Then, with $\Imu(\delta)$ defined by~\eqref{eq:mudelta2}, we have
$$
\mathcal E(f) \ge \delta\,\Imu(\delta)\,.
$$
\end{theorem}
%-------------------------------------------------
We will prove this theorem by symmetrization. First, we will use a discrete symmetrization procedure to get somewhat close to the manifold, then we will use a further continuous symmetrization procedure to fine tune the distance to the manifold.

%%%%%%%%%%%%%%%%%%%%%%%%%%%%%%%%%%%%%%%%%%%%%%%%%%
\subsubsection{Competing symmetries}\label{sec:competing}

The functional $\mathcal E(f)$ is conformally invariant
in the sense that if $C:\R^d\cup\{\infty\} \to \R^d\cup\{\infty\}$ is a conformal map, the function
$$
f_C(x) = |{\rm det}\,DC(x)|^{1/2^*}f\big(C( x)\big)
$$
satisfies
$$
\mathcal E(f_C) = \mathcal E(f)\,.
$$
In order to verify this, we recall that any conformal map is a composition of scalings, translations, rotations and inversions. For scalings, translations and rotations in $\R^d$ the claimed invariance is easy to see. The additional map to consider is the inversion $I(x)= \frac{x}{|x|^2}$
and a straightforward change of variables shows that
$$
\Vert \nabla f_I \Vert_2^2 = \Vert \nabla f\Vert_2^2\,,
\quad \Vert f_I \Vert_{2^*}^2 = \Vert f \Vert_{2^*}^2\,.
$$
The equality
$$
\inf_{\bb\in\mathcal M} \Vert \nabla(f_I-\bb) \Vert_2^2 = \inf_{\bb\in\mathcal M} \Vert \nabla f-\nabla\bb \Vert_2^2
$$
follows from
$$
\inf_{\bb\in\mathcal M} \Vert \nabla(f_I-\bb) \Vert_2^2=\inf_{\bb\in\mathcal M} \Vert \nabla(f-\bb_I) \Vert_2^2 = \inf_{\bb\in\mathcal M} \Vert \nabla f-\nabla\bb \Vert_2^2
$$
since $I^2=I$ and $\bb \to \bb_I$ maps the set $\mathcal M$ to itself in a one-to-one and onto fashion.

Another and perhaps easier way to see the conformal invariance is to pull the problem up to the sphere via the stereographic projection, as discussed in Section~\ref{sec:preliminary}. On the sphere the inversion~$I$ takes the form of the reflection $(s_1, \dots, s_d, s_{d+1}) \to (s_1, \dots, s_d, -s_{d+1})$, which clearly leaves the functional on the sphere unchanged.

A second ingredient for the construction of the discrete symmetrization flow is the technique of `competing symmetries', invented in~\cite{CarlenLoss}. Consider any nonnegative function $f\in\dot{\mathrm H}^1(\R^d)$ and its counterpart $F \in \mathrm H^1(\Sp^d)$ given by~\eqref{eq:euclidtosphere}. Set
$$
(UF)(\omega) = F(\omega_1,\omega_2, \dots, \omega_{d+1}, -\omega_d)\,,
$$
which corresponds to a rotation by $\pi/2$ that maps the `north pole' axis $(0, 0, \dots, 1)$ to $(0,\dots,1,0)$.
Reversing~\eqref{eq:euclidtosphere} the function on $\R^d$ that corresponds to $UF$ is given by
\begin{equation}\label{eq:U}
(Uf)(x) = \left(\frac2{|x-e_d|^2}\right)^{\frac{d-2}2} f\left(\frac{x_1}{|x-e_d|^2}, \dots, \frac{x_{d-1}}{|x-e_d|^2}, \frac{|x|^2-1}{|x-e_d|^2}\right),
\end{equation}
where $e_d=(0,\dots, 0, 1)\in\R^d$. It follows that
$$
\mathcal E(Uf) = \mathcal E(f)\,.
$$
The operation $U$ is obviously linear, invertible and an isometry on $\mathrm L^{2^*}(\R^d)$.

We also consider the symmetric decreasing rearrangement
$$
\mathcal R f(x) = f^*(x)\,.
$$
The most important properties are that $f$ and $f^*$ are equimeasurable
and that $\Vert \nabla f^* \Vert_2 \le \Vert \nabla f\Vert_2$. For elementary properties of rearrangements the reader may consult~\cite{LiebLoss}.
Being equimeasurable, this map is also an isometry on $\mathrm L^{2^*}(\R^d)$. It is when using the decreasing rearrangement that
we use the fact that $f$ is a nonnegative function. For functions that change sign one conventionally defines their rearrangement as the rearrangement of their absolute value. Passing from a function to its absolute value does not alter the numerator of $\mathcal E(f)$ but may decrease the denominator so that other arguments are needed.

On $\R^d$, let
\begin{equation}
\label{gstar}
\bb_*(x):=|\Sp^d|^{-\frac{d-2}{2\,d}} \left( \frac2{1+|x|^2}\right)^\frac{d-2}2\,.
\end{equation}
Note that $\Vert \bb_*\Vert_{2^*} = 1$ because it is obtained as the stereographic projection of the constant function on $\Sp^d$ with $2^*$-norm equal to $1$. The following theorem was proved in~\cite{CarlenLoss}.
%-------------------------------------------------
\begin{theorem}\label{thm:competingsymmetries}
Let $f \in \mathrm L^{2^*}(\R^d)$ be a nonnegative function. Consider the sequence $(f_n)_{n\in\N}$ of functions
\begin{equation}\label{fn}
f_n =(\mathcal RU)^n f\quad\forall\,n\in\N\,.
\end{equation}
Then
$$
\lim_{n \to \infty} \Vert f_n - h_f\Vert_{2^*} = 0
$$
where $h_f=\Vert f \Vert_{2^*}\,\bb_*\in \mathcal M$. Moreover, if $f\in\dot{\mathrm H}^1(\R^d)$, then $(\Vert \nabla f_n \Vert_2^2)_{n\in\N}$ is a nonincreasing sequence.
\end{theorem}
%-------------------------------------------------
It does not seem clear whether the functional $\mathcal E(f)$ decreases or increases under
rearrangement. The next lemma helps to explain this point. Define $\mathcal M_1$ to be the set of the elements in
$\mathcal M$ with $2^*$-norm equal to $1$.
%-------------------------------------------------
\begin{lemma}\label{innerproduct}
For any $f\in\dot{\mathrm H}^1(\R^d)$, we have
$$
{\rm dist}(f,\mathcal M)^2=\inf_{\bb\in\mathcal M} \Vert \nabla f - \nabla\bb\Vert^2_2 = \Vert \nabla f \Vert_2^2 - S_d\,\sup_{\bb\in\mathcal M_1} \left(f, \bb^{2^*-1}\right)^2.
$$
\end{lemma}
%-------------------------------------------------
Here and in the sequel, $(\cdot,\cdot)$ is the $\mathrm L^2(\R^d)$ inner product or, more precisely, the duality pairing between $\mathrm L^{2^*}(\R^d)$ and~$\mathrm L^{(2^*)'}(\R^d)$.

\begin{proof}
Let $\bb$ be any Aubin--Talenti function. The function $\bb$ is an optimizer of the Sobolev inequality, {\it i.e.}, $\Vert \nabla\bb \Vert_2^2 = S_d\,\Vert \bb \Vert_{2^*}^2=S_d$ and is a solution of the Sobolev equation
\be{Sobolev:equation}
-\Delta \bb = S_d\,\frac{\bb^{2^*-1}}{\Vert \bb \Vert_{2^*}^{2^*-2}} = S_d\,\bb^{2^*-1}\,.
\ee
Hence for any nonnegative constant $c$, if $\Vert \bb \Vert_{2^*}=1$, we find
\begin{multline*}
\Vert \nabla (f-c\,\bb)\Vert_2^2 = \Vert \nabla f \Vert_2^2 - 2\,c\,(\nabla f,\nabla\bb) + c^2\,\Vert \nabla\bb \Vert_2 \\ = \Vert \nabla f \Vert_2^2 - 2\,c\,S_d\,\left(f, \bb^{2^*-1}\right) + S_d\,c^2
\end{multline*}
and minimizing with respect to $c$ we find the lower bound $ \Vert \nabla f \Vert_2^2 - S_d\,\left(f,\bb^{2^*-1}\right)^2$, which proves the lemma.
\end{proof}

In the above proof, notice that the optimal value of $c$ is such that $c=\left(f,\bb^{2^*-1}\right)\le\nrm f{2^*}$ and elementary considerations on $\mathcal M_1$ show that $c>0$ if $f\neq0$ is nonnegative, so that ${\rm dist}(f,\mathcal M)^2<\Vert \nabla f \Vert_2^2$. With the notation of~\eqref{eq:optimizers} and $g_{a,b}(x):=\big(a\,|\Sph^d|^{1/d}\big)^{-(d-2)/2}\,\bar g\big((x-b)/a\big)$, we find that $\big(f,g_{a,b}^{2^*-1}\big)$ converges to $0$ as $a+a^{-1}+|b|\to+\infty$. A sequence $(a_n,b_n,c_n)$ such that
$$
\lim_{n\to+\infty}\nrm{\nabla(f-c_n\,g_{a_n,b_n})}2^2={\rm dist}(f,\mathcal M)^2
$$
is therefore relatively compact in $(0,+\infty)\times\R^d\times\R$, which proves that ${\rm dist}(f,\mathcal M)$ is attained at some Aubin--Talenti function of the form~\eqref{eq:optimizers}.

We note that, under the decreasing rearrangement, the term $\Vert \nabla f\Vert_2^2$ does not increase whereas the term $\sup_{\bb\in\mathcal M_1} \left(f, \bb^{2^*-1}\right)^2$ increases.
To see this, note that ${\rm dist}(f,\mathcal M)$ is attained at some Aubin--Talenti function which is a strictly symmetric decreasing
function about some point $b\in\R^d$. Replacing $f$ by its symmetric decreasing rearrangement about that point increases $\big(f, \bb^{2^*-1}\big)^2$, in
fact strictly unless $f$ is already symmetric decreasing about the point~$b$. Thus, while the numerator in $\mathcal E(f)$ decreases under rearrangements, so does the denominator and there are no direct conclusions to be drawn from this. The next lemma summarizes what we have shown.
%-------------------------------------------------
\begin{lemma}\label{lm:monotone}
For the sequence $(f_n)_{n\in\N}$ in Theorem~\ref{thm:competingsymmetries} we have that $n\mapsto\sup_{\bb\in\mathcal M_1} \left(f_n, \bb^{2^*-1}\right)^2$ is strictly increasing, $n\mapsto\inf_{\bb\in\mathcal M} \Vert \nabla f_n - \nabla\bb\Vert_{2^*}^2$ is strictly decreasing and, with $h_f=\Vert f \Vert_{2^*}\,\bb_*$ as in Theorem~\ref{thm:competingsymmetries},
\begin{multline*}
\lim_{n \to \infty} \inf_{\bb\in\mathcal M} \Vert \nabla f_n - \nabla\bb\Vert_2^2 = \lim_{n \to \infty} \Vert \nabla f_n \Vert^2_2 - S_d\,\Vert h_f \Vert_{2^*} ^2\\ = \lim_{n \to \infty} \Vert \nabla f_n \Vert^2_2 - S_d\,\Vert f \Vert_{2^*} ^2\,.
\end{multline*}
\end{lemma}
%-------------------------------------------------
\begin{proof}
{}From
$$
\inf_{\bb\in\mathcal M} \Vert \nabla f_n - \nabla\bb\Vert_2^2 = \Vert \nabla f_n \Vert^2_2 -
S_d\,\sup_{\bb\in\mathcal M_1} \left(f_n, \bb^{2^*-1}\right)^2
$$
we see that the first term converges since $(\Vert \nabla f_n \Vert^2_2)_{n\in\N}$ is a nonincreasing sequence.
For the second term, which is strictly increasing, we have by H\"older's inequality
$$
\sup_{\bb\in\mathcal M_1} \left(f_n, \bb^{2^*-1}\right)^2 \le \Vert f_n \Vert_{2^*}^2 = \Vert f \Vert_{2^*}^2
$$
and since $\bb_*$ as defined in~\eqref{gstar} is in $\mathcal M_1$ we have
$$
\liminf_{n\to \infty} \sup_{\bb\in\mathcal M_1} \left(f_n, \bb^{2^*-1}\right)^2 \ge \liminf_{n\to \infty} \left(f_n,\bb_*^{2^*-1}\right)^2 = \Vert f \Vert_{2^*}^2
$$
by Theorem~\ref{thm:competingsymmetries}.
\end{proof}

%-------------------------------------------------
\begin{lemma} \label{alternatives}
Assume that $0\leq f\in\dot{\mathrm H}^1(\R^d)\setminus\mathcal M$ satisfies
$$
\inf_{\bb\in\mathcal M} \Vert \nabla f - \nabla\bb\Vert_2^2 \ge \delta\,\Vert \nabla f \Vert_2^2
$$
and let $(f_n)_{n\in\N}$ be the sequence defined by~\eqref{fn}.
Then one of the following alternatives holds:
\begin{enumerate}
\item[(a)] for all $n=0,1,2 \dots$ we have
$$
\inf_{\bb\in\mathcal M} \Vert \nabla f_n - \nabla\bb\Vert_2^2 \ge \delta\,\Vert \nabla f_n \Vert_2^2\,,
$$
\item[(b)] there is a natural number $n_0$ such that
$$
\inf_{\bb\in\mathcal M} \Vert \nabla f_{n_0} - \nabla\bb\Vert_2^2 \ge \delta\,\Vert \nabla f_{n_0} \Vert_2^2
$$
and
$$
\inf_{\bb\in\mathcal M} \Vert \nabla f_{n_0+1} - \nabla\bb\Vert_2^2 < \delta\,\Vert \nabla f_{n_0+1} \Vert_2^2\,.
$$
\end{enumerate}
\end{lemma}
%-------------------------------------------------
\begin{proof}
Assume that alternative (a) does not hold. Then there is a largest value $n_0\ge0$ such that
$\inf_{\bb\in\mathcal M} \Vert \nabla f_{n_0} - \nabla\bb\Vert_2^2 \ge \delta\,\Vert \nabla f_{n_0} \Vert_2^2$.
\end{proof}
%-------------------------------------------------
\begin{lemma}\label{alternativea}
Assume that $0\leq f\in\dot{\mathrm H}^1(\R^d)\setminus \mathcal M$ satisfies
$$
\inf_{\bb\in\mathcal M} \Vert \nabla f - \nabla\bb\Vert^2_2 \ge \delta\,\Vert \nabla f \Vert_2^2
$$
and suppose that in Lemma~\ref{alternatives} alternative {\rm (a)} holds for the sequence $(f_n)_{n\in\N}$ defined by~\eqref{fn}. Then
$$
\mathcal E(f) \ge \delta\,.
$$
\end{lemma}
%-------------------------------------------------
\begin{proof}
We have
\begin{multline}\label{Efn}
\mathcal E(f) = \frac{\Vert \nabla f \Vert^2_2 - S_d\,\Vert f \Vert_{2^*}^2} {\inf_{\bb\in\mathcal M} \Vert \nabla f - \nabla\bb\Vert^2_2 } \ge \frac{\Vert \nabla f \Vert^2_2 - S_d\,\Vert f \Vert_{2^*}^2} {\Vert \nabla f \Vert^2_2 } \\ \geq \frac{\Vert \nabla f_n \Vert^2_2 - S_d\,\Vert f \Vert_{2^*}^2} {\Vert \nabla f _n\Vert^2_2}\,,
\end{multline}
where the second inequality is a consequence of $\Vert \nabla f _n\Vert^2_2\leq\Vert \nabla f\Vert^2_2$ for all $n=0$, $1$, $2$,\dots proved in Theorem~\ref{thm:competingsymmetries}. By the assumption that alternative {\rm (a)} holds and by Lemma~\ref{lm:monotone}, we learn that
\begin{multline*}
\lim_{n \to \infty}\Vert \nabla f _n\Vert^2_2\leq\frac1\delta\,\lim_{n \to \infty}\,\inf_{\bb\in\mathcal M} \Vert \nabla f_n - \nabla\bb\Vert_2^2 \\ = \frac1\delta\left(\lim_{n \to \infty} \Vert \nabla f_n \Vert^2_2 - S_d\,\Vert f \Vert_{2^*} ^2\right)\,.
\end{multline*}
Since
\begin{multline*}
\lim_{n \to \infty} \Vert \nabla f_n \Vert_2^2 - S_d\,\Vert f \Vert_{2^*}^2 \ge \delta\,\lim_{n \to \infty} \Vert \nabla f_n \Vert_2^2 \\ \ge \delta\,S_d\,\lim_{n \to \infty} \Vert f_n\Vert_{2^*}^2= \delta\,S_d\,\Vert f \Vert_{2^*}^2 >0\,,
\end{multline*}
we can take the limit as $n\to\infty$ on the right side of~\eqref{Efn} and compute the limit of the quotient as the quotient of the limits. This proves the lemma.
\end{proof}

%%%%%%%%%%%%%%%%%%%%%%%%%%%%%%%%%%%%%%%%%%%%%%%%%%
\subsubsection{Continuous rearrangement}

Next, we analyze the case where the alternative (b) in Lemma~\ref{alternatives} holds. We recall that $\Imu(\delta)$ was defined in~\eqref{eq:mudelta2}.
%-------------------------------------------------
\begin{lemma}\label{mu1}
For any $\delta\in (0,1]$, we have $\Imu(\delta)\leq 1$.
\end{lemma}
%-------------------------------------------------
\begin{proof}
By Lemma~\ref{innerproduct}, we have
$$
\inf_{\bb\in\mathcal M} \|\nabla f-\nabla\bb\|_2^2 = \Vert \nabla f \Vert^2_2 - S_d\,\sup_{\bb\in\mathcal M_1} \left(f,\bb^{2^*-1}\right)^2
$$
and it follows from H\"older's inequality that
$$
\sup_{\bb\in\mathcal M_1} \left(f,\bb^{2^*-1}\right)^2 \leq \|f\|_{2^*}^2\,.
$$
Thus, the denominator in $\mathcal E(f)$ that enters the definition of $\Imu(\delta)$ is at least as large as the numerator, so the quotient is at most 1.
\end{proof}

Our goal in this subsection is to prove the following lower bound on $\mathcal E(f)$.
%-------------------------------------------------
\begin{lemma}\label{alternativeb}
Assume that $0\leq f\in\dot{\mathrm H}^1(\R^d)\setminus \mathcal M$ satisfies
$$
\inf_{\bb\in\mathcal M} \Vert \nabla f - \nabla\bb\Vert_2^2 \ge \delta\,\Vert \nabla f \Vert_2^2
$$
for some $\delta\in(0,1/2)$ and suppose that in Lemma~\ref{alternatives} alternative {\rm (b)} holds for the sequence $(f_n)_{n\in\N}$ of Theorem~\ref{thm:competingsymmetries} defined by~\eqref{fn}. Then, with $\Imu(\delta)$ defined by~\eqref{eq:mudelta2}, we have
$$
\mathcal E(f)\ge \delta\,\Imu(\delta)\,.
$$
\end{lemma}
%-------------------------------------------------
For the proof of this lemma we introduce a continuous rearrangement flow that interpolates between a function and its symmetric decreasing rearrangement. The basic ingredient for this flow is similar to a flow that Brock introduced~\cite{Brock,Brock2} and that interpolates between a function and its Steiner symmetrization with respect to a given hyperplane. Brock's construction, in turn, is based on ideas of Rogers~\cite{Rogers} and Brascamp--Lieb--Luttinger~\cite{BrascampLiebLuttinger}. Our flow is obtained by glueing together infinitely many copies of Brock's flows with respect to a sequence of judiciously chosen hyperplanes. A similar construction was performed by Bucur and Henrot~\cite{BucurHenrot}; see also~\cite{Christ}.

More specifically, for a given hyperplane $H$, Brock's flow interpolates between a given function $f$ and $f^{*H}$, the Steiner symmetrized function with respect
to $H$. The family that interpolates between $f$ and $f^{*H}$ is denoted by $f^H_\tau, \tau \in [0,\infty]$, and we have
$$
f_0=f\,,\quad f^H_\infty = f^{*H}\,.
$$
Further, for any $\tau$, $f^H_\tau$ and $f$ are equimeasurable, {\it i.e.},
$$
\left|\left\{x \in \R^d: f^H_\tau(x) >t \right\}\right| = \left|\left\{x \in \R^d: f(x) >t \right\}\right|
\quad\forall\,t>0\,.
$$
Moreover, if $f\in \mathrm L^p(\R^d)$ for some $1\leq p<\infty$, then $\tau\mapsto f^H_\tau$ is continuous in $\mathrm L^p(\R^d)$.

By choosing a sequence of hyperplanes we construct another flow $\tau \mapsto f_\tau$ that has the same properties but interpolates between $f$ and $f^*$, the symmetric decreasing rearrangement. In Appendix~\ref{contrearr} we explain this in more detail and prove the following properties that are important for our proof, assuming $f\in\dot{\mathrm H}^1(\R^d)$. From the $\mathrm L^{2^*}(\R^d)$ continuity of the flow we will deduce that
\begin{equation}
\label{eq:flowprop2}
\lim_{\tau \to \tau_0} \sup_{\bb\in\mathcal M_1}\left(f_\tau,\bb\right)^2 = \sup_{\bb\in\mathcal M_1}\left(f_{\tau_0},\bb\right)^2\,.
\end{equation}
Concerning the gradient we prove the monotonicity
\begin{equation*}
\label{eq:flowprop1}
\Vert \nabla f_{\tau_2} \Vert_2 \le \Vert \nabla f_{\tau_1} \Vert_2\,,\quad 0\leq\tau_1\leq\tau_2\leq\infty\,,
\end{equation*}
and the right continuity
\begin{equation}
\label{eq:flowprop3}
\lim_{\tau_2\to\tau_1^+} \Vert \nabla f_{\tau_2} \Vert_2 = \Vert \nabla f_{\tau_1} \Vert_2\,,\quad 0\leq\tau_1<\infty\,.
\end{equation}

\begin{proof}[Proof of Lemma~\ref{alternativeb}]
We begin by motivating and explaining the strategy of the proof. As before, we bound
\begin{multline}
\label{eq:flowargument}
\mathcal E(f) = \frac{\Vert \nabla f \Vert^2_2 - S_d\,\Vert f \Vert_{2^*}^2} {\inf_{\bb\in\mathcal M} \Vert \nabla f - \nabla\bb\Vert^2_2 } \ge \frac{\Vert \nabla f \Vert^2_2 - S_d\,\Vert f \Vert_{2^*}^2} {\Vert \nabla f \Vert^2_2 }\\  \ge \frac{\Vert \nabla f_{n_0} \Vert^2_2 - S_d\,\Vert f_{n_0} \Vert_{2^*}^2} {\Vert \nabla f_{n_0} \Vert^2_2 }\,.
\end{multline}
We could bound the right side further from below by replacing $f_{n_0}$ by $f_{n_0+1}$. This bound, however, might be too crude for our purposes and we proceed differently. The move from $f_{n_0}$ to $f_{n_0+1}$ consists of two steps, namely first applying a conformal rotation and second applying symmetric decreasing rearrangement. The first step leaves all terms on the right side invariant and we do carry out this step. The second step leaves the $2^*$-norm invariant, while the gradient term does not go up. In fact, the gradient term might go down too far. Therefore, we replace the application of the rearrangement by a continuous rearrangement flow. We denote by $\mathsf f_\tau$, $n_0\leq\tau<n_0+1$, the continuous rearrangement starting at $\mathsf f_{n_0}:=Uf_{n_0}$, where $U$ denotes the conformal rotation~\eqref{eq:U}. The `time' variable $\tau$ has been reparametrized so that at $\tau=n_0+1$ we have arrived at the symmetric decreasing rearrangement of $\mathsf f_{n_0}$, that is,
\begin{equation}
\label{eq:finfty}
\mathsf f_{n_0+1}=(\mathsf f_{n_0})^* = f_{n_0+1}\,.
\end{equation}
Ideally, we would like to find $\tau_0\in[n_0,n_0+1)$ such that
$$
\inf_{\bb\in\mathcal M} \Vert \nabla \mathsf f_{\tau_0} - \nabla\bb\Vert^2_2 = \delta\,\Vert \nabla \mathsf f_{\tau_0} \Vert^2_2\,.
$$
Then the right side of~\eqref{eq:flowargument} is equal to
$$
1 - S_d\,\frac{\Vert \mathsf f_{n_0} \Vert_{2^*}^2} {\Vert \nabla \mathsf f_{n_0} \Vert^2_2 }
\ge 1 - S_d\,\frac{\Vert \mathsf f_{\tau_0} \Vert_{2^*}^2} {\Vert \nabla \mathsf f_{\tau_0} \Vert^2_2 } = \delta\,\frac{\Vert \nabla \mathsf f_{\tau_0} \Vert^2_2 - S_d\,\Vert \mathsf f_{\tau_0} \Vert_{2^*}^2}{\inf_{\bb\in\mathcal M} \Vert \nabla \mathsf f_{\tau_0} - \nabla\bb\Vert^2_2}\,,
$$
which can be bounded from below by $\delta\,\Imu(\delta)$, since $\mathsf f_{\tau_0}$ is admissible in the infimum~\eqref{eq:mudelta2}. This would prove the desired bound.

The problem with this argument is that the existence of such a $\tau_0\in[n_0,n_0+1)$ is in general not clear, since neither of the terms $\inf_{\bb\in\mathcal M} \Vert \nabla \mathsf f_{\tau} - \nabla\bb\Vert^2_2$ and $\Vert \nabla \mathsf f_{\tau} \Vert^2_2$ needs to be continuous in~$\tau$. Nevertheless, we will be able to adapt the above argument to yield the same conclusion.

We now turn to the details of the argument. Recalling that
$$
\inf_{\bb\in\mathcal M} \|\nabla \mathsf f_0 - \nabla\bb\|_2^2 \geq \delta\,\|\nabla \mathsf f_0\|_2^2\,,
$$
we define
$$
\tau_0 := \inf\left\{ \tau\in(n_0,n_0+1)\,:\,\inf_{\bb\in\mathcal M} \|\nabla \mathsf f_\tau - \nabla\bb\|_2^2 < \delta\,\|\nabla \mathsf f_\tau \|_2^2 \right\}
$$
with the convention that $\inf\emptyset = n_0+1$. If $\tau<\tau_0\in(n_0,n_0+1]$, similarly as before, the right side of~\eqref{eq:flowargument} is equal to
\begin{multline*}
\frac{\Vert \nabla \mathsf f_{n_0} \Vert^2_2-S_d\,\Vert \mathsf f_{n_0} \Vert_{2^*}^2} {\Vert \nabla \mathsf f_{n_0} \Vert^2_2 } =
1 - S_d\,\frac{\Vert \mathsf f_{n_0} \Vert_{2^*}^2} {\Vert \nabla \mathsf f_{n_0} \Vert^2_2 } \ge \frac{\Vert \nabla \mathsf f_{\tau} \Vert^2_2 - S_d\,\Vert \mathsf f_{\tau_0} \Vert_{2^*}^2}{\Vert \nabla \mathsf f_{\tau} \Vert^2_2}\\ \ge\delta\,\frac{\Vert \nabla \mathsf f_{\tau} \Vert^2_2 - S_d\,\Vert \mathsf f_{\tau_0} \Vert_{2^*}^2}{\inf_{\bb\in\mathcal M} \|\nabla \mathsf f_\tau - \nabla\bb\|_2^2}\,,
\end{multline*}
where the last inequality arises from $\inf_{\bb\in\mathcal M} \|\nabla \mathsf f_\tau - \nabla\bb\|_2^2 \geq \delta\,\|\nabla \mathsf f_\tau \|_2^2$ for any $\tau\in[n_0,\tau_0)$. Taking the limit inferior as $\tau\to\tau_0^-$, we obtain
\begin{equation}
\label{limtau0}
\frac{\Vert \nabla \mathsf f_{n_0} \Vert^2_2-S_d\,\Vert \mathsf f_{n_0} \Vert_{2^*}^2} {\Vert \nabla \mathsf f_{n_0} \Vert^2_2 }\geq \delta\,\frac{\lim_{\tau\to\tau_0^-} \Vert \nabla \mathsf f_{\tau} \Vert^2_2 - S_d\,\Vert \mathsf f_{\tau_0} \Vert_{2^*}^2}{\liminf_{\tau\to\tau_0^-} \inf_{\bb\in\mathcal M} \Vert \nabla \mathsf f_{\tau} - \nabla\bb \Vert^2_2}\,.
\end{equation}
Note that the denominator appearing here does not vanish. Indeed, we have
$$
\inf_{\bb\in\mathcal M} \|\nabla \mathsf f_\tau - \nabla\bb\|_2^2 \geq \delta\,\|\nabla \mathsf f_\tau \|_2^2\geq \delta\,S_d\,\| \mathsf f_\tau \|_{2^*}^2= \delta\,S_d\,\|f\|_{2^*}^2>0\quad\forall\,\tau\in[n_0,\tau_0)
$$
and, as a consequence,
\begin{equation*} \label{eq:biggerequal}
\liminf_{\tau\to\tau_0^-} \inf_{\bb\in\mathcal M} \Vert \nabla \mathsf f_{\tau} - \nabla\bb \Vert^2_2\geq \delta\,S_d\,\|f\|_{2^*}^2>0\,.
\end{equation*}
The same inequality~\eqref{limtau0} remains valid if $\tau_0=n_0$ and if we interpret $\lim_{\tau \to \tau_0^-}$ and $\liminf_{\tau\to\tau_0^-}$ as evaluating at $\tau_0=n_0$.

At this point we find it convenient to apply Lemma~\ref{innerproduct} and use the representation
$$
\inf_{\bb\in\mathcal M} \Vert \nabla \mathsf f_{\tau} - \nabla\bb\Vert^2_2 = \Vert \nabla \mathsf f_{\tau} \Vert_2^2 - S_d\,\sup_{\bb\in\mathcal M_1} \left(\mathsf f_{\tau}, \bb^{2^*-1}\right)^2\,.
$$
Using~\eqref{eq:flowprop2}, that is, the continuity of $\tau\mapsto\sup_{\bb\in\mathcal M_1} \left(\mathsf f_\tau, \bb^{2^*-1}\right)^2$, we see that
$$
\liminf_{\tau\to\tau_0^-} \inf_{\bb\in\mathcal M} \Vert \nabla \mathsf f_{\tau} - \nabla\bb \Vert^2_2
= \lim_{\tau\to\tau_0^-} \Vert \nabla \mathsf f_{\tau} \Vert^2_2 - S_d\,\sup_{\bb\in\mathcal M_1} \left(\mathsf f_{\tau_0}, \bb^{2^*-1}\right)^2\,.
$$
Thus, the relevant quotient is equal to
\begin{equation}
\label{eq:quotientproof}
\frac{\lim_{\tau\to\tau_0^-} \Vert \nabla \mathsf f_{\tau} \Vert^2_2 - S_d\,\Vert \mathsf f_{\tau_0} \Vert_{2^*}^2}{\lim_{\tau\to\tau_0^-} \Vert \nabla \mathsf f_{\tau} \Vert^2_2 - S_d\,\sup_{\bb\in\mathcal M_1} \left(\mathsf f_{\tau_0}, \bb^{2^*-1}\right)^2}\,.
\end{equation}
Our goal in the remainder of this proof is to show that this quotient is larger or equal than $\Imu(\delta)$. We will use the fact that
\begin{equation}
\label{eq:holder}
\sup_{\bb\in\mathcal M_1} \left(\mathsf f_{\tau_0}, \bb^{2^*-1}\right)^2 \le \Vert f _{\tau_0} \Vert_{2^*}^2\,,
\end{equation}
which follows from H\"older's inequality. We also note that equality holds here if and only if $\mathsf f_{\tau_0}\in\mathcal M$.

Let us first handle the case where $\mathsf f_{\tau_0}\in\mathcal M$. Then by~\eqref{eq:biggerequal} and because of equality in~\eqref{eq:holder}, the quotient~\eqref{eq:quotientproof} is equal to 1, which by Lemma~\ref{mu1} can be further bounded from below by~$\Imu(\delta)$, leading to the claimed bound. This completes the proof in the case $\mathsf f_{\tau_0}\in\mathcal M$ and in what follows we assume
\begin{equation*}
\mathsf f_{\tau_0}\not\in\mathcal M\,.
\end{equation*}
As a consequence of this assumption and~\eqref{eq:holder}, we have
\begin{equation}
\label{eq:proofuseass2}
\|\nabla \mathsf f_{\tau_0}\|_2^2 > S_d\,\| \mathsf f_{\tau_0}\|_{2^*}^2 \geq S_d\,\sup_{\bb\in\mathcal M_1} \left(\mathsf f_{\tau_0}, \bb^{2^*-1}\right)^2\,.
\end{equation}
Next, we observe that for $\alpha > \beta$ the function $x\mapsto(x-\alpha)/(x-\beta)$ is monotone increasing on the interval $(\beta,\infty)$. This, together with the strict inequality in~\eqref{eq:proofuseass2}, implies that the quotient~\eqref{eq:quotientproof} can be bounded from below by
\begin{multline}
\label{eq:lowerboundproof}
\frac{\lim_{\tau\to\tau_0^-} \Vert \nabla \mathsf f_{\tau} \Vert^2_2 - S_d\,\Vert \mathsf f_{\tau_0} \Vert_{2^*}^2}{\lim_{\tau\to\tau_0^-} \Vert \nabla \mathsf f_{\tau} \Vert^2_2 - S_d\,\sup_{\bb\in\mathcal M_1} \left(\mathsf f_{\tau_0}, \bb^{2^*-1}\right)^2}
\\ \geq \frac{\Vert \nabla \mathsf f_{\tau_0} \Vert^2_2 - S_d\,\Vert \mathsf f_{\tau_0} \Vert_{2^*}^2}{\Vert \nabla \mathsf f_{\tau_0} \Vert^2_2 - S_d\,\sup_{\bb\in\mathcal M_1} \left(\mathsf f_{\tau_0}, \bb^{2^*-1}\right)^2}\,.
\end{multline}
We now claim that
\begin{equation}
\label{eq:goodineq}
\inf_{\bb\in\mathcal M} \|\nabla \mathsf f_{\tau_0} - \nabla\bb\|_2^2 \leq \delta\,\|\nabla \mathsf f_{\tau_0}\|_2^2\,.
\end{equation}
Once this is proved, we can bound the right side of~\eqref{eq:lowerboundproof} from below by $\Imu(\delta)$. This inequality is the claimed inequality after taking into account~\eqref{limtau0}.

To prove~\eqref{eq:goodineq}, we first note that it is verified if $\tau_0=n_0+1$. Indeed, $\mathsf f_{n_0+1}=f_{n_0+1}$ by~\eqref{eq:finfty} and therefore, by assumption of alternative (b), $\inf_{\bb\in\mathcal M} \|\nabla \mathsf f_{n_0+1} - \nabla\bb\|_2^2 <\delta\,\|\nabla \mathsf f_{n_0+1}\|_2^2$.

Now let $\tau_0<n_0+1$. We argue by contradiction and assume that
\begin{equation}
\label{eq:case2}
\inf_{\bb\in\mathcal M} \|\nabla \mathsf f_{\tau_0} - \nabla\bb\|_2^2 > \delta\,\|\nabla \mathsf f_{\tau_0}\|_2^2\,.
\end{equation}
Because of this strict inequality and the definition of $\tau_0$, for any $k\in\N$ there are $\sigma_k\in(\tau_0,n_0+1)$ with $\lim_{k\to\infty}\sigma_k=\tau_0$ such that $\inf_{\bb\in\mathcal M} \|\nabla \mathsf f_{\sigma_k} - \nabla\bb\|_2^2 < \delta\,\|\nabla \mathsf f_{\sigma_k}\|_2^2$, that is,
$$
\|\nabla \mathsf f_{\sigma_k}\|_2^2 - S_d\,\sup_{\bb\in\mathcal M_1} \left(\mathsf f_{\sigma_k},\bb^{2^*-1}\right)^2 < \delta\,\|\nabla \mathsf f_{\sigma_k} \|_2^2\quad\forall\,k\in\N\,.
$$
Letting $k\to\infty$ and using~\eqref{eq:flowprop2} as well as the right continuity of $\|\nabla \mathsf f_\tau\|_2^2$, see~\eqref{eq:flowprop3}, we deduce that
$$
\|\nabla \mathsf f_{\tau_0}\|_2^2 - S_d\,\sup_{\bb\in\mathcal M_1} \left(\mathsf f_{\tau_0},\bb^{2^*-1}\right)^2 \leq \delta\,\|\nabla \mathsf f_{\tau_0} \|_2^2\,.
$$
This is the same as $\inf_{\bb\in\mathcal M} \|\nabla \mathsf f_{\tau_0} - \nabla\bb\|_2^2 \leq \delta\,\|\nabla \mathsf f_{\tau_0}\|_2^2$ and contradicts~\eqref{eq:case2}. This proves~\eqref{eq:goodineq} and completes the proof of the lemma.
\end{proof}
%-------------------------------------------------
\begin{remark}
The above argument would be simpler if $\tau\mapsto \Vert \nabla \mathsf f_\tau \Vert_2^2$ were continuous for an appropriate choice of hyperplanes (see Appendix~\ref{contrearr}) in the definition of the flow. Since the flow is weakly continuous in $\dot{\mathrm H}^1(\R^d)$, continuity of the norm is equivalent to (strong) continuity of the flow in $\dot{\mathrm H}^1(\R^d)$. Thus, for continuity of the norm for an appropriate choice of hyperplanes, it is necessary that there is such a choice for which the Steiner symmetrizations approximate $f^*$ in $\dot{\mathrm H}^1(\R^d)$. According to a theorem of Burchard~\cite{Burchard} this holds if and only if $f$ is co-area regular, i.e, if and only if the distribution function
$$
h\mapsto |\{x \in \R^d: f(x)>h,\,\nabla f(x) = 0\}|
$$
has no absolutely continuous component. As shown by Almgren and Lieb~\cite{AlmgrenLieb}, both co-area regular and co-area irregular functions are dense for $d\geq 2$.
\end{remark}
%-------------------------------------------------

%%%%%%%%%%%%%%%%%%%%%%%%%%%%%%%%%%%%%%%%%%%%%%%%%%
\subsubsection{Proof of Theorem~\ref{Cor:Summary}}

It is now easy to prove the main result of this section, Theorem~\ref{Cor:Summary}. Let $\delta\in(0,1/2)$ and assume that $0\leq f\in\dot{\mathrm H}^1(\R^d)\setminus\mathcal M$ satisfies
$$
\inf_{\bb\in\mathcal M} \|\nabla f-\nabla\bb\|_2^2 \geq \delta\,\|\nabla f\|_2^2\,.
$$
By Lemma~\ref{alternatives} either alternative (a) or (b) holds. In the first case, we apply Lemmas~\ref{alternativea} and~\ref{mu1}, and in the second case, we apply Lemma~\ref{alternativeb}. This completes the proof.\qed

%%%%%%%%%%%%%%%%%%%%%%%%%%%%%%%%%%%%%%%%%%%%%%%%%%
\subsection{From nonnegative functions to arbitrary functions}\label{sec:posneg}

We recall that $\mathbf{\mathscr C}_{d,\rm BE}$ denotes the optimal constant in~\eqref{eq:bianchi-egnell0}. Similarly, we denote by $\mathbf{\mathscr C}_{d,\rm BE}^{\rm pos}$ the optimal constant in~\eqref{eq:bianchi-egnell0} when restricted to nonnegative functions $f$. Thus, $\mathbf{\mathscr C}_{d,\rm BE}^{\rm pos}\geq \mathbf{\mathscr C}_{d,\rm BE}$.
We do not know whether these two constants coincide or not. The main result in this section will be to prove the following lower bound on $\mathbf{\mathscr C}_{d,\rm BE}$ in terms of $\mathbf{\mathscr C}_{d,\rm BE}^{\rm pos}$.

%-------------------------------------------------------------------------
\begin{proposition}\label{Prop:BE} For any $d\ge3$,
\[\label{eq:cBE1}
\mathbf{\mathscr C}_{d,\rm BE}\ge \min\left\{ \textstyle \frac12\,\mathbf{\mathscr C}_{d,\rm BE}^{\rm pos}, 1-2^{-\frac 2d} \right\} .
\]
\end{proposition}
%-------------------------------------------------------------------------
\begin{proof}
To simplify the notation, given a function $v\in\dot{\mathrm H}^1(\R^d)$, we introduce the deficit
\[
\deficit(v):=\nrm{\nabla v}2^2-S_d\,\nrm v{2^*}^2=\nrm v{2^*}^2\,\Deficit(v) \,,
\]
where $\Deficit$ is defined in Section~\ref{sec:intro} (see for instance~\eqref{eq:bianchi-egnell0}). Also, we set $\alpha_d:=\frac2{2^*}=1-\frac2d<1$,
\[
h(p):=p^{\alpha_d}+(1-p)^{\alpha_d}-1\,,\quad\mbox{and}\quad h_d:=h(\tfrac12)=2^{1-\alpha_d}-1=2^\frac2d-1\,.
\]

Let us consider a function $u\in\dot{\mathrm H}^1(\R^d)$. By homogeneity we can assume that $\nrm u{2^*}=1$.
Let $u_\pm$ denote the positive and negative parts of $u$, set
\[
m :=\nrm{u_-}{2^*}^{2^*},
\]
and assume (without loss of generality) that
\begin{equation}\label{m:onehalf}
m\in[0,1/2]\,.
\end{equation}
Note that $\nrm{u_+}{2^*}^{2^*}=1-m$ and $\nrm{\nabla u}2^2=\nrm{\nabla u_-}2^2+\nrm{\nabla u_+}2^2$. Hence, we have
\begin{equation}\label{eq:deficit1}
\deficit(u)=\nrm{\nabla u}2^2-S_d= \deficit(u_+)+\deficit(u_-)+ S_d\,h(m).
\end{equation}
Since the function $p\mapsto h(p)$ is monotone increasing and concave on $[0,1/2]$, we have
\begin{equation}\label{eq:h1}
2\,h_d\,p\le h(p)\,.
\end{equation}
Also, if we set $\xi_d:=2\,(1-2^{-\alpha_d})$, the function $f(p):=(1-p)^{\alpha_d}-1+\xi_d\,p$ satisfies $f(0)=f(1/2)=0$ and $f''(p)\le0$, so that $f(p)\ge0$ for all $p\in[0,1/2]$. Hence, by~\eqref{m:onehalf}, we have
\[\label{Claim4}
(1-p)^{\alpha_d}\ge1-\xi_d\,p\,,
\]
which, by the definition of $h(p)$, yields
$$
h(p)\geq p^{\alpha_d} - \xi_d\,p\,.
$$
Combining this bound with~\eqref{eq:h1}, this gives
$$
\Big(1+\frac{\xi_d}{2\,h_d}\Big)\,h(p) \geq p^{\alpha_d}\,.
$$
Therefore, recalling~\eqref{eq:deficit1} and noticing that $\deficit(u_-)+S_d\,m^{\alpha_d}=\nrm{\nabla u_-}2^2$, we get
\begin{multline*}
\deficit(u)\geq \deficit(u_+)+\deficit(u_-)+ S_d\,\frac{2\,h_d}{2\,h_d+\xi_d}\,m^{\alpha_d} \\ \geq \deficit(u_+)+\frac{2\,h_d}{2\,h_d+\xi_d} \nrm{\nabla u_-}2^2.
\end{multline*}
By definition, we have
\[
\deficit(u_+)\ge \mathbf{\mathscr C}_{d,\rm BE}^{\rm pos}\,\inf_{\bb\in\mathcal M}\nrm{\nabla u_+-\nabla\bb}2^2.
\]
As a consequence, if $g_+\in\mathcal M$ is optimal for $u_+$, we obtain
\begin{multline*}
\deficit(u)\geq \mathbf{\mathscr C}_{d,\rm BE}^{\rm pos} \nrm{\nabla u_+\!-\!\nabla g_+}2^2 \!+\!\frac{2\,h_d}{2\,h_d+\xi_d} \nrm{\nabla u_-}2^2\\
\geq \min\Big\{\mathbf{\mathscr C}_{d,\rm BE}^{\rm pos},\frac{2\,h_d}{2\,h_d+\xi_d}\Big\}\nrm{\nabla u_+-\nabla g_+}2^2\\ +\min\Big\{\mathbf{\mathscr C}_{d,\rm BE}^{\rm pos},\frac{2\,h_d}{2\,h_d+\xi_d}\Big\}\nrm{\nabla u_-}2^2
\\ \geq \frac12 \min\Big\{\mathbf{\mathscr C}_{d,\rm BE}^{\rm pos},\frac{2\,h_d}{2\,h_d+\xi_d}\Big\}\nrm{\nabla u-\nabla g_+}2^2\,.
\end{multline*}
Since $2\,h_d+\xi_d=2\cdot 2^\frac2{d} - 2+2-2^{1-\alpha_d}=2^\frac2{d}$
we get
$$
\frac{h_d}{2\,h_d+\xi_d}=2^{-\frac2{d}}\left(2^\frac2d-1\right)= 1-2^{-\frac2{d}}\,,
$$
which concludes the proof.
\end{proof}

%%%%%%%%%%%%%%%%%%%%%%%%%%%%%%%%%%%%%%%%%%%%%%%%%%
\subsection{Stability of the Sobolev inequality: Proof of Theorem~\ref{main}}\label{sec:proof}

We now combine the results from the previous three sections and deduce in this way the main result of this paper.

\begin{proof}
We recall that the constant $\mathbf{\mathscr C}_{d,\rm BE}^{\rm pos}$ was defined in the previous subsection and that $\Imu(\delta)$ was defined in~\eqref{eq:mudelta2}. Then, as a consequence of Theorem~\ref{Cor:Summary}, we have
$$
\mathbf{\mathscr C}_{d,\rm BE}^{\rm pos} \geq \sup_{0<\delta\leq 1} \delta\,\Imu(\delta)\,.
$$
(Indeed, for any $\delta\in(0,1/2)$, if $f$ satisfies $\|\nabla f - \nabla g\|_2^2 \geq \delta\,\|\nabla f\|^2$, then $\mathcal E(f)\geq \delta\,\Imu(\delta)$, while if $\|\nabla f - \nabla g\|_2^2 \leq \delta\,\|\nabla f\|^2$, then $\mathcal E(f)\geq \Imu(\delta)\geq \delta\,\Imu(\delta)$.)
Thus, it remains to bound $\Imu(\delta)$ for a suitable $\delta\in(0,1/2)$.

We let $\epsilon_0$, $\tilde\delta\in(0,1/2)$ be as in Theorem~\ref{unifboundclose}. We will bound $\Imu(\delta)$ with $\delta = \frac{\tilde\delta}{1+\tilde\delta}\in(0,\tfrac12)$. Thus, let $0\leq f\in \dot{\mathrm H}^1(\R^d)$ with
$$
\inf_{\bb\in\mathcal M} \|\nabla f-\nabla\bb\|_2^2 \leq \tfrac{\tilde\delta}{1+\tilde\delta}\,\|\nabla f\|_2^2\,.
$$
The infimum on the left side is attained for the reasons given in Section~\ref{sec:flow}. After a translation, a dilation and multiplication by a constant, we may assume that it is attained at $g= (2/(1+|x|^2))^{(d-2)/2}$. We now pass to the sphere using the stereographic projection as in Section~\ref{sec:preliminary}. Let $0\leq u\in \mathrm H^1(\Sph^d)$ be the function associated to $f$. The function $1$ is associated to $g$ and we set $r:= u-1$. The fact that the distance is attained at $1$ implies that $r$ satisfies the orthogonality conditions~\eqref{eq:orthosphere}. Moreover, with $\mathsf A$ given by~\eqref{A}, we have
\begin{multline*}
\isd{ \(|\nabla r|^2+\mathsf A\,r^2\)} \leq \tfrac{\tilde\delta}{1+\tilde\delta} \isd{ \(|\nabla u|^2+\mathsf A\,u^2\)} \\ = \tfrac{\tilde\delta}{1+\tilde\delta} \left(\mathsf A+ \isd{ \(|\nabla r|^2+\mathsf A\,r^2\)} \right),
\end{multline*}
so
$$
\isd{ \(|\nabla r|^2+\mathsf A\,r^2\)} \leq \tilde \delta\,\mathsf A\,.
$$
By the Sobolev inequality, this implies
$$
\left( \isd{ r^q} \right)^{2/q} \leq \tilde\delta\,,
$$
and therefore we are in the situation of Theorem~\ref{unifboundclose}. We deduce that
$$
\isd{\(|\nabla u|^2+ \mathsf A\,u^2\)} - \mathsf A\left( \isd{ u^q} \right)^{2/q} \geq \theta\,\epsilon_0 \isd{ \(|\nabla r|^2+\mathsf A\,r^2\)}\,.
$$
Translating this result back to $\R^d$, we have shown that
$$
\Imu\Big(\tfrac{\tilde\delta}{1+\tilde\delta}\Big) \geq \theta\,\epsilon_0 = \tfrac{4\,\epsilon_0}{d-2}\,,
$$
and therefore
$$
\mathbf{\mathscr C}_{d,\rm BE}^{\rm pos} \geq \tfrac{\tilde\delta}{1+\tilde\delta} \,\tfrac{4\,\epsilon_0}{d-2}\,,
$$
where we recall that $0<\epsilon_0<\frac13$ is fixed and $\tilde\delta$ depends on $\epsilon_0$, but not on $d$. This constant has the claimed $d^{-1}$ behavior.

We turn now to the case of general, not necessarily nonnegative functions. By Proposition~\ref{Prop:BE}
$$
\mathbf{\mathscr C}_{d,\rm BE} \geq \min\left\{ \tfrac12\,\mathbf{\mathscr C}_{d,\rm BE}^{\rm pos}, 1- 2^{-\frac2d}\right\}\,.
$$
Using $1- 2^{-\frac2d} \geq (2\,\ln2)/d$ together with the result for $\mathbf{\mathscr C}_{d,\rm BE}^{\rm pos}$ we obtain also in the general case the claimed $d^{-1}$ behavior. As constant in Theorem~\ref{unifboundclose} we get
\be{beta}
\beta = \min\left\{ \tfrac{2\,\epsilon_0\,\tilde\delta}{1+\tilde\delta}, 2\,\ln 2\right\}\,,
\ee
which is computable, since $\tilde\delta$ depends in a complicated, yet explicit way on $\epsilon_0$.
\end{proof}

%-------------------------------------------------
\begin{remark}\label{firstboundrem2} The constant given by~\eqref{beta} is a lower estimate of $d\,\mathbf{\mathscr C}_{d,\rm BE}$, which for large $d$ is of the same order as the strict upper estimate obtained from~\eqref{eq:upperbound}. If we apply Proposition~\ref{firstbound} instead of Theorem~\ref{unifboundclose} in the above argument, we obtain
$$
\mathbf{\mathscr C}_{d,\rm BE}^{\rm pos} \geq \sup_{0<\delta<1/2} \delta \Imu(\delta) \geq \sup_{0<\tilde\delta<1} \tfrac{\tilde\delta}{1+\tilde\delta}\,\mathsf m(\tilde\delta^{1/2}) = \sup_{0<\delta<1/2} \delta\,\mathsf m\Big(\sqrt{\tfrac{\delta}{1-\delta}}\,\Big)
$$
with $\mathsf m$ given by~\eqref{eq:mudelta}. As explained in Remark~\ref{firstboundrem}, this lower bound is not very good for large dimensions. In the above expression, it corresponds to a right-hand side~of the order of $2^{-d}\,d^{-(d+2)/2}$ as $d\to+\infty$, but for $d=3$, $4$, $5$, $6$ it gives decent numerical lower bounds on~$\mathbf{\mathscr C}_{d,\rm BE}^{\rm pos}$.
\end{remark}
%-------------------------------------------------

%%%%%%%%%%%%%%%%%%%%%%%%%%%%%%%%%%%%%%%%%%%%%%%%%%
%%%%%%%%%%%%%%%%%%%%%%%%%%%%%%%%%%%%%%%%%%%%%%%%%%
\section{The large-dimensional limit: Proof of Corollary~\ref{logsob}}\label{sec:logsob}

Assume that $d\ge3$ and consider the stability estimate for Sobolev's inequality
\be{Bianchi-Egnell}
\nrm{\nabla f}2^2-\mathsf S_d\,\nrm f{2^*}^2\ge\frac{\beta(d)}d\inf_{\bb\in\mathcal M}\nrm{\nabla f-\nabla\bb}2^2\,,
\ee
for all $\, 0\le f\in\dot{\mathrm H}^1(\R^d)$, 
where $\beta(d)=d\,\mathbf{\mathscr C}_{d,\rm BE}^{\rm pos}>0$ denotes the optimal stability constant for {\it nonnegative functions.\/} Theorem~\ref{main} (also see Theorem~\ref{Cor:Summary}) provides us with an explicit lower estimate of $\beta(d)$ and shows that
\be{bstar}
\beta_\star=\liminf_{d\to+\infty}\beta(d)>0\,.
\ee

As noted for instance in~\cite{MR1164616}, to obtain the logarithmic Sobolev inequality as a limit of the Sobolev inequality when $d\to+\infty$, an important step is to perform a rescaling depending on $d$. In order to do this, let $u$ be a nonnegative Lipschitz function of compact support in~$\R^N$ and consider the ansatz
\begin{equation} \label{eq:ansatz}
f(x) := u(x_1, \dots, x_N) \,f_*(x) \,,
\end{equation}
where $f_*$ is a Sobolev optimizer in dimension $d\geq N$. Specifically, we choose
\[
f_*(x)=Z_d^{\frac{2-d}{2\,d}}\,\(1+\tfrac1{r_d^2}\,|x|^2\)^{1-\frac d2}\quad \forall\,x\in\R^d\,,
\]
with
$$
r_d =\scriptstyle \sqrt{\frac{d}{2\,\pi}} \,.
$$
The normalization constant $Z_d$ is chosen to render $\nrm{f_*}{2^*} =1$.
Note that $f_*(x)=r_d^{1-d/2}\,g_*(x/r_d)$, with $g_*$ given by~\eqref{gstar}, solves the Sobolev equation~\eqref{Sobolev:equation} with sharp Sobolev constant $S_d=d\,(d-2)\,r_d^{-2}\,Z_d^{2/d}$ and
\be{Zd}
Z_d=\(\tfrac d2\)^\frac d2\,\frac{\Gamma\(\tfrac d2\)}{\Gamma(d)} =\(\tfrac{d}{8\,\pi}\)^\frac{d}2\,|\mathbb S^d|=\frac{r_d^d}{2^d}\,|\mathbb S^d|\,.
\ee
It is also easy to see that
\be{limitzd}
\lim_{d\to+\infty} Z_d^{\frac2d} = \frac{e}4\,.
\ee
By integration by parts, using the fact that $f_*$ is a Sobolev optimizer, we find
\begin{multline}\label{gradientlimit}
\nrm{\nabla f}2^2 = \int_{\R^d} |\nabla u|^2 \,f_*^2 \,dx - \int u^2 \,f_* \,\Delta f_* \,dx \\ =\ird{|\nabla u|^2\,f_*^2}+\tfrac{d\,(d-2)}{r_d^2}\,Z_d^{\frac2{d}} \,\ird{u^2\,f_*^{2^*}}\,.
\end{multline}
It follows that the left-hand side~of the stability inequality~\eqref{Bianchi-Egnell}, written for $f=u\,f_*$, is
\[
\ird{|\nabla u|^2\,f_*^2}+\tfrac{d\,(d-2)}{r_d^2}\,Z_d^{\frac2{d}} \,\ird{u^2\,f_*^{2^*}}-\tfrac{d\,(d-2)}{r_d^2}\,Z_d^{\frac2{d}}\, \(\ird{u^{2^*}\,f_*^{2^*}}\)^{2/2^*},
\]
which can be written as
\begin{multline*}
Z_d^{\frac2{d}} \irdmu{|\nabla u|^2\(1+\tfrac1{r_d^2}\,|x|^2\)^2}\\ -2\,\pi\,(d-2)\,Z_d^{\frac2{d}}\left(\(\irdmu{u^{2^*}}\)^{2/2^*}-\irdmu{u^2}\right),
\end{multline*}
where $d\mu_d=\,f_*^{2^*}(x)\,dx$ is a probability measure.

Let us write $x=(y,z)\in\R^N\times\R^{d-N}\approx\R^d$, for some integer $N$ such that $1\le N<d$. With $|x|^2=|y|^2+|z|^2$~and
\[
1+\tfrac1{r_d^2}\,|x|^2=1+\tfrac1{r_d^2}\,\big(|y|^2+|z|^2\big)=\(1+\tfrac1{r_d^2}\,|y|^2\)\(1+\tfrac{|z|^2}{r_d^2+|y|^2}\)\,,
\]
we can integrate over the $z$ variable to obtain
\begin{multline}\label{dz}
\int_{\R^{d-N}}\frac{dz}{\(1+\tfrac1{r_d^2}\big(|y|^2+|z|^2\big)\)^d}\\=\frac{r_d^{d-N}}{\(1+\tfrac1{r_d^2}\,|y|^2\)^\frac{N+d}2}\int_{\R^{d-N}}\frac{d\zeta}{\(1+|\zeta|^2\)^d}=\frac{\Gamma\(\tfrac{d+N}2\)\(\tfrac d2\)^\frac{d-N}2}{\Gamma(d)\(1+\tfrac1{r_d^2}\,|y|^2\)^\frac{N+d}2}\,.
\end{multline}
By taking into account the limits
\begin{multline}\label{dzlim}
\lim_{d\to+\infty}\(1+\tfrac1{r_d^2}\,|y|^2\)^{-\frac{N+d}2}=e^{-\pi\,|y|^2}\;\mbox{and} \\ \lim_{d\to+\infty}\frac{r_d^{d-N}}{Z_d}\int_{\R^{d-N}}\kern-5pt\frac{d\zeta}{\(1+|\zeta|^2\)^d} =\lim_{d\to+\infty}\frac{\Gamma\(\tfrac{d+N}2\)}{Z_d\,\Gamma(d)}\(\frac d2\)^\frac{d-N}2\kern-6pt=1\,,
\end{multline}
we obtain
\be{udgamma}
\lim_{d\to+\infty}\irdmu{|u(y)|^2}=\irNg{u^2}
\ee
where $d\gamma(y):=e^{-\pi\,|y|^2}\,dy$ is a Gaussian probability measure. A similar computation shows that
\[
\lim_{d\to+\infty}\irdmu{|\nabla u|^2\(1+\tfrac1{r_d^2}\,|x|^2\)^2}=4\irNg{|\nabla u|^2} \,,
\]
because
\[
\lim_{d\to+\infty}\frac1{Z_d}\int_{\R^{d-N}}\(1+\tfrac1{r_d^2}\(|y|^2+|z|^2\)\)^{2-d}dz=4\,.
\]
On the other hand, let $\varepsilon:=1/(d-2)$ and write
\begin{multline*}
(d-2)\left[\(\irNg{u^{2^*}}\)^{2/2^*}-\irNg{u^2}\right]\\ =\frac1\varepsilon\left[\(\irNg{u^{2+4\,\varepsilon}}\)^\frac1{1+2\,\varepsilon}-\irNg{u^2}\right]\,.
\end{multline*}
As a consequence, we obtain
\begin{multline*}
\lim_{d\to+\infty}(d-2)\left[\(\irNg{u^{2^*}}\)^{2/2^*}-\irNg{u^2}\right]\\
=\frac d{d\varepsilon} \Big|_{\varepsilon=0} \(\irNg{u^{2\,(1+2\,\varepsilon)}}\)^\frac1{1+2\,\varepsilon} =2\irNg{u^2\,\ln\(\frac{u^2}{\irNg{u^2}}\)}\,.
\end{multline*}
Altogether, we find that
\begin{multline*}
\frac14\,\lim_{d\to+\infty}\left[\irdmu{|\nabla u|^2\(1+\tfrac1{r_d^2}\,|x|^2\)^2}\right.\\ \left.-2\,\pi\,(d-2)\(\(\irdmu{u^{2^*}}\)^{2/2^*}-\irdmu{u^2}\)\right]\\
=\irNg{|\nabla u|^2}-\pi\irNg{u^2\,\ln\(\frac{u^2}{\irNg{u^2}}\)}\,.
\end{multline*}
Using~\eqref{limitzd}, we have proved
%-------------------------------------------------
\begin{lemma}\label{lem:limit--lhs}
Let $f$ be given by~\eqref{eq:ansatz} where $u$ is a nonnegative Lipschitz function of compact support in $\R^N$. Then the limit of the left-hand side~of the stability inequality~\eqref{Bianchi-Egnell} as $d\rightarrow +\infty$ is
\begin{multline*}
\lim_{d\to +\infty} \nrm{\nabla f}2^2-\mathsf S_d\,\nrm f{2^*}^2 \\ = e\,\left[\irNg{|\nabla u|^2}-\pi\irNg{u^2\,\ln\(\frac{u^2}{\irNg{u^2}}\)}\right]\,.
\end{multline*}
\end{lemma}
%-------------------------------------------------
Next we deal with the large $d$ limit of the right side of~\eqref{Bianchi-Egnell}.
%-------------------------------------------------
\begin{lemma}\label{lem:limit-rhs}
Let $f$ be given by~\eqref{eq:ansatz} where $u$ is a nonnegative Lipschitz function of compact support in $\R^N$. Then
\begin{multline*}
\lim_{d\to+\infty} \frac 1d \inf_{\begin{array}{c}\scriptstyle a>0,\,b\in \R^d\\[-4pt] \scriptstyle c\in\R\end{array}}\nrm{\nabla f - c\,\nabla h_{a,b}(x)}2^2 \\ = \frac{\pi\,e}2 \inf_{c \in \R, \,b'\in \R^N} \irNg{\big|u(y)-c\,e^{\pi\,b'\cdot y}\big|^2} \,,
\end{multline*}
where $h_{a,b}(x) := |\Sp^d|^{-\frac{d-2}{2\,d}} \left( \frac{2\,a}{a^2+ |x-b|^2}\right)^{\frac{d-2}2}$ is, up to a multiplicative constant, any Sobolev optimizer.
\end{lemma}
%-------------------------------------------------
\begin{proof}
In the main part of this proof, using $(\cdot,\cdot)$ as in Lemma~\ref{innerproduct}, we shall show that
\begin{equation}
\label{eq:limitdinner}
\lim_{d\to+\infty} \sup_{a>0, \, b\in \R^d} \Big(f, h_{a,b}^{\frac{d+2}{d-2}}\Big)_{}
= \sup_{b'\in\R^N} \int_{\R^N}u(y)\,e^{-\frac\pi2\,|y|^2}\,e^{-\frac\pi2\,|y-b'|^2}\,dy \,.
\end{equation}

Before proving~\eqref{eq:limitdinner}, let us show that it implies the assertion of the lemma. As in Lemma~\ref{innerproduct} we can optimize the right side of~\eqref{Bianchi-Egnell} over $c$ and find
\begin{multline}\label{supc}
\inf_{a>0, \, b\in \R^d}\, \inf_{c \in \R}\nrm{\nabla f - c\,\nabla h_{a,b}}2^2\\ = \nrm{\nabla f}2^2 - S_d\, \sup_{a>0, \, b\in \R^d}\Big(f, h_{a,b}^{\frac{d+2}{d-2}}\Big)^2_{\mathrm L^2(\R^d)} \,,
\end{multline}
where $h_{a,b}$ satisfies $$\int_{\R^d} h_{a,b}(x)^{\frac{2\,d}{d-2}}\,dx = 1\,.$$ Similarly, from \begin{multline*}
\irNg{\big|u(y)-c\,e^{\pi\,b'\cdot y}\big|^2}\\ =\irNg{|u(y)|^2}+c^2\,e^{\pi\,|b'|^2}-2\,c\irNg{u(y)\,e^{\pi\,b'\cdot y}}
\end{multline*}
we deduce that
\begin{multline*}
\sup_{c\in\R}\irNg{\big|u(y)-c\,e^{\pi\,b'\cdot y}\big|^2}\\ =\irNg{|u(y)|^2}-e^{-\pi\,|b'|^2}\(\irNg{u(y)\,e^{\pi\,b'\cdot y}}\)^2\\
=\irNg{|u(y)|^2}-\(\int_{\R^N}u(y)\,e^{-\frac\pi2\,|y|^2}\,e^{-\frac\pi2\,|y-b'|^2}\,dy\)^2
\end{multline*}
and, consequently,
\begin{multline*}
\inf_{c \in \R, \,b'\in \R^N} \irNg{\big|u(y)-c\,e^{\pi\,b'\cdot y}\big|^2}
\\ = \int_{\R^N} u^2\,d\gamma - \sup_{b'\in\R^N} \(\int_{\R^N}u(y)\,e^{-\frac\pi2\,|y|^2}\,e^{-\frac\pi2\,|y-b'|^2}\,dy\)^2.
\end{multline*}
Now as before, using~\eqref{gradientlimit}, we get
\[
\lim_{d\to+\infty}\, \frac 1d\, \nrm{\nabla f}2^2= \frac{\pi\,e}2 \,\int e^{-\pi\,|y|^2} |u(y)|^2 \,dy \,.
\]
Inserting this together with the fact that $\lim_{d\to+\infty} S_d/d = \pi\,e/2$ into~\eqref{supc}, shows that~\eqref{eq:limitdinner} implies the assertion of the lemma.

Thus, from now on we concentrate on proving~\eqref{eq:limitdinner}. Clearly, we may assume $u\not\equiv 0$. It is easy to see that for every $d$, there are $a_d>0$ and $b_d\in\R^d$ such that
\[
\sup_{a>0, \, b\in \R^d}\Big(f, h_{a,b}^{\frac{d+2}{d-2}}\Big)_{}=\Big(f, h_{a_d,b_d}^{\frac{d+2}{d-2}}\Big)_{}\,.
\]
To pass to the limit in~\eqref{eq:limitdinner} as $d\to+\infty$, we have to study the asymptotic behavior of $a_d$ and $b_d$.

\medskip\noindent$\bullet$ {\it The limit of $a_d$.\/}
We will derive a lower and an upper bound on $\Big(f, h_{a_d,b_d}^{\frac{d+2}{d-2}}\Big)_{}$. For the lower bound we test the supremum defining this quantity with $a=r_d$ and $b=0$, in which case $h_{r_d,0}=f_*$. Arguing similarly as in~\eqref{udgamma} and recalling $u\not\equiv 0$, we obtain
\begin{equation}
\label{eq:limitlower}
\liminf_{d\to+\infty}\Big(f, h_{a_d,b_d}^{\frac{d+2}{d-2}}\Big)_{}\ge\lim_{d\to+\infty}\Big(f, h_{r_d,0}^{\frac{d+2}{d-2}}\Big)_{}=\irNg u>0\,.
\end{equation}
To derive an upper bound we use the fact that $f_*$ and $h_{a_d,0}$ are symmetric decreasing functions, which implies that
\begin{multline*}
0\le\Big(f, h_{a_d,b_d}^{\frac{d+2}{d-2}}\Big)_{}\le\nrm u\infty\int_{\R^d}f_*(x) \,h_{a_d,0}(x-b_d)^{\frac{d+2}{d-2}}\,dx\\ \le\nrm u\infty\int_{\R^d}f_*\,h_{a_d,0}^{\frac{d+2}{d-2}}\,dx\,.
\end{multline*}
By inserting the expression~\eqref{Zd} of $Z_d$ and setting $\alpha_d=a_d/r_d$, we obtain
\begin{align*}
\int_{\R^d}f_*\,h_{a_d,0}^{\frac{d+2}{d-2}}\,dx=&\,\frac{2^d}{r_d^d\,|\Sp^d|}\ird{\(1+\frac{|x|^2}{r_d^2}\)^{-\frac{d-2}2}\(\alpha_d+\frac{|x|^2}{\alpha_d\,r_d^2}\)^{-\frac{d+2}2}}\\
=&\,\frac1{|\Sp^d|}\ird{\(\frac2{1+|x|^2}\)^{\frac{d-2}2}\,\Bigg(\frac{2\,\alpha_d}{\alpha_d^2+|x|^2}\Bigg)^\frac{d+2}2}\\
=&\,\frac{|\Sp^{d-1}|}{|\Sp^d|}\int_0^{+\infty}\(\frac{2\,r}{1+r^2}\)^{\frac{d-2}2}\(\frac{2\,\alpha_d\,r}{\alpha_d^2+r^2}\)^\frac{d+2}2\,\frac{dr}r \,,
\end{align*}
where we scaled $x\mapsto r_d\,x$ and introduced radial coordinates. If we now set $\alpha_d=e^{s_d}$ and change variables to $r=e^t$, and then rescale according to $t\mapsto t/\sqrt d$, we find
\begin{multline*}
\int_{\R^d}f_*\,h_{a_d,0}^{\frac{d+2}{d-2}}\,dx=\frac{|\Sp^{d-1}|}{|\Sp^d|}\int_{-\infty}^\infty\big(\cosh t\big)^{-\frac{d-2}2}\big(\cosh(t-s_d)\big)^{-\frac{d+2}2}\,dt\\
=\frac{|\Sp^{d-1}|}{\sqrt d\,|\Sp^d|}\int_{-\infty}^\infty\big(\cosh\tfrac t{\sqrt d}\big)^{-\frac{d-2}2}\big(\cosh\tfrac{t-\sigma_d}{\sqrt d}\big)^{-\frac{d+2}2}\,dt
\end{multline*}
with $s_d=\sigma_d/\sqrt d$. By the elementary inequality $\cosh t\ge1+t^2/2$, we find the following bound for the integral on the right side for large $d$:
\[
\int_{-\infty}^\infty \big( 1+\tfrac{t^2}{2\,d}\big)^{-\frac{d-2}2}\,\big( 1+\tfrac{(t-\sigma_d)^2}{2\,d}\big)^{-\frac{d+2}2}dt\!\approx\!\int_{-\infty}^\infty \!\!e^{-\frac{t^2}{4}}\, e^{-\frac{(t-\sigma_d)^2}{4}}dt = \sqrt{2\,\pi}\,e^{-\frac{\sigma_d^2}{8}}.
\]
Using $\lim_{d\to+\infty}\frac{|\Sp^{d-1}|}{\sqrt d\,|\Sp^d|}=\sqrt{2\,\pi}$, we finally conclude by combining the upper and the lower bound that
\[
2\,\pi\,\liminf_{d\to+\infty}e^{-\frac{\sigma_d^2}{8}}\ge\liminf_{d\to+\infty}\Big(f, h_{a_d,b_d}^{\frac{d+2}{d-2}}\Big)_{}\ge\irNg u>0\,.
\]
As a consequence, $|\sigma_d|$ is bounded and we deduce that
\be{alphad}
\lim_{d\to+\infty}\frac{a_d}{r_d}=\lim_{d\to+\infty}\alpha_d=\lim_{d\to+\infty}e^{s_d}=\lim_{d\to+\infty}e^\frac{\sigma_d}{\sqrt d}=1\,.
\ee

\medskip\noindent$\bullet$ {\it A uniform bound on $b_d$.\/} We begin by noting that
\begin{multline*}
\Big(f,h_{a_d,b}^{\frac{d+2}{d-2}}\Big)_{}=\iint_{\R^N\times\R^{d-N}}u(y)\,f_*(y,z)\,h_{a_d,0}\big(y-b',z-b''\big)^{\frac{d+2}{d-2}}\,dy\,dz\\
\le\int_{\R^N}u(y)\(\int_{\R^{d-N}}f_*(y,z) \,h_{a_d,0}\big(y-b',z\big)^{\frac{d+2}{d-2}}\,dz\)dy
\end{multline*}
with $b=(b',b'')\in\R^N\times\R^{d-N}$, because $u$ is nonnegative, and $z\mapsto f_*(y,z)$ and $z\mapsto h_{a_d,0}(y,z)^{\frac{d+2}{d-2}}$ are symmetric decreasing functions. As a consequence, we can assume without loss of generality $b_d=(b_d',0)\in\R^N\times\R^{d-N}$.

Our task is to obtain a bound on $|b_d'|$. As before, we obtain this by deriving a lower and upper bound on $\Big(f,h_{a_d,b_d}^{\frac{d+2}{d-2}}\Big)_{}$. As lower bound we use again~\eqref{eq:limitlower}. For the upper bound we write
\begin{multline}
\label{eq:overlap}
\(f,h_{a_d,(b_d',0)}^{\frac{d+2}{d-2}}\)_{} \\
=\frac1{Z_d\,\alpha_d^\frac{d+2}2}\iint_{\R^N\times\R^{d-N}}u(y)\(1+\tfrac1{r_d^2}\(|y|^2+|z|^2\)\)^{-\frac{d-2}2}\\ \times \(1+\tfrac1{\alpha_d^2\,r_d^2}\(|y-b_d'|^2+|z|^2\)\)^{-\frac{d+2}2}\,dy\,dz \,,
\end{multline}
where $Z_d$ is given by~\eqref{Zd}. From H\"older's inequality we learn that
\begin{multline*}
\(f,h_{a_d,(b_d',0)}^{\frac{d+2}{d-2}}\)_{} \!\le\!\(\frac1{Z_d}\iint_{_{\R^N\times\R^{d-N}}}\!\!\!\!\!\!\!\!\!\!u(y)\(1+\tfrac1{r_d^2}\(|y|^2+|z|^2\)\)^{-d}\!\!dy\,dz\)^{\frac{d-2}{2\,d}}\\[-8pt]
\times\(\frac1{Z_d\,\alpha_d^d}\iint_{\R^N\times\R^{d-N}}u(y)\(1+\tfrac1{\alpha_d^2\,r_d^2}\(|y-b_d'|^2+|z|^2\)\)^{-d}\,dy\,dz\)^{\frac{d+2}{2\,d}}.
\end{multline*}
Let $R>0$ be such that $u$ is supported in the centered ball $\overline{B_R}$ of radius $R>0$ and assume that $|b_d'|>R$ (otherwise $|b_d'|\leq R$ and we have the claimed bound). It follows that $|y-b_d'|^2\ge(|b_d'|-R)^2$ in the support of $u$. Using the identity
\be{zeta}
\int_{\R^{d-N}}\(A^2+\tfrac1{\lambda^2}\,|z|^2\)^{-d}\,dz=\frac{\lambda^{d-N}}{A^{d+N}}\int_{\R^{d-N}}(1+|\zeta|^2)^{-d}\,d\zeta
\ee
based on the change of variables $z=A\,\lambda\,\zeta$, and applying it with $A=\frac1{\alpha_d\,r_d}\,\scriptstyle\sqrt{\alpha_d^2\,r_d^2+(|b_d'|-R)^2}$ and $\lambda=\alpha_d\,r_d$, we obtain
\begin{multline*}
\frac1{Z_d\,\alpha_d^d}\iint_{\R^N\times\R^{d-N}}u(y)\(1+\tfrac1{\alpha_d^2\,r_d^2}\(|y-b_d'|^2+|z|^2\)\)^{-d}\,dy\,dz\\
\le|B_R|\,\nrm u\infty\,\frac1{\alpha_d^N}\(1+\tfrac{(|b_d'|-R)^2}{\alpha_d^2\,r_d^2}\)^{-\frac{d+N}2}\,\frac{r_d^{d-N}}{Z_d}\int_{\R^{d-N}}(1+|\zeta|^2)^{-d}\,d\zeta\\
\le|B_R|\,\nrm u\infty\,\frac{d\,\alpha_d^{2-N}}{(d+N)\,\pi\,(|b_d'|-R)^2}\,\frac{r_d^{d-N}}{Z_d}\int_{\R^{d-N}}(1+|\zeta|^2)^{-d}\,d\zeta
\end{multline*}
using the inequality $\(1+t/k\)^{-k}\le1/t$ for all $t>0$ with $k=(d+N)/2$. As in~\eqref{udgamma}, using~\eqref{dzlim} and~\eqref{alphad}, this yields
\[
\liminf_{d\to+\infty}\(f,h_{a_d,(b_d',0)}^{\frac{d+2}{d-2}}\)_{}
\le\sqrt{\frac{|B_R|\,\nrm u\infty\irNg u}{\pi\,\limsup_{d\to\infty} (|b_d'|-R)^2}}\,.
\]
Taking the lower bound in~\eqref{eq:limitlower} into account, we obtain
\[
\limsup_{d\to\infty} |b_d'|\le R+\sqrt{\frac{|B_R|\,\nrm u\infty}{\pi\irNg u}}\,.
\]
This proves that $b'_d$ is uniformly bounded w.r.t.~$d$.

\medskip\noindent$\bullet$ {\it The large dimensional limit.\/} We are finally in position to prove~\eqref{eq:limitdinner}. We first show that
\begin{equation}
\label{eq:limitdinner1}
\limsup_{d\to+\infty} \sup_{a>0, \, b\in \R^d} \Big(f, h_{a,b}^{\frac{d+2}{d-2}}\Big)_{}
\leq \sup_{b'\in\R^N} \int_{\R^N}u(y)\,e^{-\frac\pi2\,|y|^2}\,e^{-\frac\pi2\,|y-b'|^2}\,dy \,.
\end{equation}
To do so, we consider a sequence of $d$'s along which the limsup is attained. Because of the uniform bound on $b_d'$ we may pass to a subsequence along which $b_d'$ converges to some $b_\infty'\in\R^N$. It then suffices to prove~\eqref{eq:limitdinner1} where the limsup is taken along the chosen subsequence. In the following, we will always consider this subsequence, without displaying it in our notation.

It remains to identify a bound on $\limsup_{d\to+\infty} \Big(f, h_{a_d,(b_d',0)}^{\frac{d+2}{d-2}}\Big)_{}$. Our starting point is~\eqref{eq:overlap}. By H\"older's inequality, we obtain
\begin{multline*}
\int_{\R^{d-N}}\(1+\tfrac1{r_d^2}\(|y|^2+|z|^2\)\)^{-\frac{d-2}2}\(1+\tfrac1{\alpha_d^2\,r_d^2}\(|y-b_d'|^2+|z|^2\)\)^{-\frac{d+2}2}\!dz\\
\le\(\int_{\R^{d-N}}\(1+\tfrac1{r_d^2}\(|y|^2+|z|^2\)\)^{-d}\,dz\)^{-\frac{d-2}{2\,d}}\\ \times \(\int_{\R^{d-N}}\(1+\tfrac1{\alpha_d^2\,r_d^2}\(|y-b_d'|^2+|z|^2\)\)^{-d}\,dz\)^{-\frac{d+2}{2\,d}}\\
=\alpha_d^{d+2}\,r_d^{2\,d}\left(r_d^2+|y|^2\right)^{-\frac{(d-2)\,(d+N)}{4\,d}}\left(\alpha_d^2\,r_d^2+|y-b_d'|^2\right)^{-\frac{(d+2)\,(d+N)}{4\,d}}\\ \times \int_{\R^{d-N}}(1+|\zeta|^2)^{-d}\,d\zeta\\
=\alpha_d^\frac{(d+2)\,(d-N)}{2\,d}\,r_d^{d-N}\left(1+\tfrac1{r_d^2}\,|y|^2\right)^{-\frac{(d-2)\,(d+N)}{4\,d}}\\ \times \left(1+\tfrac1{\alpha_d^2\,r_d^2}\,|y-b_d'|^2\right)^{-\frac{(d+2)\,(d+N)}{4\,d}}\int_{\R^{d-N}}(1+|\zeta|^2)^{-d}\,d\zeta \,.
\end{multline*}
Here we used the change of variables identity~\eqref{zeta}, with $A=\frac1{r_d}\,\scriptstyle\sqrt{r_d^2+|y|^2}$ and $\lambda=r_d$ for the first integral in the above right-hand side, and $A=\frac1{\alpha_d\,r_d}\,\scriptstyle\sqrt{\alpha_d^2\,r_d^2+|y|^2}$ and $\lambda=\alpha_d\,r_d$ for the second integral. We learn from~\eqref{dzlim} and~\eqref{alphad} that
\begin{multline*}
\limsup_{d\to+\infty} \Big(f, h_{a_d,(b_d',0)}^{\frac{d+2}{d-2}}\Big)_{}
\le\int_{\R^N}u(y)\,e^{-\frac{\pi}2|y|^2}\,e^{-\frac{\pi}2|y-b_\infty'|^2}\,dy\\ \le\sup_{b'\in\R^N}\int_{\R^N}u(y)\,e^{-\frac{\pi}2|y|^2}\,e^{-\frac{\pi}2|y-b'|^2}\,dy\,.
\end{multline*}
This proves~\eqref{eq:limitdinner1}.

The converse asymptotic inequality, namely
\begin{equation}
\label{eq:limitdinner2}
\liminf_{d\to+\infty} \sup_{a>0, \, b\in \R^d} \Big(f, h_{a,b}^{\frac{d+2}{d-2}}\Big)_{}
\geq \sup_{b'\in\R^N} \int_{\R^N}u(y)\,e^{-\frac\pi2\,|y|^2}\,e^{-\frac\pi2\,|y-b'|^2}\,dy \,,
\end{equation}
follows in a similar, but simpler fashion. Indeed, it is easy to see that the supremum on the right side is attained at some $b_*'\in\R^N$, which we can use to bound the supremum on the left side from below by $\Big(f, h_{0,(b_*',0)}^{\frac{d+2}{d-2}}\Big)_{}$. Starting from~\eqref{eq:overlap} and using similar arguments as above it is easy to see that
\begin{multline*}
\lim_{d\to+\infty} \Big(f, h_{0,(b_*',0)}^{\frac{d+2}{d-2}}\Big)_{}
= \int_{\R^N}u(y)\,e^{-\frac\pi2\,|y|^2}\,e^{-\frac\pi2\,|y-b_*'|^2}\,dy \\ = \sup_{b'\in\R^N} \int_{\R^N}u(y)\,e^{-\frac\pi2\,|y|^2}\,e^{-\frac\pi2\,|y-b'|^2}\,dy \,.
\end{multline*}
This proves~\eqref{eq:limitdinner2} and consequently also~\eqref{eq:limitdinner}.
\end{proof}

Using Lemmata~\ref{lem:limit--lhs} and~\ref{lem:limit-rhs}, with $b=\pi\,\tilde b$, for nonnegative Lipschitz functions $u$ with compact support, we have proved the following result.
%-------------------------------------------------
\begin{proposition}\label{Prop:logsob} With $\beta_\star$ given by~\eqref{bstar}, for all nonnegative $u\in\mathrm H^1(\gamma)$,
\begin{multline*}
\irNg{|\nabla u|^2}-\pi\irNg{u^2\,\ln\(\frac{u^2}{\|u\|_{\mathrm L^2(\gamma)}^2}\)}\\ \ge\frac{\beta_\star\,\pi}2\,\inf_{b\in\R^N\!,\,c\in\R}\int_{\R^N} \big(u - c\,e^{b\cdot x}\big)^2\,d\gamma\,.
\end{multline*}
\end{proposition}
%-------------------------------------------------
The extension to any nonnegative function $u\in\mathrm H^1(\gamma)$ follows by a simple density argument, as the constants in Proposition~\ref{Prop:logsob} depend neither on the support nor on the bound on $|\nabla u|$. A crucial feature of Proposition~\ref{Prop:logsob} is that the stability constant $\beta_\star\,\pi/2$ does not depend on $N$. It is worth pointing out the constant $\beta_\star$ in this bound comes from the (unknown) best stability constant for Sobolev's inequality for nonnegative functions. Any lower bound on this stability constant gives a lower bound on the constant $\beta_*$. In particular, we have $\beta_\star\ge\beta$ with $\beta$ as in Theorem~\ref{main}.

\begin{proof}[Proof of Corollary~\ref{logsob}] We have to extend the result of Proposition~\ref{Prop:logsob} to the case of sign-changing functions. This part of the proof is a variation of the argument used in the proof of Proposition~\ref{Prop:BE}. We shall use the notation
$$
\deficit(u) := \int_{\R^N} |\nabla u|^2\,d\gamma - \pi \int_{\R^N} u^2 \ln\(\frac{u^2}{\|u\|_{\mathrm L^2(\gamma)}^2}\)\,d\gamma
\quad\text{for}\,u\in \mathrm H^1(\gamma)\,.
$$
By homogeneity we can assume $\|u\|_{\mathrm L^2(\gamma)}=1$. Replacing $u$ by $-\,u$ if necessary, we can also assume that
$$
m:= \|u_-\|_{\mathrm L^2(\gamma)}^2 \in [0,\tfrac12]\,.
$$
Then
$$
\deficit(u) = \deficit(u_+) + \deficit(u_-) + \pi\,h(m)\quad\mbox{with}\quad h(p) := - \big( p\ln p + (1-p) \ln(1-p) \big)\,.
$$
Since the function $p\mapsto h(p)$ is monotone increasing and concave on $[0,\frac12]$,
$$
h(p) \geq (2\ln 2)\,p \quad \text{for all}\quad p\in[0,\tfrac12]\,.
$$
Thus, with $\beta_\star$ denoting the constant in~\eqref{bstar},
\begin{multline*}
\deficit(u) \geq \deficit(u_+) + (2\,\pi\, \ln 2)\, m \\  \geq \frac{\beta_\star\, \pi}2\, \inf_{b\in\R^N\!,\,c\in\R} \|u_+ - c\,e^{b\cdot x} \|_{\mathrm L^2(\gamma)}^2 + (2\,\pi\, \ln 2)\, \|u_-\|_{\mathrm L^2(\gamma)}^2\\
 \geq \tfrac12\, \min\left\{\frac{\beta_\star\, \pi}2,\, 2\,\pi\, \ln 2 \right\}\,\inf_{b\in\R^N\!,\,c\in\R} \| u - c\,e^{b\cdot x} \|_{\mathrm L^2(\gamma)}^2\,.
\end{multline*}
This proves the inequality for the general case with
\be{bLSI}
\beta = \tfrac12\, \min\big\{\beta_\star,\, 4\, \ln 2 \big\}
\ee
and $\beta_\star$ given by~\eqref{bstar}. \end{proof}

Up to this point, we have stated the logarithmic Sobolev inequality in its version with respect to the normalized Gaussian measure. It has an equivalent version with respect to the Euclidean measure. We set $u=e^{\pi\,|x|^2/2}\,v$ and obtain from Corollary~\ref{logsob} and Proposition~\ref{Prop:logsob}
\begin{multline*}
\int_{\R^N} |\nabla v|^2\,dx - \pi \int_{\R^N} v^2 \ln\Bigg(\frac{v^2}{\|v\|_{\mathrm L^2(\R^N)}^2}\Bigg)\,dx - N\,\pi\,\|v\|_{\mathrm L^2(\R^N)}^2 \\ \geq \frac{\beta\,\pi}2 \inf_{b\in\R^N\!,\,c\in\R} \int_{\R^N} \Big|v - c\,e^{-\,\frac\pi2|x-b|^2} \Big|^2dx
\end{multline*}
by a simple integration by parts. Writing $v(x) = \lambda^{N/2}\,w(\lambda\, x)$ with a parameter $\lambda>0$, we obtain equivalently
\begin{multline*}
\lambda^2 \int_{\R^N} |\nabla w|^2\,dy - \pi \int_{\R^N} \!\!\!\!w^2 \ln\(\frac{w^2}{\| w\|_{\mathrm L^2(\R^N)}^2}\)dy - N\,\pi\,(1+\ln\lambda) \| w\|_{\mathrm L^2(\R^N)}^2\\
\geq \frac{\beta\,\pi}2 \inf_{b\in\R^N\!,\,c\in\R} \int_{\R^N} \Big|w - c\,e^{-\,\frac\pi{2\lambda^2}\,|y-b|^2} \Big|^2\,dy\,.
\end{multline*}
We bound the right side from below by extending the infimum over all $\lambda>0$ and then we optimize the left side with respect to $\lambda>0$. In this way we obtain the following stability version of the Euclidean logarithmic Sobolev inequality.
%-------------------------------------------------
\begin{corollary}
With $\beta>0$ given by~\eqref{bLSI} we have for all $N\in\N$ and all $w\in \mathrm H^1(\R^N)$,
\begin{multline*}
\| w\|_{\mathrm L^2(\R^N)}^2 \ln \left( \frac2{N\,\pi\,e}\,\frac{\int_{\R^N} |\nabla w|^2\,dx}{\| w\|_{\mathrm L^2(\R^N)}^2} \right) - \frac2N \int_{\R^N} w^2 \ln\(\frac{w^2}{\| w\|_{\mathrm L^2(\R^N)}^2}\)dx\\
\geq \frac{\beta}N \inf_{\lambda>0,\,b\in\R^N\!,\,c\in\R} \int_{\R^N} \Big|w - c\,e^{-\,\frac\pi{2\lambda^2}\,|y-b|^2} \Big|^2\,dy\,.
\end{multline*}
\end{corollary}
%-------------------------------------------------

%%%%%%%%%%%%%%%%%%%%%%%%%%%%%%%%%%%%%%%%%%%%%%%%%%
%%%%%%%%%%%%%%%%%%%%%%%%%%%%%%%%%%%%%%%%%%%%%%%%%%
\appendix\section{Some properties of continuous rearrangement}\label{contrearr}

In this subsection we discuss several aspects of the continuous rearrangement and prove some of its properties.

Brock's continuous Steiner rearrangement is based on the following operation for functions of one real variable that are finite unions of disjoint characteristic functions $\sum_{k=1}^N \chi_{(-a_k,a_k)}(x-b_k)$. Replace this function by $\sum_{k=1}^N \chi_{(-a_k,a_k)}\big(x-e^{-\,t}\,b_k\big)$ where $t$ varies from $0$ to $\infty$. As $t$ increases, the intervals start moving closer and as soon as any two intervals touch
one stops the process and redefines the set of intervals by joining the two that touched. Then one restarts the process and
keeps repeating it until all of them are joined into one. The movement stops once this interval is centered at the origin.
By the outer regularity of Lebesgue measure the level sets of a measurable function can be approximated by open sets
and, since in one dimension this is a countable union of open intervals, one can further approximate the level set by a
finite number of open disjoint intervals for which one uses the sliding argument explained above.

As mentioned before, this procedure can be generalized to higher dimensions by considering Steiner symmetrization with respect to a hyperplane.
One considers any hyperplane $H$ through the origin and then rearranges the function symmetrically about the hyperplane along each
line perpendicular to $H$, resulting in a function denoted by $f^{*H}$. For more information see~\cite{LiebLoss}.
In this fashion one obtains a continuous rearrangement $f \to \mathsf f^H_\tau, \tau \in [0,\infty]$, which was studied in detail by Brock~\cite{Brock,Brock2}. We shall refer
to the statements in those papers.

To pass from Steiner symmetrization to the symmetric decreasing rearrangement we consider a sequence of continuous Steiner symmetrizations and chain them with a new continous parameter {\it à la\/} Bucur--Henrot. Inspired by~\cite{BucurHenrot,Christ}, we proceed as follows. Given a function $f\in \mathrm L^p(\R^d)$ for some $1\leq p<\infty$ there is a sequence $(H_n)_{n\in\N}$ of hyperplanes such that, defining recursively with $f_0=f$,
$$
f_n := f^{*H_n}_{n-1}\,,\quad n=1,2, \dots\,,
$$
we have
$$
f_n \to f^*
\quad\text{in}\,\mathrm L^p(\R^d)\quad\text{as}\quad n\to\infty\,.
$$
In fact, it is shown in~\cite[Theorem~4.3]{Volcic} that this holds for `almost every' (in an appropriate sense) choice of hyperplanes. It is also of interest that this sequence can actually be chosen in a universal fashion (that is, independent of $f$ and $p$); see~\cite[Theorem~5.2]{vanSchaftingen}.

Given $f$ and the sequence $(f_n)_{n\in\N}$ as above, we set for any $n=0,1,2,\dots$
$$
\phi_n(\tau) := e^{\frac{\tau-n}{n+1-\tau}} - 1\,,
\quad \tau\in [n,n+1]\,,
$$
and define
\begin{equation}
\label{eq:ftau}
\mathsf f_\tau := (f_n)^{H_{n+1}}_{\phi_n(\tau)}\,,
\end{equation}
where the right side denotes Brock's continuous Steiner symmetrization with respect to the hyperplane $H_{n+1}$ with parameter $\phi_n(\tau)$ applied to $f_n$. As $\tau$ runs from $n$ to $n+1$, $\phi_n(\tau)$ runs from $0$ to~$\infty$ and there is no ambiguity at $\tau \in\N$ since $f_n = f_{n-1}^{* H_n}$ by definition. Thus, $\mathsf f_\tau$ is well defined for $\tau\in [0,\infty]$.

{}From the properties of Brock's flow, see, in particular,~\cite[Lemma~4.1]{Brock2}, we obtain the following properties for our flow.
%-------------------------------------------------
\begin{proposition}\label{symmlpcont}
Let $d\ge1$, $1\leq p<\infty$ and let $0 \le f \in \mathrm L^p(\R^d)$. Then, for any $\tau \in [0,\infty]$, the function $\mathsf f_\tau$ defined by~\eqref{eq:ftau} is in $\mathrm L^p( \R^d)$ and
$\Vert \mathsf f_\tau \Vert_p= \Vert \mathsf f \Vert_p$. Moreover, for any $\tau\in[0,\infty]$ and any sequence $(\tau_n)_{n\in\N}$ with $\lim_{n\to\infty}\tau_n=\tau$,
$$
\lim_{n\to\infty} \Vert \mathsf f_{\tau_n} - \mathsf f_\tau\Vert_p = 0\,.
$$
\end{proposition}
%-------------------------------------------------
The following fact is important for us.
%-------------------------------------------------
\begin{lemma}\label{continuous}
Let $d\ge3$ and $0\leq f\in \mathrm L^{2^*}(\R^d)$. The function
$$
\tau \mapsto \sup_{u \in \mathcal M_1}\left(\mathsf f_\tau,u^{2^*-1}\right)^2
$$
with $\mathsf f_\tau$ defined by~\eqref{eq:ftau} is continuous.
\end{lemma}
%-------------------------------------------------
\begin{proof}
We use the fact, shown in Proposition~\ref{symmlpcont}, that
$$
\lim_{\tau_1 \to \tau_2} \Vert \mathsf f_{\tau_1} -\mathsf f_{\tau_2}\Vert_{2^*} = 0\,.
$$
Fix $\varepsilon>0$. There exists $u_1\in \mathcal M_1$ such that $\sup_{u \in \mathcal M_1}\left|\left(\mathsf f_{\tau_1},u^{2^*-1}\right)\right| \le \left|\left(\mathsf f_{\tau_1},u^{2^*-1}_1\right)\right|+\varepsilon$ and hence
\begin{multline*}
\sup_{u \in \mathcal M_1}\left|\left(\mathsf f_{\tau_1},u^{2^*-1}\right)\right| - \sup_{u \in \mathcal M_1}\left|\left(\mathsf f_{\tau_2},u^{2^*-1}\right)\right| \\ \le \left|\left(\mathsf f_{\tau_1},u^{2^*-1}_1\right)\right|+\varepsilon - \left|\left(\mathsf f_{\tau_2},u^{2^*-1}_1\right)\right|\\
\le \left|\left(\mathsf f_{\tau_1},u^{2^*-1}_1\right) -\left(\mathsf f_{\tau_2},u^{2^*-1}_1\right)\right|+ \varepsilon\,,
\end{multline*}
which by H\"older's inequality is bounded above by
$$
\Vert \mathsf f_{\tau_1} -\mathsf f_{\tau_2}\Vert_{2^*}\,\Vert u^{2^*-1}_1 \Vert_{q} + \varepsilon = \Vert \mathsf f_{\tau_1} -\mathsf f_{\tau_2}\Vert_{2^*} + \varepsilon
$$
with $q= \frac{2^*}{2^*-1}\,$. Hence
$$
\limsup_{\tau_2 \to \tau_1}\left( \sup_{u \in \mathcal M_1}\left|\left(\mathsf f_{\tau_1},u^{2^*-1}\right)\right| - \sup_{u \in \mathcal M_1}\left|\left(\mathsf f_{\tau_2},u^{2^*-1}\right)\right|\right) \le \varepsilon\,.
$$
There exists $u_2 \in \mathcal M_1$ such that $\sup_{u \in \mathcal M_1}\left|\left(\mathsf f_{\tau_2},u^{2^*-1}\right)\right| \le \left|\left(\mathsf f_{\tau_2},u^{2^*-1}_2\right)\right|+\varepsilon$ and hence
$$
\sup_{u \in \mathcal M_1}\left|\left(\mathsf f_{\tau_1},u^{2^*-1}\right)\right| - \sup_{u \in \mathcal M_1}\left|\left(\mathsf f_{\tau_2},u^{2^*-1}\right)\right| \ge \left|\left(\mathsf f_{\tau_1},u_2^{2^*-1}\right)\right| - \left|\left(\mathsf f_{\tau_2},u_2^{2^*-1}\right)\right| -\varepsilon\,,
$$
which is greater or equal to
$$
-\left|\left(\mathsf f_{\tau_1},u_2^{2^*-1}\right) - \left(\mathsf f_{\tau_2},u_2^{2^*-1}\right)\right| -\varepsilon \ge - \Vert \mathsf f_{\tau_1}- \mathsf f_{\tau_2}\Vert_{2^*} - \varepsilon\,.
$$
Hence
$$
\liminf_{\tau_2 \to \tau_1}\left( \sup_{u \in \mathcal M_1}\left|\left(\mathsf f_{\tau_1},u^{2^*-1}\right)\right| - \sup_{u \in \mathcal M_1}\left|\left(\mathsf f_{\tau_2},u^{2^*-1}\right)\right|\right) \ge -\,\varepsilon\,.
$$
This proves the claimed continuity.
\end{proof}

We now consider the behavior of the gradient under the rearrangement flow. The following proposition is closely related to~\cite[Theorems~3.2 and~4.1]{Brock2}, but there inhomogeneous Sobolev spaces are considered, which leads to some minor changes. For the sake of simplicity we provide the details.
%-------------------------------------------------
\begin{proposition}
Let $0\leq f \in \dot{\mathrm H}^1(\R^d)$. Then $\mathsf f_\tau$ defined by~\eqref{eq:ftau} is in $\dot{\mathrm H}^1(\R^d)$ and
$\tau\mapsto \Vert \nabla \mathsf f_\tau \Vert_2$ is a nonincreasing, right-continuous function.
\end{proposition}
%-------------------------------------------------
\begin{proof}
By construction, it suffices to prove these properties for Brock's flow. Since the latter has the semigroup property $(\mathsf f_\sigma)_\tau = \mathsf f_{\sigma+\tau}$ for all $\sigma$, $\tau \ge 0$, it suffices to prove monotonicity and right-continuity at $\tau=0$.

We begin with the proof of monotonicity, which we first prove under the additional assumption that $f\in \mathrm L^2(\R^d)$. This is shown in~\cite[Theorem~3.2]{Brock2}, but we give an alternative proof. We proceed as in the proof of~\cite[Lemma~1.17]{LiebLoss}.
Extending~\cite[Corollary~2]{Brock} to the sequence of Steiner symmetrizations we find for three nonnegative functions
$f$, $g$, $h$ that
$$
\iint_{\R^d\times\R^d} \mathsf f_\tau(x)\,g_\tau(x-y)\,h_\tau(y)\,dx\,dy \ge \iint_{\R^d\times\R^d} f(x)\,g(x-y)\,h(y)\,dx\,dy\,.
$$
If we choose $g(x-y)$ to be the standard heat kernel, {\it i.e.}, $g(x-y) = e^{\Delta t}(x-y)$, then $g_\tau(x-y) = g(x-y)$ and hence
$$
\iint_{\R^d\times\R^d} \mathsf f_\tau(x)\,e^{\Delta t}(x-y)\,\mathsf f_\tau(y)\,dx\,dy \ge \iint_{\R^d\times\R^d} f(x)\,e^{\Delta t}(x-y)\,f(y)\,dx\,dy\,.
$$
Since $\Vert \mathsf f_\tau \Vert_2 = \Vert f \Vert_2$ by the equimeasurability of rearrangement,
$$
\frac1t \left(\Vert \mathsf f_\tau \Vert_2^2 - \left(\mathsf f_\tau, e^{\Delta t}\,\mathsf f_\tau\right) \right) \le \frac 1t \left(\Vert f \Vert_2^2 - \left(f, e^{\Delta t}f\right)\right)
$$
and letting $t \to 0$ yields the first claim under the additional assumption $f\in \mathrm L^2(\R^d)$.

For general $0\leq f\in\dot{\mathrm H}^1(\R^d)$ we apply the above argument to the functions $(f-\epsilon)_+$, $\epsilon>0$. They belong to $\mathrm L^2(\R^d)$ since $f$ vanishes at infinity and belongs to $\mathrm L^{2^*}(\R^d)$. We obtain
\begin{equation}
\label{eq:rearrangegrad}
\left\|\nabla\big((f-\epsilon)_+\big)_\tau\right\|_2 \leq \|\nabla (f-\epsilon)_+\|_2 \leq \|\nabla f\|_2\,.
\end{equation}
We claim that $\mathsf f_\tau\in\dot{\mathrm H}^1(\R^d)$ and $\nabla \big((f-\epsilon)_+\big)_\tau\rightharpoonup \nabla \mathsf f_\tau$ in $\mathrm L^2(\R^d)$ as $\epsilon\to 0^+$. Once this is shown, the claimed inequality follows from~\eqref{eq:rearrangegrad} by the weak lower semicontinuity of the $\mathrm L^2$ norm.

To prove the claimed weak convergence, note that by~\eqref{eq:rearrangegrad}, $\nabla\big((f-\epsilon)_+\big)_\tau$ is bounded in $\mathrm L^2(\R^d)$ as $\epsilon\to 0^+$ and therefore has a weak limit point. Let $F\in \mathrm L^2(\R^d)$ be any such limit point. Since $(f-\epsilon)_+\to f$ in $\mathrm L^{2^*}(\R^d)$, the nonexpansivity of the rearrangement~\cite[Lemma~3]{Brock} implies that $\big((f-\epsilon)_+\big)_\tau\to \mathsf f_\tau$ in $\mathrm L^{2^*}(\R^d)$. Thus, for any $\Phi\in C^1_c(\R^d)$,
\begin{multline*}
\int_{\R^d} (\nabla\cdot\Phi)\,\mathsf f_\tau\,dx \leftarrow
\int_{\R^d} (\nabla\cdot\Phi)\,\big((f-\epsilon)_+\big)_\tau\,dx \\ = - \int_{\R^d} \Phi\cdot \nabla \big((f-\epsilon)_+\big)_\tau\,dx \to - \int_{\R^d} \Phi\cdot F\,dx
\end{multline*}
as $\epsilon\to0^+$. This proves that $\mathsf f_\tau$ is weakly differentiable with $\nabla \mathsf f_\tau = F$. In particular, $\mathsf f_\tau\in\dot{\mathrm H}^1(\R^d)$ (note that $\mathsf f_\tau$ vanishes at infinity since $f$ does and since these functions are equimeasurable) and the limit point $F$ is unique. This concludes the proof of the first part of the proposition.

Let us now show the right-continuity at $\tau=0$. It follows from Proposition~\ref{symmlpcont} that $\mathsf f_\tau \to f$ in $\mathrm L^{2^*}(\R^d)$ as $\tau\to 0^+$. This implies that $\nabla \mathsf f_\tau \rightharpoonup \nabla f$ in $\mathrm L^2(\R^d)$ as $\tau\to 0^+$. (Indeed, the argument is similar to the one used in the first part of the proof. The family $\nabla \mathsf f_\tau$ is bounded in $\mathrm L^2(\R^d)$ as $\tau\to 0^+$ and, if $F$ denotes any weak limit point in $\mathrm L^2(\R^d)$, then the convergence in $\mathrm L^{2^*}(\R^d)$ and the definition of weak derivatives implies that $F=\nabla f$.) By weak lower semicontinuity, we deduce that
$$
\|\nabla f \|_2 \leq \liminf_{\tau\to 0^+} \|\nabla \mathsf f_\tau \|_2\,.
$$
This, together with the reverse inequality, which was established in the first part of the proof, proves the claimed right continuity.
\end{proof}

We note that the proposition remains valid for $0\leq f\in \dot{\mathrm W}^{1,p}(\R^d)$ with $1\leq p<d$. If $p\neq 2$, the monotonicity for the gradient for $f\in \mathrm W^{1,p}(\R^d)$ is proved in~\cite[Theorem~3.2]{Brock2}. The remaining arguments above carry over to $p\neq 2$.

%%%%%%%%%%%%%%%%%%%%%%%%%%%%%%%%%%%%%%%%%%%%%%%%%%
%%%%%%%%%%%%%%%%%%%%%%%%%%%%%%%%%%%%%%%%%%%%%%%%%%
\section*{Acknowledgements}

% Partial support through US National Science Foundation grants DMS-1954995 (R.L.F.) and DMS-2154340 (M.L.), as well as through the Deutsche Forschungsgemeinschaft (DFG, German Research Foundation) Germany's Excellence Strategy EXC-2111-390814868 (R.L.F.), the French National Research Agency (ANR) projects EFI (ANR-17-CE40-0030, J.D.) and molQED (ANR-17-CE29-0004, M.J.E.) and the European Research Council under the Grant Agreement No. 721675 (RSPDE) {\it Regularity and Stability in Partial Differential Equations} (A.F.).

The authors thank two referees for their careful reading and useful suggestions. \copyright~2024 by the authors. Reproduction of this article by any means permitted for noncommercial purposes. \hbox{\href{https://creativecommons.org/licenses/by/4.0/legalcode}{CC-BY 4.0}}

%%%%%%%%%%%%%%%%%%%%%%%%%%%%%%%%%%%%%%%%%%%%%%%%%%
%%%%%%%%%%%%%%%%%%%%%%%%%%%%%%%%%%%%%%%%%%%%%%%%%%
\raggedbottom
%\bibliography{biblio}{}
%\bibliographystyle{imsart-number}

%%%%%%%%%%%%%%%%%%%%%%%%%%%%%%%%%%%%%%%%%%%%%%%%%%
%%%%%%%%%%%%%%%%%%%%%%%%%%%%%%%%%%%%%%%%%%%%%%%%%%
\end{document}